\documentclass[10pt,reqno]{amsart}

\usepackage{latexsym,amsmath}
\usepackage{amsfonts}
\usepackage{amssymb}
\usepackage{fixmath}

\usepackage{fullpage}
\usepackage{graphicx}
\usepackage[colorlinks=true,citecolor=blue]{hyperref}
\usepackage[dvipsnames]{xcolor}

\usepackage{enumerate}

\usepackage[T1]{fontenc}
\usepackage{fourier}
\usepackage{bbm}

\usepackage[boxed, linesnumbered, noend, noline]{algorithm2e}
\SetKwInput{KwData}{Input}
\SetKwInput{KwResult}{Output}

\numberwithin{equation}{section}

\renewcommand{\tanh}{\operatorname{th}}
\renewcommand{\sinh}{\operatorname{sh}}
\renewcommand{\cosh}{\operatorname{ch}}
\newcommand{\atanh}{\operatorname{ath}}
\newcommand{\Zbal}{Z_{\GG,\mathrm{bal}}}

\newcommand\betaf{\beta_{\mathrm f}}
\renewcommand\root{\varrho}

\newcommand\PM{\{\pm1\}}
\newcommand\ism{\cong}

\newcommand\disteq{\sim}
\newcommand\dist{\mathrm{dist}}

\newcommand{\BP}{\mathrm{BP}}

\newcommand{\fm}{\mathfrak m}

\newcommand{\fB}{\mathfrak B}

\newcommand{\fT}{\mathfrak T}

\newcommand{\vDelta}{\vec\Delta}

\newcommand{\vC}{\vec C}

\newcommand{\vd}{\vec d}

\newcommand{\vN}{\vec N}

\newcommand{\vw}{\vec w}
\newcommand{\vc}{\vec c}

\newcommand{\vX}{\vec X}

\renewcommand{\epsilon}{\eps}
\renewcommand{\subset}{\subseteq}

\newcommand\vI{\vec I}

\newcommand\vQ{\vec Q}

\newcommand\vY{\vec Y}

\newcommand\vm{\vec m}

\newcommand\vU{\vec U}

\newcommand\nix{\,\cdot\,}

\newcommand\dd{{\mathrm d}}

\renewcommand{\vec}[1]{\boldsymbol{#1}}

\newcommand\SIGMA{\vec\sigma}

\newcommand\cutm{\Delta_{\Box}}

\newtheorem{definition}{Definition}[section]
\newtheorem{claim}[definition]{Claim}

\newtheorem{remark}[definition]{Remark}
\newtheorem{theorem}[definition]{Theorem}
\newtheorem{lemma}[definition]{Lemma}
\newtheorem{proposition}[definition]{Proposition}
\newtheorem{corollary}[definition]{Corollary}

\newtheorem{fact}[definition]{Fact}

\newcommand\fG{\mathfrak{G}}
\newcommand\fS{\mathfrak{S}}

\newcommand\fD{\mathfrak{D}}

\newcommand\fW{\mathfrak{W}}
\newcommand\fV{\mathfrak{V}}

\newcommand\fA{\mathfrak{A}}
\newcommand\fE{\mathfrak{E}}

\newcommand\fR{\mathfrak{R}}
\newcommand\fF{\mathfrak{F}}
\newcommand\fP{\mathfrak{P}}

\newcommand\cB{\mathcal{B}}
\newcommand\cC{\mathcal{C}}
\newcommand\cD{\mathcal{D}}
\newcommand\cF{\mathcal{F}}

\newcommand\cE{\mathcal{E}}

\newcommand\cN{\mathcal{N}}

\newcommand\cH{\mathcal{H}}
\newcommand\cS{\mathcal{S}}
\newcommand\cT{\mathcal{T}}

\newcommand\cM{\mathcal{M}}

\newcommand\cP{\mathcal{P}}
\newcommand\cX{\mathcal{X}}
\newcommand\cY{\mathcal{Y}}
\newcommand\cV{\mathcal{V}}
\newcommand\cW{\mathcal{W}}
\newcommand\cZ{\mathcal{Z}}

\def\cE{{\mathcal E}}

\newcommand\fX{\mathfrak{X}}
\newcommand\ve{\vec e}

\newcommand\vv{\vec v}

\newcommand\vZ{\vec Z}

\newcommand\eul{\mathrm{e}}
\newcommand\eps{\varepsilon}

\newcommand\ZZ{\mathbb{Z}}

\newcommand\TT{\mathbb{T}}
\newcommand\GG{\mathbb{G}}
\newcommand\ZZpos{\mathbb{Z}_{\geq0}}
\newcommand\Var{\mathrm{Var}}
\newcommand\Erw{\mathbb{E}}
\newcommand{\vecone}{\mathbb{1}}

\newcommand{\Po}{{\rm Po}}

\newcommand\dTV{d_{\mathrm{TV}}}

\newcommand\bc[1]{\left({#1}\right)}
\newcommand\cbc[1]{\left\{{#1}\right\}}
\newcommand\bcfr[2]{\bc{\frac{#1}{#2}}}
\newcommand{\bck}[1]{\left\langle{#1}\right\rangle}
\newcommand\brk[1]{\left\lbrack{#1}\right\rbrack}
\newcommand\Scal[2]{\langle{{#1},{#2}\rangle}}
\newcommand\scal[2]{\bck{{#1},{#2}}}
\newcommand\norm[1]{\left\|{#1}\right\|}
\newcommand\abs[1]{\left|{#1}\right|}

\newcommand\RR{\mathbb{R}}
\newcommand\RRpos{\RR_{\geq0}}
\def\?#1{}
\makeatletter
\def\whp{w.h.p\@ifnextchar-{.}{\@ifnextchar.{.\?}{\@ifnextchar,{.}{\@ifnextchar){.}{\@ifnextchar:{.:\?}{.\ }}}}}}
\makeatother

\makeatletter
\def\Whp{W.h.p\@ifnextchar-{.}{\@ifnextchar.{.\?}{\@ifnextchar,{.}{\@ifnextchar){.}{\@ifnextchar:{.:\?}{.\ }}}}}}
\makeatother

\newcommand{\tensor}{\otimes}

\newcommand{\Erdos}{Erd\H{o}s}
\newcommand{\Renyi}{R\'enyi}
\newcommand{\Lovasz}{Lov\'asz}

\newcommand\pr{\mathbb{P}}

\newcommand\Lem{Lemma}
\newcommand\Prop{Proposition}

\newcommand\Thm{Theorem}

\newcommand\Cor{Corollary}
\newcommand\Sec{Section}
\newcommand\Chap{Chapter}

\DeclareMathOperator{\tr}{tr}

\DeclareMathOperator{\br}{br}

\newcommand{\supp}[1]{{\text{supp}\left(#1\right)}}

\usepackage{upgreek}

\newcommand{\enrg}{\cH}

\usepackage{wasysym}
\usepackage{dsfont}
\usepackage{paralist}

\definecolor{aoEng}{rgb}{0, 0.5,0}

\usepackage{pifont}

\newcounter{kcomcount}
\def\ex{{\mathbb E}}
\def\pr{{\mathbb P}}

\def\cH{{\mathcal H}}

\title{Fluctuations of the Ising free energy on Erd\H os-R\'enyi graphs}

\thanks{
Amin Coja-Oghlan is supported by DFG CO 646/3, DFG CO 646/5 and DFG CO 646/6.
Maurice Rolvien is supported by DFG Research Group ADYN (FOR 2975) under DFG grant no.\ 411362735.
Pavel Zakharov is supported by DFG CO 646/6.
Kostas Zampetakis is supported by DFG CO 646/5.
}

\author{Amin Coja-Oghlan, 
Dominik Kaaser, Maurice Rolvien, Pavel~Zakharov, Kostas~Zampetakis}

\address{Amin Coja-Oghlan, {\tt amin.coja-oghlan@tu-dortmund.de}, TU Dortmund, Faculty of Computer Science and Faculty of Mathematics, 12 Otto-Hahn-St, Dortmund 44227, Germany.}
\address{Dominik Kaaser, {\tt dominik.kaaser@tuhh.de}, TU Hamburg,  Institute for Data Engineering, Blohmstra\ss e 15, 21079 Hamburg, Germany.}
\address{Maurice Rolvien, {\tt maurice.rolvien@uni-hamburg.de}, University of Hamburg, Faculty of Mathematics, Informatics and Natural Sciences, Department of Informatics, Vogt-K\"olln-Str.\ 30, 22527 Hamburg, Germany.}
\address{Pavel Zakharov, {\tt pavel.zakharov@tu-dortmund.de}, TU Dortmund, Faculty of Computer Science and Faculty of Mathematics, 12 Otto-Hahn-St, Dortmund 44227, Germany.}
\address{Kostas Zampetakis, {\tt konstantinos.zampetakis@tu-dortmund.de}, TU Dortmund, Faculty of Computer Science, 12 Otto-Hahn-St, Dortmund 44227, Germany.}

\begin{document}

\begin{abstract}
	We investigate the ferromagnetic Ising model on the \Erdos-\Renyi\ random graph $\GG(n,m)$ with bounded average degree $d=2m/n$.
	Specifically, we determine the limiting distribution of $\log Z_{\GG(n,m)}(\beta,B)$, where
	$Z_{\GG(n,m)}(\beta,B)$ is the partition function at inverse temperature $\beta>0$ and external field $B\geq0$.
	If either $B>0$, or $B=0$, $d>1$ and $\beta>\atanh(1/d)$ the limiting distribution
	is a Gaussian whose variance is of order $\Theta(n)$ and is described by a family of stochastic fixed point problems that encode the root magnetisation of two correlated Galton-Watson trees.
	By contrast, if $B=0$ and either $d\leq1$ or $\beta<\atanh(1/d)$ the limiting distribution is an infinite sum of independent random variables and has bounded variance.
	\hfill[MSc:~05C80,~82B44,~82B20]
\end{abstract}

\maketitle

\section{Introduction}\label{sec_intro}

\subsection{Background and motivation}\label{sec_mot}
The Ising model has captivated mathematical physicists for a century~\cite{Ising,Lenz_1920}.
This is partly because the model furnishes perhaps the most elegant example of a non-trivial phase transition~\cite{Onsager}.
Another reason may be that, even though the definition fits in a single line, the model gives rise to profound mathematical challenges, many of which remain unresolved to this day~\cite{DC}.
Additionally, the study of the Ising model has had a remarkable impact across several disciplines.
For instance, the Ising model has been seized upon to explain assortative phenomena in statistics and computer science~\cite{Abbe_2017,Baldassarri}.

The original purpose of the Ising model was to demonstrate how ferromagnetism can emerge from the interactions of microscopic particles associated with the vertices of a graph $G=(V(G),E(G))$.
Each of these particles takes one of two spins $\pm1$.
The edges of the graph prescribe the geometry of interactions.
The model comes with two scalar parameters called the inverse temperature $\beta>0$ and the external field $B\geq0$.
Depending on the graph, phase transitions may or may not occur as $\beta,B$ vary.
Indeed, while Ising proved the absence of phase transitions if $G$ is the one-dimensional integer `lattice'~\cite{Ising}, a deep line of subsequent work revealed phase transitions on higher dimensional lattices; see \Sec~\ref{sec_discussion}.

The aim of the present contribution is to achieve a more complete understanding of the Ising model on the \Erdos-\Renyi\ random graph $\GG=\GG(n,m)$, a uniformly random graph on $n$ vertices  with $m$ edges.
Specifically, we investigate the regime where the average degree $d=2m/n$, i.e., the average number of interactions per vertex, is bounded, just like in a finite-dimensional lattice.
Since the geometry of interactions and thus the Boltzmann distribution are random, the Ising model on $\GG(n,m)$ is a specimen of a disordered system, reminiscent of but conceptually simpler than a spin glass~\cite{MM}.

In an important contribution Dembo and Montanari~\cite{Dembo_2010} proved that to the leading order the free energy of the Ising model on $\GG(n,m)$ follows the `replica symmetric solution' predicted by the cavity method, a heuristic from statistical physics~\cite{MP1,MP2}.
The replica symmetric solution comes in terms of a fixed point problem on the space $\cP([-1,1])$ of probability measures on the interval $[-1,1]$.
To elaborate, let
\begin{align*}
	Z_\GG(\beta,B)&=\sum_{\sigma\in\PM^{V(\GG)}}\exp\bc{\beta\sum_{vw\in E(\GG)}\sigma_v\sigma_w+B\sum_{v\in V(\GG)}\sigma_v}
\end{align*}
be the Ising partition function.
Of course,  due to its dependence on the random graph $Z_{\GG}(\beta,B)$ is a random variable.
Further, consider the {\em Belief Propagation operator}
\begin{align}\label{eqBPop}
	\BP_{d,\beta,B}&:\cP([-1,1])\to\cP([-1,1]),&\pi\mapsto\hat\pi&=\BP_{d,\beta,B}(\pi),
\end{align}
defined as follows.
Let $\vd\disteq\Po(d)$ be a Poisson variable and let $(\vec\mu_{\pi,i})_{i\geq0}$ be a sequence of random variables with distribution $\pi$, all mutually independent and independent of $\vd$.
Then $\hat\pi=\BP_{d,\beta,B}(\pi)$ is the distribution of the random variable
\begin{align}\label{eqBPrec}
		\frac{\sum_{s\in\PM}s\eul^{sB}\prod_{i=1}^{\vec d}\bc{1+s\vec\mu_{\pi,i}\tanh\beta}}{\sum_{s\in\PM}\eul^{sB}\prod_{i=1}^{\vec d}\bc{1+s\vec\mu_{\pi,i}\tanh\beta}}&
			\in[-1,1].
	\end{align}
Write $\BP_{d,\beta,B}^\ell$ for the $\ell$-fold application of $\BP_{d,\beta,B}$ and $\delta_1$ for the atom on $1$.
Then for any $\beta,B\geq0$ the weak limit
\begin{align}\label{eqDMweakLimit}
	\pi_{d,\beta,B}&=\lim_{\ell\to\infty}\BP^\ell_{d,\beta,B}(\delta_1)
\end{align}
exists.
Finally, define the {\em Bethe free energy} functional $\cB_{d,\beta,\pi}:\cP([-1,1])\to\RR$ by letting
\begin{align}
	\cB_{d,\beta,B}(\pi)&=\frac d2\brk{\log\cosh\beta-\ex\log\bc{1+\vec\mu_{\pi,1}\vec\mu_{\pi,2}\tanh\beta}}+\ex\log\brk{\eul^B\prod_{i=1}^{\vd}\bc{1+\vec\mu_{\pi,i}\tanh\beta}+ \eul^{-B}\prod_{i=1}^{\vd}\bc{1-\vec\mu_{\pi,i}\tanh\beta}}.\label{eqBFE}
\end{align}
Dembo and Montanari \cite[\Thm~2.4]{Dembo_2010} show that for all $d,\beta,B\geq0$,
\begin{align}\label{eqDM}
	\lim_{n\to\infty}\frac1n\log Z_{\GG}(\beta,B)&=\cB_{d,\beta,B}(\pi_{d,\beta,B})
\end{align}
in probability.%
	\footnote{%
		Actually Dembo and Montanari prove a more general result.
		Namely, they show that for any sequence of graphs that converges to a unimodular tree in the topology of local weak convergence the free energy converges to the leading order to the replica symmetric solution associated with the limiting unimodular random tree.
	We will come back to this in \Sec~\ref{sec_discussion}.
}

Due to the normalising factor $1/n$ on the l.h.s., the result~\eqref{eqDM} determines the free energy only to a first-order, law-of-large-numbers degree of precision.
The contribution of the present work is to establish the precise limiting distribution of $\log Z_{\GG}(\beta,B)$.
As we will discover, for most values of the parameters $d,\beta,B$ the random variable $\log Z_{\GG}(\beta,B)$ has variance $\Theta(n)$ with Gaussian fluctuations.
We will derive a precise but intricate expression for the variance.
Only in the high temperature phase with no external field (i.e., $B=0$ and $\beta>0$ small) do we obtain a non-Gaussian limiting distribution with simple explicit parameters.
In this regime $\log Z_{\GG}(\beta,0)$ has bounded fluctuations.

\subsection{Results}\label{sec_res}
We proceed to state the main results precisely, beginning with the somewhat simpler case $B>0$ where an external field is present.

\subsubsection{The case $B>0$}\label{sec_res_ext}
We prove that the (suitably shifted and scaled) free energy converges to a Gaussian whose variance derives from the fixed points of a family of operators on the space $\cP([-1,1]^2)$ of probability measures on the square $[-1,1]^2$.
To elaborate, for a parameter $t\in[0,1]$ define the {\em coupled Belief Propagation operator}
\begin{align}\label{eqBPtensor}
	\BP_{d,\beta,B,t}^\tensor&:\cP([-1,1]^2)\to\cP([-1,1]^2),&\pi\mapsto\hat\pi=\BP_{d,\beta,B,t}^\tensor(\pi)&&(t\in[0,1])
\end{align}
as follows.
Let $(\vec\mu_{\pi,i}^{(0)})_{i\geq0}$, $(\vec\mu_{\pi,i}^{(1)})_{i\geq0}$, $(\vec\mu_{\pi,i}^{(2)})_{i\geq0}$ with
\begin{align}\label{eqIndependentSamples}
	\vec\mu_{\pi,i}^{(0)}&=(\vec\mu_{\pi,i,1}^{(0)},\vec\mu_{\pi,i,2}^{(0)}),\enspace
	\vec\mu_{\pi,i}^{(1)}=(\vec\mu_{\pi,i,1}^{(1)},\vec\mu_{\pi,i,2}^{(1)}),\enspace
	\vec\mu_{\pi,i}^{(2)}=(\vec\mu_{\pi,i,1}^{(2)},\vec\mu_{\pi,i,2}^{(2)})\in[-1,1]^2
\end{align}
be sequences of random variables (viz., vectors) with distribution $\pi$.
In addition, let
\begin{align}\label{eqDeltat}
	\vd_t^{(0)}&\disteq\Po(dt),&
	\vd_t^{(1)}&\disteq\Po(d(1-t)),&
	\vd_t^{(2)}&\disteq\Po(d(1-t))
\end{align}
be Poisson variables.
All these random variables are mutually independent.
Then $\hat\pi=\BP_{d,\beta,B,t}^\tensor(\pi)$ is defined to be the distribution of the random vector
\begin{align}\label{eqBPtensorRec}
	\bc{\frac{\sum_{s\in\PM}s\exp(sB)\prod_{i=1}^{\vd_t^{(0)}}\bc{1+s\vec\mu_{\pi,i,h}^{(0)}\tanh\beta}\prod_{i=1}^{\vd_t^{(h)}}\bc{1+s\vec\mu_{\pi,i,h}^{(h)}\tanh\beta}}
	{\sum_{s\in\PM}\exp(sB)\prod_{i=1}^{\vd_t^{(0)}}\bc{1+s\vec\mu_{\pi,i,h}^{(0)}\tanh\beta}\prod_{i=1}^{\vd_t^{(h)}}\bc{1+s\vec\mu_{\pi,i,h}^{(h)}\tanh\beta}}}_{h=1,2}\in(-1,1)^2\enspace.
\end{align}
Write $\BP^{\tensor\,\ell}_{d,\beta,B,t}$ for the $\ell$-fold application of the operator $\BP^{\tensor}_{d,\beta,B,t}$ and let $\delta_{(1,1)}\in\cP([-1,1]^2)$ be the atom on $(1,1)$.
Furthermore, define a functional
\begin{align}\label{eqBFEtensor}
	\cB^\tensor_{\beta}(\pi)&:\cP([-1,1]^2)\to\RRpos,& \pi&\mapsto\ex\brk{\prod_{h=1}^2\log\bc{1+\vec\mu_{\pi,1,h}^{(0)}\vec\mu_{\pi,2,h}^{(0)}\tanh\beta}}.
\end{align}
Finally, let us write $\cN(M,\Sigma^2)$ for the normal distribution with mean $M$ and variance $\Sigma^2$.
Then in the case $B>0$ we obtain the following central limit theorem.

\begin{theorem}\label{thm_ex}
	For any $d,\beta,B>0$ and $t\in[0,1]$ the weak limit
	\begin{align}\label{eqfix_ex}
		\pi_{d,\beta,B,t}^\tensor&=\lim_{\ell\to\infty}\BP^{\tensor\,\ell}_{d,\beta,B,t}(\delta_{(1,1)})
	\end{align}
	exists and, in distribution,
	\begin{align}\nonumber
		\lim_{n\to\infty}\frac{\log Z_{\GG}(\beta,B)-\ex\log Z_{\GG}(\beta,B)}{\sqrt m}&=\cN(0,\Sigma(d,\beta,B)^2),&&\mbox{where}\\
	\label{eqSigmadbetaB}
	\Sigma(d,\beta,B)^2&=\int_0^1\cB^\tensor_{\beta}(\pi^\tensor_{d,\beta,B,t})\dd t-\cB^\tensor_{\beta}(\pi^\tensor_{d,\beta,B,0})>0.
	\end{align}
\end{theorem}

Naturally, the fact that $\Sigma(d,\beta,B)^2$ is strictly positive is part of the statement of \Thm~\ref{thm_ex}.
The theorem implies that $\log Z_{\GG}(\beta,B)$ exhibits fluctuations of order $\Theta(\sqrt m)$ around its expectation.
Moreover, as we will discover in \Sec~\ref{sec_overview} the fixed point distributions $\pi_{d,\beta,B,t}^\tensor$ admit a neat combinatorial interpretation, as does the functional~$\cB^\tensor_{\beta}$.

\subsubsection{The case $B=0$}\label{sec_res_no}
In the absence of an external field (i.e., $B=0$) the Ising model possesses an inherent symmetry under inversion.
This symmetry causes significant complications, at least if $d>1$, i.e., if the random graph $\GG(n,m)$ possesses a giant component with high probability.
More precisely, together with~\eqref{eqDM} Lyons' theorem~\cite[\Thm~1.1]{Lyons} implies that in the case $B=0$ and $d>1$ there occurs a phase transition at
\begin{align}\label{eqbetaf}
	\betaf(d)=\atanh(1/d).
\end{align}
Indeed, for $d\leq1$ or $\beta<\betaf(d)$ the distribution $\pi_{d,\beta,0}$ from \eqref{eqDMweakLimit} turns out to be just the atom at zero.
In effect, the functional~\eqref{eqBFE} simply evaluates to
\begin{align}\label{eqBFEtrivial}
	\cB_{d,\beta,0}(\delta_0)&=\log2+\frac d2\log\cosh\beta.
\end{align}
The following theorem pinpoints the limiting distribution of $\log Z_{\GG}(\beta,0)$ in this `high temperature' regime.

\begin{figure}
	\centering
	\includegraphics[scale=0.66]{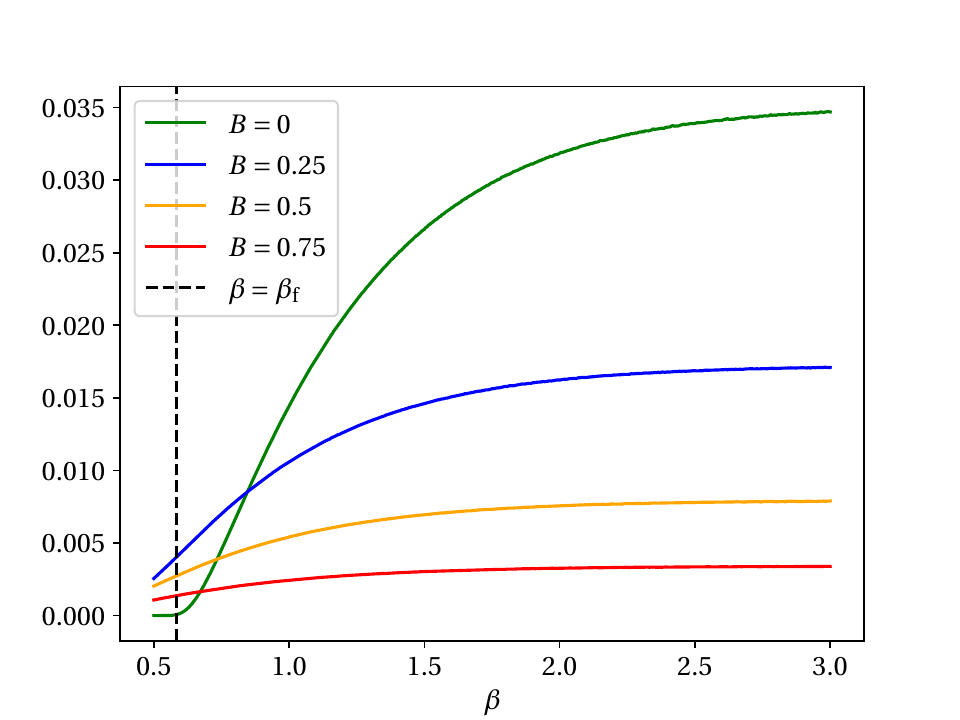}
	\caption{Plots of $\beta\mapsto\Sigma(d,\beta)^2$ at $d =1.9$ and at external fields $B=0$, $B=0.25$, $B=0.5$ and $B=0.75$.
		The dashed vertical line marks the critical inverse temperature $\betaf(1.9)$. }\label{fig_var}
\end{figure}

\begin{theorem}\label{thm_an}
	Assume that either $d\leq1$, or $d>1$ and $\beta<\betaf(d)$.
	Let $(\vY_\ell)_{\ell\geq3}$ be a sequence of independent random variables with distribution $\vY_\ell\disteq\Po(d^\ell/(2\ell))$.
	Then, in distribution,
	\begin{align}\label{eqthm_an}
    \lim_{n\to\infty}\log Z_{\GG}(\beta,0)-\left[n\log2+m\log\cosh\beta-\frac12\log\bc{1 - d \tanh \beta}\right]&=\sum_{\ell\geq3}\bc{\vY_\ell\log\bc{1+\tanh^\ell\beta}-\frac{d^\ell}{2\ell}\tanh^\ell\beta}-\frac d2\tanh\beta-\frac{d^2}4\tanh^2\beta.
	\end{align}
\end{theorem}

\noindent
Since the sum $\sum_{\ell\geq3}\bc{\vY_\ell\log\bc{1+\tanh^\ell\beta}-\frac{d^\ell}{2\ell}\tanh^\ell\beta}$ converges in $L^2$, \Thm~\ref{thm_an} implies that the limiting distribution on the r.h.s.\ of~\eqref{eqthm_an} has variance
\begin{align*}
	\sum_{\ell\geq3}\frac{d^\ell}{2\ell}\log^2(1+\tanh^\ell\beta)<\infty.
\end{align*}
Thus, the random variable $\log Z_{\GG}(\beta,0)$ superconcentrates, i.e., has {\em bounded} fluctuations.

For $d>1$ and $\beta>\betaf(d)$, i.e., in the `low temperature regime', the Ising model behaves fundamentally differently~\cite{Dembo_2010,Lyons}.
In particular, the limiting distribution \eqref{eqDMweakLimit} is no longer just a point mass and
\begin{align*}
	\cB_{d,\beta,0}(\pi_{d,\beta,0})>\log2+\frac d2\log\cosh\beta.
\end{align*}
Furthermore, the next theorem implies that for almost all $\beta>\betaf(d)$ the free energy $\log Z_{\GG}(\beta,0)$ has a Gaussian limit with fluctuations of order $\Theta(\sqrt m)$, as in \Thm~\ref{thm_ex}.

To state the theorem, we introduce
\begin{align}\label{eqrdbeta}
	r_{d,\beta}&=\frac 12\ex\brk{\frac{\vec\mu_{\pi_{d,\beta,0},1}\vec\mu_{\pi_{d,\beta,0},2}+\tanh\beta}{1+\vec\mu_{\pi_{d,\beta,0},1}\vec\mu_{\pi_{d,\beta,0},2}\tanh\beta}}&&
	(\mbox{cf.\ \cite[Eq.~(1.8)]{Basak} and \cite[Eq.~(6.11)]{Dembo_2010}}).
\end{align}
The function $\beta\mapsto r_{d,\beta}$ is bounded, monotonically increasing and right-continuous for any $d>0$~\cite[\Thm~1.6]{Basak}.
Hence, the set $\fX_d$ of $\beta\in(0,\infty)$ where $\beta\mapsto r_{d,\beta}$ fails to be continuous is (at most) countable.

\begin{figure}
	\centering
	\includegraphics[scale=0.66]{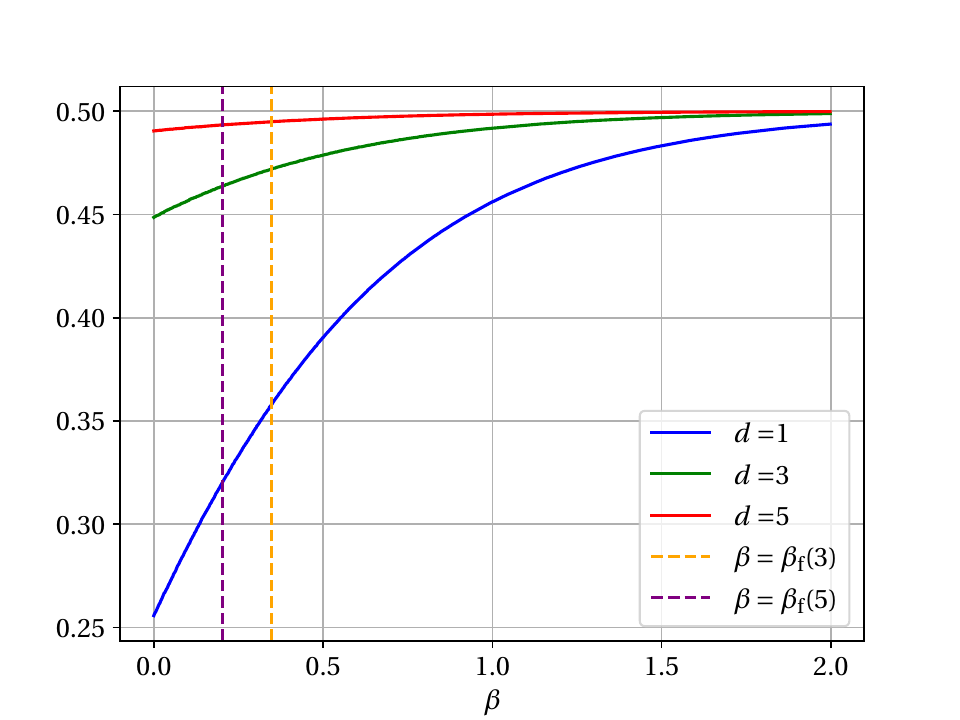}
	\caption{Plot of $\beta\mapsto {r_{d,\beta}}$ for $d=1$, $d=3$ and $d=5$.}\label{fig_rdbeta}
\end{figure}

\begin{theorem}\label{thm_no}
	Let $d>1$ and $\beta\in(\betaf(d),\infty)\setminus\fX_d$.
	For all $t\in[0,1]$ the weak limit
	\begin{align}\label{eqfix_no}
		\pi_{d,\beta,0,t}^\tensor&=\lim_{\ell\to\infty}\BP^{\tensor\,\ell}_{d,\beta,0,t}(\delta_1)
	\end{align}
	exists and in distribution,
	\begin{align*}
		\lim_{n\to\infty}\frac{\log Z_{\GG}(\beta,0)-\ex\log Z_{\GG}(\beta,0)}{\sqrt m}&=\cN(0,\Sigma(d,\beta)^2),&\mbox{where}&&
	\Sigma(d,\beta)^2&=\int_0^1\cB^\tensor_{\beta}(\pi^\tensor_{d,\beta,0,t})\dd t-\cB^\tensor_{\beta}(\pi^\tensor_{d,\beta,0,0})>0.
	\end{align*}
\end{theorem}

Despite their daunting appearance the variance formulas from \Thm s~\ref{thm_ex} and~\ref{thm_no} can actually be evaluated numerically with reasonable efficiency.
This is because the fixed point iterations~\eqref{eqfix_ex} and~\eqref{eqfix_no} turn out to converge geometrically.
Figure~\ref{fig_var} displays the variances $\Sigma(d,\beta)^2$ for some instructive parameter values.

\Thm~\ref{thm_no} is syntactically similar to \Thm~\ref{thm_ex} but for the appearance of the countable `exceptional' set $\fX_d$ of discontinuities of the function $\beta\mapsto r_{d,\beta}$ from~\eqref{eqrdbeta} where we do not know that convergence to a Gaussian limit occurs.
Actually for all we know the exceptional set may be empty, i.e., the function $\beta\mapsto r_{d,\beta}$ may be continuous for all $d>1$; Figure~\ref{fig_rdbeta} displays $\beta\mapsto r_{d,\beta}$ for a few values of $d$.
But of course a proof (or refutation) of this hypothesis remains an open problem.

\begin{figure}
	\centering
	\includegraphics[scale=1]{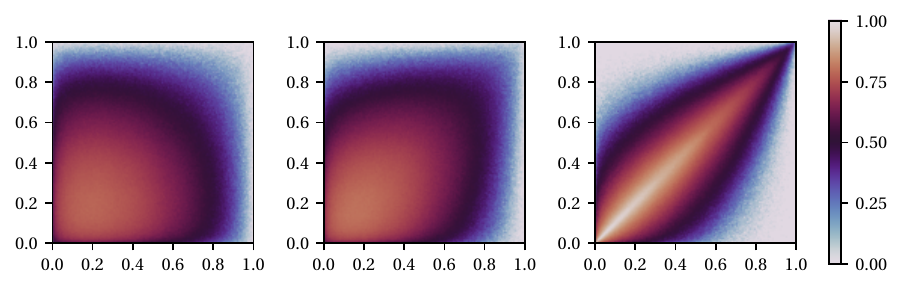}
	\caption{The distributions $\pi_{d,\beta,0,t}^\tensor$ for $d=1.9$, $\beta=0.65$, and $t=0.1, 0.5, 0.9$.
		The atom at $(0,0)$ has been removed to improve visibility.}\label{fig_heatmap}
\end{figure}

Additionally, the heat map representation in Figure~\ref{fig_heatmap} shows the distributions $\pi_{d,\beta,0,t}$ for three increasing values of $t$.
The figure shows how the probability mass gravitates towards the diagonal as $t$ increases.
Roughly speaking, as $t$ increases the two coordinates of the random pair
\begin{align}\label{eqtwocoords}
	\vec\mu^{(0)}_{\pi_{d,\beta,B,t},1}=(\vec\mu^{(0)}_{\pi_{d,\beta,B,t},1,1},\vec\mu^{(0)}_{\pi_{d,\beta,B,t},1,2})\end{align}
become more `aligned'.
For instance, for $t=0$ the Poisson variables from \eqref{eqDeltat} satisfy $\vd_t^{(0)}=0$ while $\vd_t^{(1)},\vd_t^{(2)}$ are independent $\Po(d)$ variables.
Consequently, comparing the definition~\eqref{eqBPtensorRec} of the operator $\BP_{d,\beta,B,0}^\tensor$ with the definition~\eqref{eqBPrec} of $\BP_{d,\beta,B}$, we see that the coordinates of the pair \eqref{eqtwocoords} are just two independent random variables with distribution $\pi_{d,\beta,B}$.
On the other extreme, for $t=1$ we have $\vd_t^{(1)}=\vd_t^{(2)}=0$ while $\vd_t^{(0)}$ is a $\Po(d)$ variable.
Hence, the operator \eqref{eqBPtensorRec} acts identically on both coordinates, which means that the two coordinates of~\eqref{eqtwocoords} are identical and have distribution $\pi_{d,\beta,B}$.
Naturally, for values $0<t<1$ the distribution $\pi_{d,\beta,B,t}^\tensor$ interpolates between these two extremes.
Let us make a quick note of these observations for later reference.

\begin{remark}\label{rem_extreme}
	The distribution $\pi_{d,\beta,B,0}^\tensor$ is the product measure $\pi_{d,\beta,B}\tensor\pi_{d,\beta,B}$, i.e., the two coordinates of~\eqref{eqtwocoords} are independent.
	By contrast, $\pi_{d,\beta,B,0}^\tensor$ is the distribution of an identical pair $(\vec\mu_{\pi_{d,\beta,B},1},\vec\mu_{\pi_{d,\beta,B},1})$, i.e., the two coordinates of~\eqref{eqtwocoords} are identical.
\end{remark}

\subsection{Notation}\label{sec_notation}
We denote the vertex set of the random graph $\GG(n,m)$ by $V_n=\{v_1,\ldots,v_n\}$.
Throughout the paper asymptotic notation such as $o(\nix),O(\nix)$ refers to the limit $n\to\infty$.
The random graph $\GG(n,m)$ enjoys a property $\fP$ {\em with high probability} (`\whp') if the probability that $\fP$ occurs converges to one as $n\to\infty$.
We continue to assume tacitly that $m=m(n)$ is a sequence such that $\lim_{n\to\infty}2m/n=d$ for a fixed real $d>0$.
All statements concerning the random graph are understood to hold for {\em any} sequence $m=m(n)$ with $\lim_{n\to\infty}2m/n=d$.

For a graph $G=(V,E)$ and a vertex $v\in V$ we denote by $\partial(G,v)$ the set of all neighbours of $v$.
By extension, for an integer $\ell\geq0$ we let $\partial^\ell(G,v)$ be the set of all vertices at (shortest path) distance precisely $\ell$ from $v$.
A {\em rooted graph} is a graph $G=(V,E)$ together with a distinguished root vertex $r\in V$.
For a vertex $v$ of a graph $G$ we let $\nabla^\ell(\GG,v)$ be the subgraph of $G$ induced by all vertices $u$ at distance at most $\ell$ from $v$, rooted at $v$.
Two rooted graphs $(G,r)$ and $(G',r')$ are {\em isomorphic} if there exists a graph-theoretic isomorphism $G\to G'$ that maps $r$ to $r'$.

Let $G=(V,E)$ be a graph.
For a vector $\sigma\in\PM^{V(G)}$ and a subset $U\subset V$ we write
\begin{align*}
	\sigma(U)&=\prod_{u\in U}\sigma(u).
\end{align*}
In particular, for an edge $e=\{v,w\}\in E$ we have $\sigma(e)=\sigma(v)\sigma(w)$.
The {\em Ising model on $G$ at inverse temperature $\beta$ and external field $B$} is the probability distribution on $\PM^{V(G)}$ defined by
\begin{align}
	\mu_{G,\beta,B}(\sigma)&=\frac1{Z_{G}(\beta,B)}\exp\bc{\beta\sum_{e\in E(G)}\sigma(e)+B\sum_{v\in V(G)}\sigma(v)}&&(\sigma\in\PM^{V(G)}),\qquad\mbox{ where}\label{eqIsingmu}\\
	Z_G(\beta,B)&=\sum_{\tau\in\PM^{V(G)}}\exp\bc{\beta\sum_{e\in E(G)}\tau(e)+B\sum_{v\in V(G)}\tau(v)}.\label{eqIsingZ}
\end{align}
Additionally, for a graph $G$ rooted at a vertex $r\in V(G)$ and an integer $\ell\geq0$ we will be led to consider the conditional distribution
\begin{align}\label{eqMuT}
\mu_{G,r,\beta,B}^{(\ell)}(\sigma)&=\mu_{G,\beta,B}\bc{\sigma\mid\{\forall v\in\partial^\ell(G,r):\SIGMA(v)=1\}}&&(\sigma\in\PM^{V(G)})
\end{align}
given that all vertices at distance precisely $\ell$ from the root $r$ carry spin $1$.

For a finite set $\Omega\neq\emptyset$ we let $\cP(\Omega)$ denote the set of all probability distributions on $\Omega$.
For $\mu\in\cP(\Omega)$ we let $\SIGMA_\mu\in\Omega$ denote a random element of $\Omega$ with distribution $\mu$; if and where the reference to $\mu$ is self-evident, we just write $\SIGMA$.
Moreover, for a function $X:\Omega\to\RR$ we write
\begin{align*}
	\scal{X(\SIGMA)}\mu&=\sum_{\sigma\in\Omega}X(\sigma)\mu(\sigma)
\end{align*}
for the expectation of $X(\SIGMA)$.
Further, if $\fE\subset\Omega$ is an event with $\mu(\fE)>0$, then $\mu(\nix\mid\fE)\in\cP(\Omega)$ signifies the conditional distribution given $\fE$.

For $\ell\geq1$ and probability measures $\mu,\nu$ on $[-1,1]^\ell$ we denote by $W_1(\mu,\nu)$ their $L^1$-Wasserstein distance \cite{RuschL}.
Thus, with $\Gamma(\mu,\nu)$ the set of all coupling of $\mu,\nu$, we have
\begin{align*}
	W_1(\mu,\nu)&=\inf_{\gamma\in\Gamma(\mu,\nu)}\int_{[-1,1]^\ell\times[-1,1]^\ell}\norm{x-y}_1\dd\gamma(x,y).
\end{align*}
We recall that the topology induced by $W_1(\nix,\nix)$ coincides with the topology of weak convergence \cite[\Thm~6.9]{Villani}.

{\em From here on we assume tacitly that $d,\beta>0$ and $B\geq0$.
	Unless specified otherwise all statements are understood to hold for any such $d,\beta,B$.}

\section{Overview}\label{sec_overview}

\noindent
The basic idea behind the proofs of the central limit theorems Theorems~\ref{thm_ex} and~\ref{thm_no} is to consider a family of correlated random graphs.
The degree of correlation is gauged by a coupling parameter $t\in[0,1]$, with $t=1$ meaning that the two random graphs are identical and $t=0$ corresponding to two independent random graphs.
By interpolating from $t=0$ to $t=1$ we will be able to get a handle on the variance of $\log Z_{\GG}(\beta,B)$.
To obtain the central limit theorems we will ultimately combine the analysis of the variance with an off-the-shelf martingale central limit theorem.

A broadly similar proof strategy was recently used to show that the partition function of the random 2-SAT model (i.e., the number of satisfying assignments) has a central limit theorem throughout the entire regime of clause/variable densities where satisfying assignments likely exist~\cite{2sat}.
The fundamental difference between that proof and the present contribution is that the random 2-SAT problem enjoys a very strong spatial mixing property called `Gibbs uniqueness' throughout the satisfiable regime~\cite{AchlioptasEtAl}.
The Gibbs uniqueness property constitutes the linchpin of the proofs from~\cite{2sat}.
By contrast, the Ising model on the random graph does not have the Gibbs uniqueness property for most interesting parameter regimes~\cite{meta,Dembo_2010}.
The most egregious counterexample is the case $B=0$, $d>1$ and $\beta>\betaf(d)$, i.e., low temperature and no external field.
As \Prop~\ref{prop_pure} below demonstrates, in this case two inverse `pure states' (in the sense of~\cite{Parisi}) coexist \whp, a scenario that amounts to pretty much the opposite of Gibbs uniqueness.
Additionally, the Belief Propagation operator has a spurious (and unstable) fixed point at $\delta_0$.

To deal with the challenge posed by the absence of Gibbs uniqueness, we will combine insights on the Ising model from prior work~\cite{Basak,Dembo_2010} with the positive correlation of spins under the Ising measure; see \Lem~\ref{lem_kostas} below.
In the case $B>0$ these ingredients pretty much suffice to investigate the variance and thus to establish the central limit theorem.
By contrast, in the $B=0$, $d>1$ and $\beta>\betaf(d)$ case substantial extra work is required.
In this parameter regime we need to get a precise handle on the two inverse pure states of the model, which is precisely what the aforementioned \Prop~\ref{prop_pure} delivers.
Actually the proposition plugs a significant gap in the understanding of the Ising model on the \Erdos-\Renyi\ graph and should thus be of independent interest.
Apart from insights from prior work on the Ising model, the proof of \Prop~\ref{prop_pure} combines techniques from work on diluted spin glass models from~\cite{Victor,Coja_2018,ACOPBethe,ACOPBP} with delicate correlation and monotonicity arguments.
With the pure state decomposition in place, an analysis involving careful conditioning then provides sufficient control of the variance to invoke the martingale central limit theorem and to prove \Thm~\ref{thm_no}.
Finally, the proof of \Thm~\ref{thm_an} rests on a totally different set of techniques.
Here we combine a truncated second moment argument with the small subgraph conditioning method that was originally developed towards the Hamilton cycle problem on random graphs~\cite{RW2}.

The purpose of this section is to flesh out the proof strategy in greater detail.
Along the way we will discover the combinatorics behind the operator~\eqref{eqBPtensor} and the functional~\eqref{eqBFEtensor}.
But first we need to recall the workings of the classical Belief Propagation operator~\eqref{eqBPop} that leads to the `replica symmetric solution'~\eqref{eqDM}, because Belief Propagation will play a vital role towards the proofs of \Thm s~\ref{thm_ex} and~\ref{thm_no} as well.
Subsequently we will survey the strategy behind the proofs of \Thm s~\ref{thm_ex} and~\ref{thm_no}, including the pure state decomposition, as well as of \Thm~\ref{thm_an}.
Finally, \Sec~\ref{sec_discussion} contains a detailed discussion of further related work.

\subsection{First order approximation: the replica symmetric solution}\label{sec_BP_intro}
As a preparation for the proofs of \Thm s~\ref{thm_ex} and~\ref{thm_no} we need to appraise ourselves of the workings of the proof of first order formula~\eqref{eqDM} for the free energy and the role that Belief Propagation plays in this proof.
As a key step toward~\eqref{eqDM} Dembo and Montanari show that in the case $B>0$ the distribution $\pi_{d,\beta,B}$ from~\eqref{eqDMweakLimit} is the limit of the empirical distribution of the vertex magnetisations on the random graph.
To elaborate, the {\em magnetisation} of a vertex $v_i$ in $\GG$ is defined as
	$$\scal{\SIGMA(v_i)}{\mu_{\GG,\beta,B}}=\mu_{\GG,\beta,B}(\{\SIGMA(v_i)=1\})-\mu_{\GG,\beta,B}(\{\SIGMA(v_i)=-1\}).$$
Accordingly, the empirical distribution of the vertex magnetisations reads
\begin{align}\label{eqempirical}
	\pi_{\GG,\beta,B}&=\frac1n\sum_{i=1}^n\delta_{\scal{\SIGMA(v_i)}{\mu_{\GG,\beta,B}}}\in\cP([-1,1]).
\end{align}
Of course, due to its dependence on the random graph $\GG$, $\pi_{\GG,\beta,B}$ is a {\em random} probability distribution.
The following proposition establishes that $\pi_{\GG,\beta,B}$ converges in probability to the limit $\pi_{d,\beta,B}$ from~\eqref{eqDMweakLimit} if $B>0$.

\begin{proposition}[{\cite[\Thm s~2.5 and~2.7]{Dembo_2010}}]\label{prop_DM_B>0}
	Suppose $B>0$.
	Then
	\begin{align}\label{eqprop_DM_B>0}
		\ex[W_1(\pi_{\GG,\beta,B},\pi_{d,\beta,B})]&=o(1)&&\mbox{and}\\
		\ex\abs{\scal{\SIGMA(v_1)\SIGMA(v_2)}{\mu_{\GG,\beta,B}}-\scal{\SIGMA(v_1)}{\mu_{\GG,\beta,B}}\scal{\SIGMA(v_2)}{\mu_{\GG,\beta,B}}}&=o(1).\label{eqIsingRS}
	\end{align}
\end{proposition}

\noindent
\Prop~\ref{prop_DM_B>0} does not apply to the case $B=0$ (no external field), which is due to the inversion symmetry $\mu_{\GG,\beta,0}(\sigma)=\mu_{\GG,\beta,0}(-\sigma)$ of the Ising model in this case.
We will revisit this delicate issue in \Sec~\ref{sec_B=0low} below.

To grasp the link between the empirical magnetisations $\pi_{\GG,\beta,B}$ and the distribution $\pi_{d,\beta,B}$ we need to remind ourselves of two facts.
First, that the local geometry of the random graph $\GG$ is described by a Galton-Watson process.
Second, that Belief Propagation is `exact on trees'.

On the first point, let $\TT=\TT(d)$ be the (possibly infinite) Galton-Watson tree with root $\root$ where every vertex begets a $\Po(d)$ number of children.
Moreover, let $\TT^{(\ell)}$ be the (finite) tree obtained from $\TT$ by deleting all vertices whose distance from $\root$ exceeds $\ell$.
The following well-known lemma states precisely how $\TT$ captures the local geometry of $\GG$.
Recall that $\nabla^\ell(\GG,v)$ is the subgraph of $\GG$ induced on the vertices $u$ at distance at most $\ell$ from $v$, rooted at $v$.

\begin{lemma}[{e.g.~\cite[\Prop~2.6]{Brasil}}]\label{lem_lwc}
	For any integer $\ell\geq0$ and any rooted tree $T$ we have
	\begin{align*}
		\lim_{n\to\infty}\frac1n\sum_{i=1}^n\vecone\cbc{\nabla^\ell(\GG,v_i)\ism T}&=\pr\brk{\TT^{(\ell)}\ism T}&&\mbox{in probability}.
	\end{align*}
\end{lemma}

Moving on to Belief Propagation, consider a tree $T$ rooted at $\root$ and recall that $\mu_{T,\root,\beta,B}^{(\ell)}$ is the conditional distribution from~\eqref{eqMuT} where we impose a $+1$ `boundary condition' on all vertices at distance precisely $\ell$ from $\root$.
For a vertex $v\in V(T)$ let $T_v$ be the sub-tree of $T$ comprising $v$ and its successors; in other words, $T_v$ contains all vertices $w$ such that the path from $w$ to $\root$ in $T$ passes through $v$.
The root of $T_v$ is vertex $v$ itself.
In the case of the Ising model, Belief Propagation boils down to the following simple recurrence.

\begin{lemma}\label{lem_BP_tree}
	For any $B\geq0$, any tree $T$ rooted at $\root$ and any $\ell\geq1$ we have
	\begin{align}\label{eq_lem_BP_tree}
		\scal{\SIGMA(\root)}{\mu_{T,\beta,B}^{(\ell)}}&=
		\frac{\sum_{s\in\PM}s\eul^{sB}\prod_{v\in\partial(T,\root)}\bc{1+s\scal{\SIGMA(v)}{\mu_{T_v,\beta,B}^{(\ell-1)}}\tanh\beta}}{\sum_{s\in\PM}\eul^{sB}\prod_{v\in\partial(T,\root)}\bc{1+s\scal{\SIGMA(v)}{\mu_{T_v,\beta,B}^{(\ell-1)}}\tanh\beta}}\enspace.
	\end{align}
\end{lemma}

\noindent
The upshot of \Lem~\ref{lem_BP_tree} is that we can work out the root magnetisation $\langle{\SIGMA(\root)},{\mu_{T,\beta,B}^{(\ell)}}\rangle$ of $T$ subject to the all-ones boundary condition `bottom up' from the root magnetisations of the sub-trees $T_v$, $v\in\partial(T,\root)$.
The recurrence~\eqref{eq_lem_BP_tree}  can be verified directly in a few lines (see \Sec~\ref{sec_lem_BP_tree}).
Alternatively, \Lem~\ref{lem_BP_tree} could be derived easily from general statements about Belief Propagation~\cite{Bolthausen,MM}.

\Lem~\ref{lem_BP_tree} clarifies the combinatorial meaning of the operator $\BP_{d,\beta,B}$ from~\eqref{eqBPop}.
Indeed, due to the recursive nature of the Galton-Watson tree $\TT$, the root has a $\Po(d)$ number of children.
Moreover, for each child $v\in\partial(\TT,\root)$ the tree $\TT_v$ is again nothing but a Galton-Watson tree with offspring distribution $\Po(d)$.
Thus, \eqref{eqBPrec} matches the recurrence~\eqref{eq_lem_BP_tree} on the random tree $\TT$.
Consequently, we obtain the following immediate consequence of \Lem~\ref{lem_BP_tree}.

\begin{corollary}\label{cor_BP_tree}
	For any $B\geq0$ and any $\ell\geq0$, the random variable $\Scal{\SIGMA(\root)}{\mu_{\TT,\beta,B}^{(\ell)}}$ has distribution
\begin{align}\label{eqpidbetaBell}
	\pi_{d,\beta,B}^{(\ell)}&=\BP_{d,\beta,B}^{\ell}(\delta_1).
\end{align}
\end{corollary}

\noindent
Hence, $\pi_{d,\beta,B}$ is nothing but the limiting distribution of the root magnetisations $\Scal{\SIGMA(\root)}{\mu_{\TT,\beta,B}^{(\ell)}}$ as $\ell\to\infty$.

In summary, Eq.~\eqref{eqprop_DM_B>0} from \Prop~\ref{prop_DM_B>0} shows that the empirical distribution $\pi_{\GG,\beta,B}$ on the random graph converges to the same limit as the root magnetisation of the random tree $\TT^{(\ell)}$ that describes the local geometry of $\GG$ subject to the all-ones boundary condition at distance $\ell$ from the root.
The fact that the all-ones configuration supplies the `correct' starting point is non-trivial and one of the key facts established in~\cite{Dembo_2010}.
Further, a moment's reflection shows that for $B>0$ the limiting distribution $\pi_{d,\beta,B}$ is not just an atom.
Rather, the distribution $\pi_{d,\beta,B}$ reflects the inherent randomness of the local geometry of $\GG$ around its vertices $v_i$.
This fact, to which the Ising model on the \Erdos-\Renyi\ graph owes much of its mathematical depth, sets the \Erdos-\Renyi\ graph apart from other graph classes such as integer lattices or random regular graphs~\cite{Dembo_2014}.

The second statement~\eqref{eqIsingRS} of \Prop~\ref{eqIsingRS} shows that \whp\ the correlations between the spins of different vertices are benign.
Indeed, because the \Erdos-\Renyi\ random graph model is invariant under vertex permutations,  the statement \eqref{eqIsingRS} is equivalent to
\begin{align}\label{eqRSij}
	\frac1{n^2}\sum_{i,j=1}^n\ex\abs{\scal{\SIGMA(v_i)\SIGMA(v_j)}{\mu_{\GG,\beta,B}}-\scal{\SIGMA(v_i)}{\mu_{\GG,\beta,B}}\scal{\SIGMA(v_j)}{\mu_{\GG,\beta,B}}}&=o(1).
\end{align}
In other words, for `most' vertex pairs $(v_i,v_j)$ the spins $\SIGMA(v_i),\SIGMA(v_j)$ are asymptotically independent, a property known as `replica symmetry' in cavity method parlance~\cite{pnas}.

Given replica symmetry and the distribution of the magnetisations, the Bethe free energy \eqref{eqBFE} is generally known to provide a first order approximation to $\log Z_{\GG}(\beta,B)$~\cite{ACOPBP}.
This can be verified by a coupling argument called the Aizenman-Sims-Starr scheme~\cite{Aizenman}.
The basic idea behind this technique is to reduce the calculation of the free energy to the careful computation of the effect on the partition function of a sequence of local changes to the graph, such as additions of single edges.
In combination with knowledge of the distribution of the magnetisations~\eqref{eqprop_DM_B>0} the replica symmetry condition~\eqref{eqRSij} suffices to compute the impact of such local changes accurately.

That said, the proof of \eqref{eqDM} in~\cite{Dembo_2010} happens to employ a somewhat different strategy.
Instead of the Aizenman-Sims-Starr scheme, Dembo and Montanari interpolate on the inverse temperature parameter $\beta$.
However, the Aizenman-Sim-Starr scheme provides the inspiration for the variance computation that we will perform towards \Thm s~\ref{thm_ex} and~\ref{thm_no}.

\subsection{The central limit theorem in the case $B>0$}\label{sec_B>0}
To prove the central limit theorems (\Thm s~\ref{thm_ex} and~\ref{thm_no}) we develop a refined version of the Belief Propagation formalism that captures the fluctuations of $\log Z_{\GG}(\beta,B)$.
In the case $B>0$ we can directly use some of the aforementioned ingredients of the proof of~\eqref{eqDM}, particularly \Prop~\ref{prop_DM_B>0}.
However, as we will see in \Sec~\ref{sec_B=0low} below, the case $B=0$ turns out to be more delicate.
The proofs for the following statements can be found in \Sec~\ref{sec_thm_ex}, where we prove \Lem~\ref{lem_BP_tree}, \Prop~\ref{proposition_HH}, \Lem~\ref{fact_telescope}, \Lem~\ref{fact_Zratio}, \Prop~\ref{prop_fix_ex}, \Lem~\ref{lem_lwc_tensor} and \Prop s~\ref{prop_tensor_B>0}, \ref{cor_tensor_B>0} and \ref{prop_var_pos_B>0}.
Some of the intermediate statements in this subsection actually also hold for $B=0$, in which case we will be able to reuse them in \Sec~\ref{sec_B=0low} towards the proof of \Thm~\ref{thm_no}.

We are going to derive \Thm~\ref{thm_ex} from a generic martingale central limit theorem.
The main challenge will be to get a handle on the resulting variance process.
The martingale that we use is the well-known edge exposure martingale~\cite[\Chap~7]{AlonSpencer}.
Hence, we think of the random graph $\GG=\GG(n,m)$ as being obtained by picking one random edge at a time without replacement.
Thus, let $\ve_1,\ldots,\ve_m$ be a uniformly random {\em sequence} of pairwise distinct edges and let $E(\GG(n,m))=\{\ve_1,\ldots,\ve_m\}$ be the edge set of the random graph.
Furthermore, let $\fF_{n,M}$ be the $\sigma$-algebra generated by the first $M$ edges $\ve_1,\ldots,\ve_M$.
Hence, given $\fF_{n,M}$ we know $\ve_1,\ldots,\ve_M$, while $\ve_{M+1},\ldots,\ve_m$ remain unexposed, viz.\ random (albeit not quite independent of $\ve_1,\ldots,\ve_M$ because edges are chosen without replacement).
Let
\begin{align}\label{eqZX}
	\vZ_{n,M}&=m^{-1/2}\ex\brk{\log Z_{\GG}(\beta,B)\mid\fF_{n,M}},&\vX_{n,M}&=\vZ_{n,M}-\vZ_{n,M-1}.
\end{align}
Then $(\vZ_{n,M})_{0\leq M\leq m}$ is a Doob martingale and $\vX_{n,M}$ is the associated sequence of martingale differences.
The following is an easy consequence of a general martingale central limit theorem from~\cite{HH}.

\begin{proposition}\label{proposition_HH}
	Let $B\geq0$.
	If there exists a real $\eta=\eta(d,\beta,B)>0$ such that
	\begin{align}\label{eq_fact_HH1}
		\lim_{n\to\infty}\ex\abs{\sum_{M=1}^m\vX_{n,M}^2-\eta^2}&=0
	\end{align}
	then $\lim_{n\to\infty}\vZ_{n,m}-\vZ_{n,0}=\cN(0,\eta^2)$\mbox{ in distribution}.
\end{proposition}

Thus, \Prop~\ref{proposition_HH} reduces the proof of \Thm~\ref{thm_ex} to the analysis of the variance process $(\vX_{n,M}^2)_{1\leq M\leq m}$.
To investigate the variance process we adapt a technique originally developed in~\cite{2sat} to study the random 2-SAT problem to the present setting of the Ising model on random graphs.
Namely, we will express $\vX_{n,M}^2$ in terms of a family of coupled random graphs.
Specifically, we consider three pairs of random graphs
\begin{align}\label{eqsix}
	(\GG_{h,M}(n,m))_{h=1,2},&&(\GG_{h,M}^-(n,m))_{h=1,2},&&(\GG_{h,M}^+(n,m))_{h=1,2}
\end{align}
that each have a certain number of edges in common and additionally contain a number of independent random edges.
The six random graphs are generated via the following, admittedly mildly subtle coupling.
Let $0\leq M\leq m$ and let us assume that $n>n_0(d)$ exceeds a large enough number $n_0(d)$ that depends on $d$ only.

\begin{description}
	\item[CPL0] recall that $(\ve_1,\ldots,\ve_m)$ signifies the edge sequence of the random graph $\GG=\GG(n,m)$.
	\item[CPL1] independently for $h=1,2$ choose an independent sequence $(\ve_{h,M,1},\ldots,\ve_{h,M,m-M+1})$ of pairwise distinct edges uniformly such that
		\begin{align}\label{eqCPL1}
			\{\ve_1,\ldots,\ve_{M-1}\}\cap\bigcup_{1\leq h\leq 2}\{\ve_{h,1},\ldots,\ve_{h,M,m-M+1}\}&=\emptyset&&
		\end{align}
		and let
		\begin{align*}
			\GG_{h,M}&=\GG_{h,M}(n,m)=(V_n,\{\ve_1,\ldots,\ve_{M-1},\ve_{h,M,1},\ldots,\ve_{h,M,m-M}\}),\\
		\GG^-_{h,M}&=\GG_{h,M}^-(n,m)=(V_n,\{\ve_1,\ldots,\ve_{M-1},\ve_{h,M,1},\ldots,\ve_{h,M,m-M+1}\}).
		\end{align*}
	\item[CPL2] independently for $h=1,2$ choose an edge $\ve_{h,M}^+\not\in\{\ve_1,\ldots,\ve_M,\ve_{h,M,1},\ldots,\ve_{h,M,m-M}\}$ uniformly and let
		\begin{align*}
			\GG_{h,M}^+&=\begin{cases}
				(V_n,\{\ve_1,\ldots,\ve_M,\ve_{h,M,1},\ldots,\ve_{h,M,m-M}\})&\mbox{ if }\ve_{M}\not\in\{\ve_{h,M,1},\ldots,\ve_{h,M,m-M}\},\\
				(V_n,\{\ve_1,\ldots,\ve_M,\ve_{h,M,1},\ldots,\ve_{h,M,m-M},\ve_{h,M}^+\})&\mbox{ otherwise}.
			\end{cases}
		\end{align*}
\end{description}

\begin{remark}\label{rem_CPL}
	The random graphs $\GG$ and $(\GG_{h,M},\GG^-_{h,M},\GG^+_{h,M})_{1\leq h\leq2,\,0\leq M\leq m}$ are defined on the same probability space.
	The edge sequences $(\ve_{h,M,1},\ldots,\ve_{h,M,m-M+1})_{h=1,2}$ and the edges $\ve_{h,M}^+$ are understood to be stochastically independent for different values of $M$.
	But of course the sequence $\ve_1,\ldots,\ve_m$ is independent of and thus identical for all $M$.
\end{remark}

By construction, the two random graphs $\GG_{1,M},\GG_{2,M}$ have $M-1$ edges $\ve_1,\ldots,\ve_{M-1}$ in common with each other, as well as with the \Erdos-\Renyi\ graph $\GG$.
Additionally, each $\GG_{h,M}$ contains another $m-M$ random edges that are chosen independently for $h=1,2$.
In particular, for each $h=1,2$ the individual random graphs $\GG_{h,M}$ have the same distribution as the \Erdos-\Renyi\ graph $\GG(n,m-1)$ with $m-1$ edges.
Analogously, the random graphs $(\GG_{h,M}^-)_{h=1,2}$ share $M-1$ common edges and each have a total of $m$ edges.
Thus, individually $\GG_{h,M}^-$ is an \Erdos-\Renyi\ graph $\GG(n,m)$.
Finally, the graphs $(\GG_{h,M}^+)_{h=1,2}$ have a total of $M$ shared and $m-M$ independent edges, and each $\GG_{h,M}^+$ separately is just a $\GG(n,m)$.

The coupling {\bf CPL0}--{\bf CPL2} gives us a handle on the variance process $(\vX_{n,M}^2)_{1\leq M\leq m}$.
To see this, we introduce the three random variables
	\begin{align}\label{eqDeltaM}
		\vDelta_{n,M}^+&=\prod_{h=1,2}\log\frac{Z_{\GG_{h,M}^+}(\beta,B)}{Z_{\GG_{h,M}}(\beta,B)} ,&
		\vDelta^-_{n,M}&=\prod_{h=1,2}\log\frac{Z_{\GG_{h,M}^-}(\beta,B)}{Z_{\GG_{h,M}}(\beta,B)} ,&
		\vDelta^\pm_{n,M}&=\log\frac{Z_{\GG_{1,M}^+}(\beta,B)}{Z_{\GG_{1,M}}(\beta,B)}\log\frac{Z_{\GG_{2,M}^-}(\beta,B)}{Z_{\GG_{2,M}}(\beta,B)} .
	\end{align}

\begin{lemma}\label{fact_telescope}
	For all $B\geq0$ and $1\leq M\leq m$ we have $\vX_{n,M}^2=m^{-1}\ex\brk{\vDelta_{n,M}^++\vDelta_{n,M}^--2\vDelta_{n,M}^\pm\mid\fF_{n,M}}$.
\end{lemma}

The merit of expressing the squared martingale differences in terms of $\vDelta_M^+,\vDelta_M^-,\vDelta_M^\pm$ is that the expressions on the right hand sides of~\eqref{eqDeltaM} correspond to `local' changes, similar in spirit to the Aizenman-Sims-Starr scheme that we alluded to in \Sec~\ref{sec_BP_intro}.
For instance, the ratio ${Z_{\GG_{h,M}^+}(\beta,B)}/{Z_{\GG_{h,M}}(\beta,B)}$ essentially accounts for the change in partition function upon adding one more shared edge $\ve_{M}$.
The ensuing changes can be expressed in terms of the Ising distributions of the `small' graph $\GG_{h,M}$.

\begin{lemma}\label{fact_Zratio}
	Suppose $B\geq0$.
	With probability $1-O(1/n)$ we have
	\begin{align}
		\log\frac{Z_{\GG_{h,M}^+}(\beta,B)}{Z_{\GG_{h,M}}(\beta,B)}&=\log\scal{\exp\bc{\beta\SIGMA(\ve_{M})}}{\mu_{\GG_{h,M},\beta,B}},&
		\log\frac{Z_{\GG_{h,M}^-}(\beta,B)}{Z_{\GG_{h,M}}(\beta,B)}&=\log\scal{\exp\bc{\beta\SIGMA(\ve_{h,m-M+1})}}{\mu_{\GG_{h,M},\beta,B}}.\label{eqfact_Zratio}
	\end{align}
\end{lemma}

In light of \Lem~\ref{fact_Zratio} it would be fairly straightforward to use \Prop~\ref{prop_DM_B>0} to evaluate the expectations
\begin{align*}
	\ex\brk{\log\frac{Z_{\GG_{h,M}^+}(\beta,B)}{Z_{\GG_{h,M}}(\beta,B)}\mid\fF_{n,M}},&&
	\ex\brk{\log\frac{Z_{\GG_{h,M}^-}(\beta,B)}{Z_{\GG_{h,M}}(\beta,B)}\mid\fF_{n,M}}
\end{align*}
separately (at least if $B>0$).
But as \Lem~\ref{fact_telescope} shows what we actually need are the conditional expectations of $\vDelta_M^+,\vDelta_M^-,\vDelta_M^\pm$, which each comprise a {\em product} of terms of the form~\eqref{eqfact_Zratio}.
In fact, involving the correlated random graphs from~\eqref{eqsix}, the factors that make up $\vDelta_M^+,\vDelta_M^-,\vDelta_M^\pm$ are stochastically dependent.
In order to evaluate these products, we need to control the empirical distribution
	\begin{align*}
		\vec\pi_{n,M,\beta,B}^\tensor&=\frac1n\sum_{i=1}^n\delta_{\bc{\scal{\SIGMA(v_i)}{\mu_{\GG_{1,M},\beta,B}},\scal{\SIGMA(v_i)}{\mu_{\GG_{2,M},\beta,B}}}}\in\cP([-1,1]^2)
	\end{align*}
of the {\em joint} magnetisations of the vertices $v_i$ in $\GG_{1,M}$ and $\GG_{2,M}$.
Indeed, because the random graphs $(\GG_{h,M})_{h=1,2}$ are stochastically dependent, so are the magnetisations
\begin{align*}
	\scal{\SIGMA(v_i)}{\mu_{\GG_{h,M},\beta,B}}&=\mu_{\GG_{h,M},\beta,B}(\{\SIGMA(v_i)=1\})-\mu_{\GG_{h,M},\beta,B}(\{\SIGMA(v_i)=-1\})&&(h=1,2).
\end{align*}

What we are going to show is that the empirical distribution $\vec\pi_{n,M,\beta,B}^\tensor$ of these correlated pairs of magnetisations converges to the measure $\pi_{d,\beta,B,M/m}$ from~\eqref{eqfix_ex}.
But before we come to that we need to ascertain that the limits $\pi_{d,\beta,B,t}$ from~\eqref{eqfix_ex} actually exist.

\begin{proposition}\label{prop_fix_ex}
	For any $B\geq0$, $t\in[0,1]$ the weak limit~\eqref{eqfix_ex} exists and the map
	\begin{align}\label{eqprop_fix_ex}
		t\in[0,1]\mapsto\pi_{d,\beta,B,t}^\tensor
	\end{align}
	is $W_1$-continuous.
\end{proposition}


Based on \Prop~\ref{prop_fix_ex} we can now establish the convergence of the empirical distributions $\vec\pi_{n,M,\beta,B}^\tensor$ in the case $B>0$.

\begin{proposition}\label{prop_tensor_B>0}
	Suppose $B>0$.
	Then uniformly for all $0\leq M\leq m$ we have
	\begin{align}\label{eqprop_tensor_B>0}
		\ex[W_1(\vec\pi_{n,M,\beta,B}^\tensor,\pi_{d,\beta,B,M/m})]&=o(1)&&\mbox{and}\\
		\ex\abs{\scal{\SIGMA(v_1)\SIGMA(v_2)}{\mu_{\GG_{h,M},\beta,B}}-\scal{\SIGMA(v_1)}{\mu_{\GG_{h,M},\beta,B}}\scal{\SIGMA(v_2)}{\mu_{\GG_{h,M},\beta,B}}}&=o(1).\label{eqTensorB>0RS}
	\end{align}
\end{proposition}

\Prop~\ref{prop_tensor_B>0} is to the proof of the central limit theorem what \Prop~\ref{prop_DM_B>0} is to the proof of the first order approximation~\eqref{eqDM}.
Accordingly, the proof of \Prop~\ref{prop_tensor_B>0} combines the Belief Propagation recurrence~\eqref{eq_lem_BP_tree} with local weak convergence techniques, a bit like the proof of \Prop~\ref{prop_DM_B>0} does.
The main difference here is that we are dealing with a pair $(\GG_{1,M},\GG_{2,M})$ of correlated random graphs, rather than the plain random graph $\GG$.
Hence, we need to devise a Galton-Watson process that captures the local geometry around a vertex $v_i$  jointly in the two random graphs $(\GG_{h,M})_{h=1,2}$.

The following three-type Galton-Watson tree $\TT^\tensor=\TT^\tensor(d,t)$ fits this bill.
The three vertex types are simply type~0, type~1 and type~2.
The root $\root$ is type~0.
A type~0 vertex begets a $\Po(dt)$ number of type~0 vertices as offspring, as well as $\Po(d(1-t))$ many type~1 vertices and $\Po(d(1-t))$ type~2 vertices; these three Poissons are mutually independent.
Moreover, for $h=1,2$ the offspring of a type~$h$ vertex is just a $\Po(d)$ number of type~$h$ vertices.

We say that two possible outcomes $T,T'$ of $\TT^\tensor$ are {\em isomorphic} if there exists a tree isomorphism $T\to T'$ that preserves the root and the types.
Further, for $h=1,2$ we let $\TT^\tensor[h]$ be the sub-tree of $\TT^\tensor$ obtained by deleting all vertices of type $3-h$.
In addition, for $\ell\geq0$ we obtain $\TT^{\tensor\,(\ell)}[h]$ from $\TT^{\tensor}[h]$ by deleting all vertices whose distance from $\root$ exceeds $\ell$.

As a next step we relate the Galton-Watson tree $\TT^\tensor$ to the correlated random graphs $(\GG_{h,M})_{h=1,2}$.
Thus, let $T$ be a possible outcome of $\TT^\tensor$.
We call a vertex $v\in V_n$ an {\em $\ell$-instance} of $T$ if there exist graph isomorphisms $\iota_h: T^{(\ell)}[h]\to  \nabla^\ell(\GG_{h,M},v)$ with the following properties for $h=1,2$:
\begin{description}
	\item[ISM1] $\iota_h(\root)=v$,
	\item[ISM2] $u\in T$ has type~0 iff $\iota_1(u)=\iota_2(u)$.
\end{description}
In other words, we ask that the depth-$\ell$ neighbourhood of $v$ in $\GG_{h,M}$ be isomorphic to the truncated tree $T^{(\ell)}[h]$, and vertices of type~0 be mapped to identical vertices of $\GG_{h,M}$ for $h=1,2$.
Hence, the type~0 vertices of $T$ represent the endpoints of shared edges of $(\GG_{h,M})_{h=1,2}$.

Let $\vN_{T}^{(\ell)}(n,M)$ be the number of $\ell$-instances of $T$.
The proof of \Prop~\ref{prop_tensor_B>0} builds upon the following local weak convergence result that verifies that $\TT^\tensor$ faithfully models the joint local topology of $(\GG_{h,M})_{h=1,2}$.

\begin{lemma}\label{lem_lwc_tensor}
	Suppose that $t\in[0,1]$  and $M\sim tm$ and let $\ell\geq0$.
	Then for any possible outcome $T$ of $\TT^\tensor$ we have
	\begin{align*}
		\ex\abs{\vN_T^{(\ell)}(n,M)/n-\pr\brk{\TT^{\tensor\,(\ell)}\ism T^{(\ell)}}}&=o(1).
	\end{align*}
\end{lemma}

With \Prop~\ref{prop_tensor_B>0} pinpointing the empirical distribution of the joint marginals and also providing `replica symmetry', we are ready to evaluate the expressions from~\eqref{eqDeltaM} by way of an Aizenman-Sims-Starr argument.

\begin{proposition}\label{cor_tensor_B>0}
	Suppose $B>0$.
	The integral $\Sigma(d,\beta,B)^2$ from~\eqref{eqSigmadbetaB} exists and
	\begin{align}\label{eq_cor_tensor_B>0_cor_tensor_B>0}
		\ex\abs{\sum_{M=1}^m\vX_{n,M}^2-\Sigma(d,\beta,B)^2}&=o(1).
	\end{align}
\end{proposition}

At this point one last step remains.
Namely, the proof of \Prop~\ref{cor_tensor_B>0} falls short of revealing that $\Sigma(d,\beta,B)^2$ is strictly positive.
To verify that this is the case we exploit the fact that $\pi_{d,\beta,B}$ fails to concentrate on a single point.
Together with the continuity of the map from~\eqref{eqprop_fix_ex} and a convexity argument we thus obtain the final piece of the jigsaw.

\begin{proposition}\label{prop_var_pos_B>0}
	Suppose $d,\beta,B>0$.
	Then $\Sigma(d,\beta,B)>0$.
\end{proposition}

\begin{proof}[Proof of \Thm~\ref{thm_ex}]
	The theorem is an immediate consequence of \Prop s~\ref{proposition_HH},  \ref{cor_tensor_B>0} and \ref{prop_var_pos_B>0}.
\end{proof}

\subsection{The central limit theorem in the case  $B=0$}\label{sec_B=0low}
We proceed to survey the proof of \Thm~\ref{thm_no}, the central limit theorem in the case $B=0$, $d>1$ and $\beta>\betaf(d)$.
Significant complications in this case derive from the inversion symmetry $\mu_{\GG,\beta,0}(\sigma)=\mu_{\GG,\beta,0}(-\sigma)$.
An immediate consequence of inversion symmetry is that we cannot expect the replica symmetry condition \eqref{eqIsingRS} to hold; yet replica symmetry was a cornerstone of the proof of the central limit theorem for $B>0$.
Instead, we are going to show explicitly that the phase space $\PM^{V_n}$ decomposes into two distinct `pure states' in the sense of~\cite{Parisi} and that {\em within} each of these pure states replica symmetry is restored.
This pure state decomposition (and its proof) should be of independent interest.
The proofs for the following statements can be found in \Sec s \ref{sec_B=0}, where we prove \Prop s~\ref{prop_BP_B=0}, \ref{prop_pure}, \ref{prop_tensor}, \ref{cor_tensor} and~\ref{prop_var_pos}.

Let us begin by doing our homework on the fixed point distribution $\pi_{d,\beta,0}$ from~\eqref{eqDMweakLimit} in the case $B=0$.
Let $\rho_d$ be the smallest positive solution to the equation $\rho=\exp(d(\rho-1))$.
Thus, $\rho_d$ is the extinction probability of a $\Po(d)$-Galton-Watson process.
Also recall $r_{d,\beta}$ from~\eqref{eqrdbeta} and remember that $\fX_d$ signifies the (countable) set of $\beta\in(0,\infty)$ where the function $\beta\mapsto r_{d,\beta}$ is discontinuous.

\begin{proposition}\label{prop_BP_B=0}
	The weak limit $\pi_{d,\beta,0}=\lim_{\ell\to\infty}\BP^\ell_{d,\beta,0}(\delta_\infty)$ exists and enjoys the following properties.
	\begin{enumerate}[(i)]
		\item We have $\supp{\pi_{d,\beta,0}}\subset[0,\infty)$.
		\item For $d>1$ and $\beta>\betaf(d)$ we have $\pi_{d,\beta}(\{0\})=\rho_d$ and $\int_0^\infty x\dd\pi_{d,\beta}(x)>0$.
		\item The function $\beta\mapsto r_{d,\beta}$ is monotonically increasing and
			\begin{align}\label{eqmax}
				\lim_{n\to\infty}\frac1n\ex\brk{\sum_{e\in E(\GG)}\scal{\SIGMA(e)}{\mu_{\GG,\beta,0}}}&=d\cdot r_{d,\beta}&&\mbox{for all $\beta\in(0,\infty)\setminus\fX_d$}.
			\end{align}
	\end{enumerate}
\end{proposition}

\noindent
Part~(iii) of the proposition, which clarifies the combinatorial meaning of the function $r_{d,\beta}$, is an easy consequence of statements from~\cite{Basak,Dembo_2010}; see \Sec~\ref{sec_prop_BP_B=0} for details.

We now turn our attention to the `pure states' of the Ising model on $\GG(n,m)$.
This decomposition actually undergoes a phase transition at $\beta=\betaf(d)$.
Indeed, in \Sec~\ref{sec_B=0high} we will see that for $d\leq1$ or $d>1$, $\beta<\betaf(d)$ the distribution $\mu_{\GG,\beta,0}$ concentrates on configurations $\sigma\in\PM^{V_n}$ with average magnetisation $n^{-1}\sum_{i=1}^n\sigma(v_i)=o(1)$, which implies that $\mu_{\GG,\beta,0}$ enjoys the replica symmetry property \whp\
By contrast, we are about to verify that $\mu_{\GG,\beta,0}$ is no longer replica symmetric in the `low temperature' regime $d>1$ and $\beta>\betaf(d)$ \whp\
Yet we will demonstrate that replica symmetry is restored once we condition on either of the two events
\begin{align*}
	\cS_+&=\cS_{+,n}=\cbc{\sigma\in\PM^{V_n}:\sum_{v\in V_n}\sigma(v)>0},&\cS_-=\cS_{-,n}&=\cbc{\sigma\in\PM^{V_n}:\sum_{v\in V_n}\sigma(v)<0}.
\end{align*}
Moreover, once we condition on $\cS_+$ the empirical distribution of the magnetisations converges to the limit $\pi_{d,\beta,0}$ from~\eqref{eqDMweakLimit}.
Of course, due to inversion symmetry this implies that the empirical distribution of the {\em negated} conditional magnetisations given $\cS_-$ converges to $\pi_{d,\beta,0}$ as well.

To make these statements precise let
\begin{align*}
	\vec\pi_{n,m,\beta}^+&=\frac1n\sum_{i=1}^n\delta_{\scal{\SIGMA(v_i)}{\mu_{\GG,\beta,0}(\nix\mid\cS_n^+)}}
\end{align*}
be the empirical distribution of the conditional magnetisations
\begin{align*}
	\scal{\SIGMA(v_i)}{\mu_{\GG,\beta,0}(\nix\mid\cS_+)}&=\mu_{\GG,\beta,0}(\{\SIGMA(v_i)=1\}\mid\cS_+)-\mu_{\GG,\beta,0}(\{\SIGMA(v_i)=-1\}\mid\cS_+).
\end{align*}
Also recall from~\eqref{eqMuT} that $\mu_{\GG,v_1,\beta,0}^{(\ell)}$ denotes the Ising model with a $+1$ boundary condition imposed on all vertices at distance $\ell$ from $v_1$.

\begin{proposition}\label{prop_pure}
	If $d>1$ and $\beta\in(\betaf(d),\infty)\setminus\fX_d$, then
	\begin{enumerate}[(i)]
		\item $\ex\abs{\mu_{\GG,\beta,0}(\cS_+)-1/2}=o(1)$,
		\item $\ex\abs{\scal{\SIGMA(v_1)\SIGMA(v_2)}{\mu_{\GG,\beta,0}(\nix\mid\cS_+)}-\scal{\SIGMA(v_1)}{\mu_{\GG,\beta,0}(\nix\mid\cS_+)}\scal{\SIGMA(v_2)}{\mu_{\GG,\beta,0}(\nix\mid\cS_+)}}=o(1)$, and
		\item $\ex[W_1(\vec\pi^+_{n,m,\beta},\pi_{d,\beta,0})]=o(1)$ and
		$$\limsup_{\ell\to\infty}\limsup_{n\to\infty}\ex\abs{\scal{\SIGMA(v_1)}{\mu_{\GG,\beta,0}(\nix\mid\cS_+)}-\scal{\SIGMA(v_1)}{\mu^{(\ell)}_{\GG,v_1,\beta,0}}}=0.$$
	\end{enumerate}
\end{proposition}

While~(i) above is a relatively simple consequence of  the inversion symmetry of $\mu_{\GG,\beta,0}$, the other two statements are substantial.
Indeed, (ii) establishes replica symmetry of the conditional measure $\mu_{\GG,\beta,0}(\nix\mid\cS_+)$.
Similarly as in Eq.~\eqref{eqRSij}, due to invariance under vertex permutations \Prop~\ref{prop_pure}~(ii) is equivalent to
\begin{align}\label{eqRSij+}
	\frac1{n^2}\sum_{i,j=1}^n\ex\abs{\scal{\SIGMA(v_i)\SIGMA(v_j)}{\mu_{\GG,\beta,0}(\nix\mid\cS_n^+)}-\scal{\SIGMA(v_i)}{\mu_{\GG,\beta,0}(\nix\mid\cS_n^+)}\scal{\SIGMA(v_j)}{\mu_{\GG,\beta,0}(\nix\mid\cS_n^+)}}=o(1).
\end{align}
Hence, given $\cS_+$ the spins $\SIGMA(v_i),\SIGMA(v_j)$ are asymptotically independent for all but $o(n^2)$ pairs $i,j$.
Finally, (iii) demonstrates that the (random) empirical distributions $\vec\pi_{n,m,\beta}^+$ converge to $\pi_{d,\beta,0}$ in probability.
In light of the discussion from~\cite{Parisi}, the three properties (i)--(iii) justify our referring to $\cS_+,\cS_-$ as {\em pure states}.

We emphasise that the conditioning on $\cS_+$ is crucial in \Prop~\ref{prop_pure}.
Due to inversion symmetry the unconditional magnetisations just satisfy $\scal{\SIGMA(v_i)}{\mu_{\GG,\beta,0}}=0$; thus, they are uninformative.
By contrast, \Prop~\ref{prop_BP_B=0}~(ii) and \Prop~\ref{prop_pure}~(iii) imply that the empirical distribution $\vec\pi^+_{n,m,\beta}$ of the conditional magnetisations converges to a limiting distribution with a strictly positive mean.
This fact together with the replica symmetry provided by \Prop~\ref{prop_pure}~(ii) easily implies that for the unconditional measure $\mu_{\GG,\beta,0}$ we have
\begin{align}\label{eqnors}
	\liminf_{n\to\infty}\ex\abs{\scal{\SIGMA(v_1)\SIGMA(v_2)}{\mu_{\GG,\beta,0}}-\scal{\SIGMA(v_1)}{\mu_{\GG,\beta,0}}\scal{\SIGMA(v_2)}{\mu_{\GG,\beta,0}}}&>0&&\mbox{if $d>1$ and $\beta>\betaf(d)$.}
\end{align}
In other words, the unconditional $\mu_{\GG,\beta,0}$ genuinely {\em fails} to be replica symmetric.
The proof of \Prop~\ref{prop_pure}
combines the insights from \Prop~\ref{prop_BP_B=0} and techniques for the study of diluted spin glass models~\cite{Victor,Coja_2018,ACOPBethe,ACOPBP} with monotonicity and positive correlation arguments.

Beyond the proof of \Thm~\ref{thm_no} \Prop~\ref{prop_pure} has several interesting combinatorial implications.
For example, it is easy to deduce from the proposition that the average magnetisation given $\cS_+$ concentrates on the strictly positive mean of $\pi_{d,\beta,0}$ if $d>1$, $\beta>\betaf(d)$.
Furthermore, the \Prop~\ref{prop_pure} implies in combination with arguments from~\cite{ACOPBethe} that $\cS_{\pm}$ are `Bethe states' in the sense of~\cite{MM}.
But since these facts are not required to prove \Thm~\ref{thm_no} we are not going to elaborate in the present work.

Ploughing ahead with the proof of \Thm~\ref{thm_no}, as in \Sec~\ref{sec_B>0} we are going to calculate $\vDelta_M^+,\vDelta_M^-,\vDelta_M^\pm$ from~\eqref{eqDeltaM} by way of their `local' expressions \eqref{eqfact_Zratio}.
To this end we observe that, due to \Prop~\ref{prop_pure}~(i) and inversion symmetry, \eqref{eqfact_Zratio} can be rewritten in terms of the conditional distribution $\mu_{\GG_{h,M},\beta,0}(\nix\mid\cS_+)$:
	\begin{align}\label{eqfact_Zratio1cond}
		\log\frac{Z_{\GG_{h,M}^+}(\beta,0)}{Z_{\GG_{h,M}}(\beta,0)}&=\log\scal{\exp\bc{\beta\SIGMA(\ve_{M})}}{\mu_{\GG_{h,M},\beta,0}(\nix\mid\cS_+)}+o(1)&&\mbox\whp,\\
		\log\frac{Z_{\GG_{h,M}^-}(\beta,0)}{Z_{\GG_{h,M}}(\beta,0)}&=\log\scal{\exp\bc{\beta\SIGMA(\ve_{h,m-M+1})}}{\mu_{\GG_{h,M},\beta,0}(\nix\mid\cS_+)}+o(1)&&\mbox\whp\label{eqfact_Zratio2cond}
	\end{align}

Since \Prop~\ref{prop_pure} implies that the conditional distributions $\mu_{\GG_{h,M},\beta,0}(\nix\mid\cS_n^+)$ are replica symmetric \whp, in order to evaluate~\eqref{eqfact_Zratio1cond}--\eqref{eqfact_Zratio2cond} we basically just need to get a handle on the empirical joint distribution of the conditional magnetisations, i.e.,
	\begin{align}\label{eq_emp_tensor}
		\vec\pi_{n,m,\beta,M}^\oplus&=\frac1n\sum_{i=1}^n\delta_{\bc{\scal{\SIGMA(v_i)}{\mu_{\GG_{1,M},\beta,0}(\nix\mid\cS_+)},\scal{\SIGMA(v_i)}{\mu_{\GG_{2,M},\beta,0}(\nix\mid\cS_+)}}}\in\cP([-1,1]^2)\enspace.
	\end{align}
At this point we do not need to emphasise anymore that the conditional magnetisations
\begin{align*}
	\scal{\SIGMA(v_i)}{\mu_{\GG_{h,M},\beta,0}(\nix\mid\cS_+)}&=\mu_{\GG_{h,M},\beta,0}(\{\SIGMA(v_i)=1\}\mid\cS_+)-\mu_{\GG_{h,M},\beta,0}(\{\SIGMA(v_i)=-1\}\mid\cS_+)
	&&(h=1,2)
\end{align*}
are correlated because the random graphs $(\GG_{h,M})_{h=1,2}$ are.
We proceed to show that $\vec\pi_{n,m,\beta,M}^\oplus$ converges to the distribution $\pi_{d,\beta,0,t}^\tensor$ from~\eqref{thm_no}, the existence of which already follows from \Prop~\ref{prop_fix_ex}.

\begin{proposition}\label{prop_tensor}
	Suppose $d>1,\beta\in(\betaf(d),\infty)\setminus\fX_d$.
	Then uniformly for all $0<M\leq m$ we have
	\begin{align}\label{eq_prop_tensor_W1}
		\ex[W_1(\vec\pi_{n,m,\beta,M}^\oplus,\pi^\tensor_{d,\beta,M/m})]&=o(1),\\
		\ex\abs{\scal{\SIGMA(v_1)\SIGMA(v_2)}{\mu_{\GG_{h,M},\beta,0}(\nix\mid\cS_n^+)}-\scal{\SIGMA(v_1)}{\mu_{\GG_{h,M},\beta,0}(\nix\mid\cS_n^+)}\scal{\SIGMA(v_2)}{\mu_{\GG_{h,M},\beta,0}(\nix\mid\cS_n^+)}}&=o(1).\label{eq_prop_tensor_rs}
	\end{align}
\end{proposition}

\noindent
The proof of \Prop~\ref{lem_lwc_tensor} builds upon the local weak convergence to the multi-type Galton-Watson tree $\TT^\tensor$ from \Lem~\ref{lem_lwc_tensor} and the insights from \Prop s~\ref{prop_BP_B=0} and~\ref{prop_pure}.

In analogy to the case $B>0$, we combine the conditional replica symmetry property from \Prop~\ref{prop_pure} with \Prop~\ref{prop_tensor} to calculate the squared martingale differences via an Aizenman-Sims-Starr argument.

\begin{proposition}\label{cor_tensor}
	Suppose $d>1,\beta\in(\betaf(d),\infty)\setminus\fX_d$.
	Then the integral $\Sigma(d,\beta,0)^2$ from~\eqref{eqSigmadbetaB} exists and
	\begin{align}\label{eq_cor_tensor_B>0_cor_tensor}
		\ex\abs{\sum_{M=1}^m\vX_{n,M}^2-\Sigma(d,\beta,0)^2}&=o(1).
	\end{align}
\end{proposition}

Finally, we need to verify that $\Sigma(d,\beta,0)$ is strictly positive if $d>1,\beta>\betaf(d)$.
As in the case $B>0$ this ultimately follows from the fact that $\pi_{d,\beta,0}$ fails to concentrate on a single number.

\begin{proposition}\label{prop_var_pos}
	Suppose $d>1,\beta>\betaf(d)$.
	Then $\Sigma(d,\beta,0)>0$.
\end{proposition}

\begin{proof}[Proof of \Thm~\ref{thm_no}]
	The theorem is an immediate consequence of \Prop~\ref{proposition_HH},  \Prop~\ref{cor_tensor} and \Prop~\ref{prop_var_pos}.
\end{proof}

\subsection{The case $B=0$, high temperature}\label{sec_B=0high}

\Thm~\ref{thm_an} asserts that in the high temperature scenario without an external field, i.e., $B=0$ and $d\leq1$ or $\beta<\betaf(d)$ the free energy only exhibits bounded fluctuations.
To prove this fact we employ the small subgraph conditioning method~\cite{Janson_1993,RW2}.
Thus, we are going to argue that the variance of the partition function $Z_{\GG}(\beta,0)$ can be attributed entirely to the presence of bounded-length cycles in $\GG$.
Since it is well known that the number of cycles of a fixed length $\ell\geq3$ in $\GG$ converges to $\vc_\ell\disteq\Po(d^\ell/(2\ell))$, we will thus obtain the limiting distribution stated in \Thm~\ref{thm_an}.
The proofs for the following statements can be found in \Sec~\ref{sec_thm_an1}, where we prove \Prop s~\ref{prop_EZ}, \ref{prop_ssc} and \ref{prop_EZ2}.

To prove \Thm~\ref{thm_an} we actually apply a twist: we are not going to perform small subgraph conditioning on the partition function $Z_{\GG}(\beta,0)$, but work with a truncated random variable instead.
Specifically, let
\begin{align}\label{eqZbal}
	\Zbal(\beta)&=\sum_{\sigma\in\PM^{V_n}}\vecone\cbc{\abs{\sum_{i=1}^n\sigma(v_i)}\leq2\sqrt n\log n}\exp\bc{\beta\sum_{vw\in E(\GG)}\sigma_v\sigma_w}
\end{align}
Thus, $\Zbal(\beta)$ only takes into account `almost paramagnetic' configurations $\sigma$, i.e., configurations whose empirical magnetisation $n^{-1}\sum_{i=1}^n\sigma(v_i)$ is not very much bigger than $n^{-1/2}$ in absolute value.

The first step of the small subgraph conditioning method is to work out the precise asymptotics of the first moment.
The following proposition provides such a precise asymptotic expression for $\beta<\betaf(d)$.
Indeed, the proposition shows that for $\beta<\betaf(d)$ the first moments of $Z_{\GG}(\beta,0)$ and $\Zbal(\beta)$ are asymptotically the same.

\begin{proposition}\label{prop_EZ}
	Assume that $d\leq1$ or $\beta<\betaf(d)$.
	Then
	\begin{align*}
		\Erw[Z_{\GG}(\beta,0)]&\sim\Erw[\Zbal(\beta)]\sim
		\frac{2^{n}\cosh^m\beta}{\sqrt{1 - d  \tanh\beta}}  \exp\bc{-\frac{d}2\tanh \beta-\frac{d^2}4 \tanh^2 \beta}.
	\end{align*}
\end{proposition}

As a next step we assess the correlations between the numbers of short cycles in $\GG$ and the truncated partition function $\Zbal(\beta)$.
Thus, let $\vC_\ell$ be the number of cycles of length $\ell$ in $\GG$.
Moreover, let
\begin{align}\label{eqlambdadelta}
	\vartheta_\ell&=\frac{d^\ell}{2\ell},&\chi_\ell&=\tanh^\ell\beta.
\end{align}

\begin{proposition}\label{prop_ssc}
	Assume that $d\leq1$ or $\beta<\betaf(d)$.
	Then for any fixed $L$ and $c_3,\ldots,c_L\geq0$ we have
	\begin{align*}
		\lim_{n\to\infty}\frac{\ex\brk{\Zbal(\beta)\mid\vC_3=c_3,\ldots,\vC_L=c_L}}{\ex\brk{\Zbal(\beta)}}&=\prod_{\ell=3}^L(1+\chi_\ell)^{c_\ell}\exp(-\vartheta_\ell\chi_\ell).
	\end{align*}
\end{proposition}

Finally, we work out the precise asymptotics of the second moment of $\Zbal(\beta)$.

\begin{proposition}\label{prop_EZ2}
	Assume that $d\leq1$ or $\beta<\betaf(d)$.
	Then
	\begin{align*}
		\ex[\Zbal(\beta)^2]&\sim\ex[\Zbal(\beta)]^2\exp\bc{\sum_{\ell\geq3}\vartheta_\ell\chi_\ell^2}.
	\end{align*}
\end{proposition}

\Prop~\ref{prop_EZ}, \ref{prop_ssc} and~\ref{prop_EZ2} precisely verify the conditions of the small subgraph conditioning method as summarised in~\cite[\Thm~1]{Janson_1993} (see \Thm~\ref{thm_janson} below).
Therefore, the proof of \Thm~\ref{thm_an} is now immediate.

\begin{proof}[Proof of \Thm~\ref{thm_an}]
	\Thm~\ref{thm_an} follows from \Prop s~\ref{prop_EZ}--\ref{prop_EZ2} and~\cite[\Thm~1]{Janson_1993}.
\end{proof}

\subsection{Discussion}\label{sec_discussion}
As pointed out at the start of the section, the contribution most closely related to the present work on a technical level is the recent article by Chatterjee, Coja-Oghlan, M\"uller, Riddlesden, Rolvien, Zakharov and Zhu on the random 2-SAT problem~\cite{2sat}.
Specifically, that article shows that throughout the satisfiable phase of the random 2-SAT problem the number of satisfying assignments is asymptotically log-normal.
Similarly to the present work (see the coupling {\bf CPL0--CPL2}), the proof is based on investigating pairs of correlated random 2-SAT instances that share a certain number of clauses and each contain a number of independent random clauses.
That proof strategy, in turn, is inspired by the proof of the central limit theorem for fully connected spin glass models by Chen, Dey and Panchenko~\cite{ChenDeyPanchenko}, which is based on an interpolating family of Hamiltonians.

The key fact that facilitates the proof of the 2-SAT central limit theorem is that the 2-SAT problem on the appropriate Galton-Watson tree has a unique Gibbs measure throughout the satisfiable phase~\cite{AchlioptasEtAl}.
This fact immediately implies that the uniform distribution on satisfying assignments is `replica symmetric' and that the empirical marginals converge to the Belief Propagation fixed point on the Galton-Watson tree.
As we pointed out earlier, the Ising model on the \Erdos-\Renyi\ graph does not generally enjoy the Gibbs uniqueness property.
Yet the work of Dembo and Montanari~\cite{Dembo_2010} provides a workable substitute for Gibbs uniqueness in the presence of an external field, i.e., $B>0$.
To be precise, the results from \cite[\Sec~4]{Dembo_2010} imply that in the case $B>0$ the free boundary condition and the all-$1$ boundary condition induce the same Gibbs measure on the Galton-Watson tree $\TT$.
This implies replica symmetry and the convergence of the empirical magnetisations; \Prop~\ref{prop_DM_B>0} summarises the precise statement.
Thus, by harnessing these results in place of Gibbs uniqueness, we can prove \Thm~\ref{thm_ex} concerning the case $B>0$ by following a broadly similar strategy as in~\cite{2sat}.

By contrast, in the case $B=0$ matters get more involved.
For a start, for $d>1$ and $\beta>\betaf(d)$ it is not just that the Gibbs measure on the Galton-Watson tree is no longer unique, but the free boundary condition renders an uninformative and unstable Belief Propagation fixed point distribution (the atom at zero).
Furthermore, as we pointed out in Eq.~\eqref{eqnors}, the replica symmetry condition~\eqref{eqIsingRS} does not hold in this case.
Instead, we need to derive the pure state decomposition as stated in \Prop~\ref{prop_pure}.
The proof of this proposition constitutes the main technical challenge of the present work.
To elaborate, the proof of \Prop~\ref{prop_pure} relies on a general information theoretic decomposition result for probability measures called the `pinning lemma' from~\cite{Coja_2018}, see \Lem~\ref{prop_pinning} below.
The pinning lemma is a generalisation of earlier results from Montanari~\cite{MontanariPinning} and Raghavendra and Tan~\cite{RaghavendraTan}.
It was seized upon in prior work~\cite{ACOPBethe,ACOPBP} to study diluted models of disordered systems such as spin glass models.
The pinning lemma implies that the Ising model on $\GG$ admits a `simple' decomposition into a `small' number of pure states where the conditional Ising distribution enjoys replica symmetry.
Starting from the pinning lemma, ultimately we are going to argue by means of monotonicity and positive correlation arguments that the only pure state decomposition that can be reconciled with the free energy formula \eqref{eqDM} is the one stated in \Prop~\ref{prop_pure}.
Once this is in place we will then show that in the proof of the central limit theorem the pure state decomposition can be used in lieu of the Gibbs uniqueness property.

The proof of \Thm~\ref{thm_an} relies on small subgraph conditioning, a technique originally developed by Robinson and Wormald~\cite{RW2} in order to count Hamilton cycles in random regular graphs.
Janson~\cite{Janson_1993} stated a generalised version of small subgraph conditioning that we shall employ (see \Thm~\ref{thm_janson} below).
Mossel, Neeman and Sly~\cite{Mossel_2015} applied small subgraph conditioning to the assortative stochastic block model, a random graph model closely related to the Ising model on \Erdos-\Renyi\ graphs.
There are two main differences between the present case of the Ising model and the stochastic block model.
First, in the (symmetric) stochastic block model the partition function only runs over balanced partitions, whereas here we need to explicitly truncate the partition function (see \eqref{eqZbal}).
Second, in the stochastic block model edges are present with certain probabilities independently, while here we consider a fixed number of edges.

Beyond the aforementioned articles that use or develop techniques related to the present ones, a substantial number of contributions deal with Ising models on random graphs.
The first to study the Ising model on the \Erdos-\Renyi\ random graph were Bovier and Gayrard~\cite{BovierGayard}.
They considered the case of `dense' random graphs where the average degree $d=2m/n$ diverges as $n\to\infty$.%
	\footnote{Strictly speaking, Bovier and Gayrard deal with the directed version of the \Erdos-\Renyi\ graph.
	The same applies to the work of Kabluchko, L\"owe and Schubert~\cite{KLS1,KLS2,KLS3}.}
In this regime an appropriate rescaling of the Hamiltonian is necessary.
The rescaled partition function reads
\begin{align}\label{eqBG}
	Z'_{\GG}(\beta,B)&=\sum_{\sigma\in\PM^{V_n}}\exp\bc{\frac{\beta}{2d}\sum_{e\in E(\GG)}\sigma(e)+B\sum_{v\in V_n}\sigma(v)}.
\end{align}
Bovier and Gayrard prove that the commensurate free energy, i.e., $\frac1n\log Z'_{\GG}(\beta,B)$, converges in probability to the very same limit as the free energy of the Curie-Weiss Ising model with identical $\beta,B$.
Since the Curie-Weiss Ising model is just the Ising model on the (deterministic) complete graph with the free energy rescaled as in~\eqref{eqBG} by the degree $n-1$, the symmetry properties of the complete graph entail that the limiting free energy is given by a simple analytic expression~\cite{Ellis,Kac}.
A further analogy between the Curie-Weiss model and the Ising model on the dense \Erdos-\Renyi\ graph is that in both cases the empirical distribution of the magnetisations converges to an atom if $B\neq0$ or $\beta<1$ and to a two-point distribution if $B=0$ and $\beta>1$.
Furthermore, the result of Bovier and Gayrard implies that
\begin{align}\label{eqqa}
	\ex[\log Z'_{\GG}(\beta,B)]\sim\log\ex[Z'_{\GG}(\beta,B)].
\end{align}
In disordered systems jargon \eqref{eqqa} posits that the `quenched' free energy matches the `annealed' free energy.

Following up on Bovier and Gayrard, Kabluchko, L\"owe and Schubert~\cite{KLS1,KLS2,KLS3} determined the limiting distribution of the free energy in the dense case $d=d(n)\gg n^{1/3}$ but $n-d(n)\gg1$ and obtained simple explicit expressions for the parameters of this distribution.
Specifically, for $d\gg n^{1/3}$ in the low temperature regime $\beta>1$ they obtain a central limit theorem for $\log Z'_{\GG}(\beta,0)$.
Moreover, in the high temperature regime $0<\beta<1$ with $d\gg n^{1/2}$ they show that $Z'_{\GG}(\beta,0)$ (without the log) converges to a Gaussian.
Additionally, Kabluchko, L\"owe and Schubert worked out the limiting distribution of the magnetisation $\sum_{v\in V_n}\SIGMA(v)$ of a configuration drawn from the Boltzmann distribution.

The case of sparse graphs turned out to be far more intricate than the aforementioned dense case.
This manifests itself, for instance, in the fact that the limiting distribution $\pi_{d,\beta,B}$ of the empirical distribution of the vertex magnetisations is not generally an atom or a two-point distribution~\cite{Dembo_2010}.
In effect, in the sparse case the `quenched-equals-annealed' equality~\eqref{eqqa} is generally false.
Indeed, in the case of finite average degrees $d$ (which we consider in the present paper), it is not difficult to verify that
\begin{align}\label{eqqna}
	\ex[\log Z_{\GG}(\beta,B)]&\leq\log\ex[Z_{\GG}(\beta,B)]-\Omega(n)&&\mbox{ if $B>0$ or $B=0$, $d>1$ and $\beta>\betaf(d)$}.
\end{align}

In light of these intricacies it is all the more remarkable that Dembo and Montanari established the `replica symmetric solution' not only for the sparse \Erdos-\Renyi\ graph but for a general class of graph sequences that converge to unimodular trees in the topology of local weak convergence~\cite{Brasil,Dembo_2010}.
The proofs are based on comparing the fixed points of the Belief Propagation operator on the limiting tree when starting from the free and the all-ones boundary conditions by means of the GHS inequality~\cite{GHS}.
Additionally, Basak and Dembo~\cite{Basak} investigate the local limit of the spin configurations on such graph sequences.
They show that in the case $B=0$ the local limit of the configurations results by initialising the limiting unimodular tree from the all-$1$ and the all-$(-1)$ boundary conditions.
Furthermore, under the assumption that the graph sequence enjoys a certain expansion property, they prove that the local limit of a random configuration with positive overall magnetisation results from the all-$1$ boundary condition only.
While the sparse random graph $\GG$ does not generally satisfy said expansion condition, we will be able to use a partial result~\cite[\Lem~3.3]{Basak} to derive \Prop~\ref{prop_BP_B=0}.
Furthermore, as substantial added insight that \Prop~\ref{prop_pure} on the pure state decomposition contributes by comparison to the local convergence results from~\cite{Basak} is the presence of the conditional replica symmetry condition (\Prop~\ref{prop_pure}~(ii)).
Together with the techniques from~\cite{ACOPBethe} this condition could be extended to show that $\cS_\pm$ are `Bethe states' in the sense of~\cite{Brasil,MM} (details omitted).
Finally, Giardin\`a, Giberti, van der Hofstad and Prioriello~\cite{GGHP} established a central limit theorem for the magnetisation $\sum_{v\in V_n}\SIGMA(v)$ of a random configuration $\SIGMA$ on locally tree-like graph sequences.
Their proof hinges on the free energy formula from~\cite{Dembo_2010}, which is needed to study the cumulant generating functions of the magnetisation.

\Thm s \ref{thm_an} and~\ref{thm_no} leave out the critical case of the Ising model on the \Erdos-\Renyi\ graph with $\beta=\betaf(d)$ and $B=0$.
Very recently Prodromidis and Sly~\cite{ProdromidisSly} proved that in this case the mixing time of Glauber dynamics is polynomial in $n$.
But deriving the precise limiting distribution of the free energy in the critical case remains an interesting open problem.

Apart from the \Erdos-\Renyi\ random graph, the Ising model has also been investigated on the conceptually simpler random regular graph.
Indeed, Galanis, Stefankovic, Vigoda and Yang~\cite{Galanis_2016} derive the limiting distribution of general ferromagnetic spin systems, including the Ising and Potts models, on random regular graphs by means of the small subgraph conditioning method.
They employ this result towards a complexity-theoretic reduction of certain computational counting problems.
In contrast to the sparse \Erdos-\Renyi\ case, in the case of random regular graphs the quenched and annealed free energies coincide and the limit of the empirical distribution of the vertex magnetisations is either an atom or a mixture of two atoms.
Subsequent contributions~\cite{BDS,DuZhou,HJP} derive a more precise picture in the case of random regular graphs and/or graph sequences that converge locally to a $d$-regular tree for the Potts and Ising models.
Additionally, Barbier, Chan and Macris~\cite{BarbierChanMacris} studied multi-overlaps of a general class of ferromagnetic $p$-spin models on random hypergraphs.
Their proof is based on a perturbation of the Hamiltonian that does not shift the limiting free energy to the first order, but that seems to potentially distort the limiting distribution of the free energy.

Classically the Ising model has been studied on lattices such as $\ZZ^d$ for integers $d\geq1$.
Indeed, Ising's thesis~\cite{Ising} establishes the absence of a phase transition in the case $d=1$.
Subsequently Onsager~\cite{Onsager} proved that a phase transition occurs in the case $d=2$.
For the rich literature on the Ising model on lattices, which involves a wealth of mathematical techniques, we refer to~\cite{DC}.

In the present paper we deal with the {\em ferromagnetic} Ising model on random graphs.
Remarkably, the antiferromagnetic version, which results by taking $B=0$ and $\beta<0$ in the definition~\eqref{eqIsingmu}, appears to be far more complex.
Indeed, according to physics predictions the antiferromagnetic Ising on random graphs model behaves like a genuine spin glass model that exhibits `replica symmetry breaking', at least in the low temperature phase~\cite{pnas,MP1,MP2}.
Hence, while \Prop s~\ref{prop_DM_B>0} and~\ref{prop_pure} show that the ferromagnetic Ising model on the random graph either behaves replica symmetrically or that the phase space decomposes into two mutually inverse pure states, in the antiferromagnetic model at low temperature an unbounded number of pure states are expected.
At this point only the existence of a replica symmetry breaking phase transition has been verified rigorously~\cite{CEJJKK,Coja_2018}.

\section{Preliminaries}\label{sec_pre}

\noindent
In this section we summarise a few results from the literature that we are going to build upon.
We begin with basics about the Ising ferromagnet.

\subsection{Ferromagnetic facts}\label{sec_ff}

A key property of the ferromagnetic Ising model is a `positive correlation' between spins.
Roughly speaking, given that $\SIGMA(v_i)=1$ for some vertex $v_i\in V_n$ the conditional probability that $\SIGMA(v_j)=1$ for some other vertex $v_j$ should be at least as large as the unconditional probability of the event $\SIGMA(v_j)=1$.
The precise correlation property that we will use repeatedly follows from a general inequality that is known as the `four functions theorem'.

\begin{lemma}[{Ahlswede-Daykin inequality/four functions theorem~\cite{AhlswedeDaykin}}]\label{lem_ADin}
Let $L$ be a finite distributive lattice and let $f_1,f_2,f_3,f_4: L \to \RR_{\ge 0}$ be functions satisfying
\begin{align}\label{eqlem_ADin1}
      f_1(x)f_2(y) \le f_3(x\vee y)f_4(x\wedge y)\;, && \text{for all } x, y \in L \enspace .
    \end{align}
Then for any $X, Y \subseteq L$ we have
\begin{align}\label{eqlem_ADin2}
    \bc{\sum_{x \in X}f_1(x)} \bc{\sum_{x\in Y}f_2(x)}
      \le
    \bc{\sum_{x \in X\vee Y}f_3(x)} \bc{\sum_{x\in X \wedge Y}f_4(x)}
    \enspace,
  \end{align}
  where $X \vee Y = \{ x \vee y : x \in X, y\in Y\}$ and $X \wedge Y = \{x \wedge y: x \in X, y \in Y\}$.
\end{lemma}

We apply \Lem~\ref{lem_ADin} to the Ising model as follows.
For a set $V$, a subset $U\subseteq V$ and a configuration $\tau\in\PM^U$ let
\begin{align}\label{eqsubcube}
	\cS_{\tau}&=\cbc{\sigma\in\PM^V:\forall u\in U:\sigma(u)=\tau(u)}.
\end{align}
Thus, $\cS_{\tau}$ is the discrete cube consisting of all configurations $\sigma$ that match $\tau$ on $U$.

\begin{lemma}\label{lem_kostas}
	Let $G=(V,E)$ be a graph, $B\geq0$, $U\subset V$ and $v\in V$.
	Moreover, suppose that $\tau,\tau'\in\PM^U$ satisfy $\tau(u)\leq\tau'(u)$ for all $u\in U$.
	Then $\scal{\SIGMA(v)}{\mu_{G,\beta,B}(\nix\mid\cS_{\tau})}\leq\scal{\SIGMA(v)}{\mu_{G,\beta,B}(\nix\mid\cS_{\tau'})}.$
\end{lemma}

\noindent
Hence, the magnetisation of vertex $v$ is an increasing function of the vector $\tau$ that defines the subcube~$\cS_{\tau}$.

\begin{proof}[Proof of \Lem~\ref{lem_kostas}]
	If $v\in U$ then there is nothing to show.
	Hence, we may assume that $v \in V\setminus U$.
	Let $W=V\setminus U$ for brevity.
	We are going to apply \Lem~\ref{lem_ADin} to the lattice $L=\PM^{W}$ with $\wedge$, $\vee$ being the coordinatewise min and max, respectively.
	To this end, given a configuration $\theta \in \PM^{U}$ we define  $\enrg(\cdot\mid \theta): \PM^{W} \to \RR$ by
  \begin{align*}
	  \enrg (\sigma \mid \theta) &=\beta \sum_{uw \in E(G) } \sigma(u)\sigma(w)  \vecone\cbc{u,w \in W} +\sigma(w)\theta(u) \vecone\cbc{u \in U, w \in W }
	  +B\sum_{w\in W}\sigma(w)
  \end{align*}
  Assuming $\tau(u)\leq\tau'(u)$ for all $u\in U$ let
  \begin{align}\label{eq_foufunDef}
    f_1,f_2,f_3,f_4&: \PM^{W} \to \RR_{\ge 0},  && f_1(\sigma)=f_3(\sigma)= \exp\bc{\enrg(\sigma \mid \tau')},
    \quad
    f_2(\sigma)=f_4(\sigma)= \exp\bc{ \enrg(\sigma \mid \tau)}
	\end{align}
  It is easy to verify that the above functions satisfy the hypothesis of \Lem~\ref{lem_ADin}. Indeed, for any $\sigma, \sigma' \in \PM^{W}$ we have
  \begin{align*}
	  \sigma(u)\sigma(w) + \sigma'(u)\sigma'(w) &\le (\sigma(u) \vee \sigma'(u))(\sigma(w) \vee \sigma'(w)) + (\sigma(u) \wedge \sigma'(u))(\sigma(w) \wedge \sigma'(w))  && \mbox{ for all  $u,w \in W$, and} \\
	  \sigma(w)\tau'(u) + \sigma'(w)\tau(u) &\le (\sigma(w) \vee \sigma'(w)) \tau'(u) +(\sigma(w) \wedge \sigma'(w)) \tau(u), && \mbox{ for all } w \in W,\; u \in U.
  \end{align*}
  Summing the above on $u w \in E(G)$ and taking exponentials establishes the condition~\eqref{eqlem_ADin1} from \Lem~\ref{lem_ADin}:
  \begin{align*}
    \exp\bc{ \enrg(\sigma \mid \tau')+ \enrg(\sigma' \mid \tau)} \le
    \exp\bc{ \enrg(\sigma \vee \sigma' \mid \tau')+\enrg(\sigma \wedge \sigma' \mid \tau)}\enspace.
  \end{align*}

  Now let $A \subseteq L$ be any upwards-closed subset of $L$.
  Then $A \vee L = A$ and $A \wedge L = L$.
Therefore, applying~\eqref{eqlem_ADin2} to the functions from \eqref{eq_foufunDef} with $X=A$ and $Y= L$ yields
\begin{align}\label{eq_resffgives}
\bc{\sum_{ \sigma \in \PM^{W}} \exp\bc{ \enrg(\sigma  \mid \tau')}}
\bc{\sum_{ \sigma \in A} \exp\bc{ \enrg(\sigma  \mid \tau)}}
\le
\bc{\sum_{ \sigma \in A} \exp\bc{ \enrg(\sigma  \mid \tau')}}
\bc{\sum_{ \sigma \in \PM^{W}} \exp\bc{ \enrg(\sigma  \mid \tau)}}
\enspace.
\end{align}
Letting $A^\tau = \cbc{\sigma \in \cS_{\tau}: \sigma|_{W} \in A }$ we find
\begin{align*}
	\mu_{G,\beta,B}\bc{A^\tau\mid\cS_{\tau}}=  \frac{\sum_{ \sigma \in A} \exp\bc{ \enrg(\sigma  \mid \tau)}}{\sum_{ \sigma \in \PM^{W}} \exp\bc{ \enrg(\sigma  \mid \tau)}}
  \enspace.
\end{align*}
Therefore, we may rewrite \eqref{eq_resffgives} as
	$$\mu_{G,\beta,0}\bc{A^\tau\mid\cS_\tau} \le \mu_{G,\beta,0}(A^{\tau'}\mid\cS_{\tau'}),$$
i.e., $\mu_{G,\beta,0}\bc{\nix \mid\cS_{\tau'}}$ stochastically dominates $\mu_{G,\beta,0}\bc{\nix \mid\cS_{\tau}}$.
Hence, the assertion follows from the fact that $\sigma\mapsto\sigma(v)$ is an increasing function of $\sigma\in\PM^{W}$.
\end{proof}

Furthermore, we build upon results of Lyons~\cite{LyonsBRN,Lyons} on the Ising model on infinite trees.
An infinite tree $T$ with root $\root$ is {\em locally finite} if every vertex of $T$ has a finite number of neighbours.
Lyons established a threshold for the root-magnetisation to be strictly positive in terms of a parameter of $T$ called the \emph{branching number}.
To introduce this parameter, call a finite set $\Lambda \subseteq V(T)$ of vertices of $T$ a {\em cut set of $T$} if every infinite path starting at $\root$ passes through $\Lambda$.
Then the branching number of $T$ is defined as
  \begin{align}\label{eq_Lyons_BrF}
	  \br(T) =\inf \cbc{\lambda > 0 : \inf_{\Lambda \textrm{ cut set of }T}\sum_{v \in \Lambda}\lambda^{-\dist_T(v, \root)} = 0 }
    \enspace.
  \end{align}

\begin{lemma}[{\cite[Theorem 2.1]{Lyons}}]\label{lem_LYonORg}
	Let $T$ be a locally finite tree rooted at $\root$. Then for all $\beta > \betaf\bc{\br(T)}$ we have
	\begin{align*}
		\liminf_{\ell\to\infty}\scal{\SIGMA(\root)}{\mu_{T, \root,\beta, 0}^{(\ell)}} > 0 \enspace.
	\end{align*}
\end{lemma}

We are ultimately going to apply \Lem~\ref{lem_LYonORg} to the Galton-Watson tree $\TT=\TT(d)$.
To this end we need to get a handle on the branching number of $\TT$.

\begin{lemma}[{\cite[\Cor~5.10]{LyonsPeres}}\label{cl_BrNumOfT}]
	If $d>1$ then $\pr\brk{\br(\TT)=d\mid\TT\mbox{ is infinite}}=1$.
\end{lemma}

\subsection{A martingale central limit theorem}\label{sec_HH_clt}

The following generic theorem provides the starting point for the proofs of \Thm s~\ref{thm_ex} and~\ref{thm_no}.

\begin{theorem}[{\cite[Theorem 3.2]{HH}}]\label{thm_hh_original}
	Let $(\cZ_{n,M}, \cF_{n,M})_{0\leq M \leq m_n, n \geq 1}$ be a zero-mean, square-integrable martingale array with differences $\cX_{n,M} = \cZ_{n,M} - \cZ_{n,M-1}$ for $1 \leq M \leq m_n$. Assume that there exists a constant $\eta^2>0$ such that
	\begin{description}
		\item[CLT1] $\lim_{n \to \infty}\max_{1 \leq M \leq m_n} |\cX_{n,M}| = 0 $ in probability.
		\item[CLT2]	$\limsup_{n\to\infty}\ex[\max_{1 \leq M \leq m_n} \cX_{n,M}^2]<\infty.$
		\item[CLT3]$\lim_{n \to \infty}\sum_{M=1}^{m_n}\cX_{n,M}^2 = \eta^2$ in probability.
	\end{description}
	Then $\cZ_{n,m_n}$ converges in distribution to a Gaussian with mean zero and variance $\eta^2$.
\end{theorem}

As we will see, conditions {\bf CLT1} and~{\bf CLT2} are easily satisfied in the present context of the Ising model on random graphs.
However, establishing {\bf CLT3} turns out to be anything but straightforward.

\subsection{The pinning lemma}\label{sec_extremal}
The `pinning lemma' is an important ingredient to the proof of \Thm~\ref{thm_no} and specifically to the proof of \Prop~\ref{prop_pure}.
The lemma provides that probability measures on sets of the form $\Omega=\PM^V$ for a finite set $V\neq\emptyset$ can generally be written as `canonical' convex combinations of a `small' number of measures that are `close' to product measures.
Of course, we need to clarify what we mean by `canonical', `small' and `close'.

For `close' we refer to the {\em cut metric} $\cutm(\nix,\nix)$ on the space $\cP(\PM^V)$ of probability distributions on $\PM^V$ as introduced in~\cite{ACOPBethe}.
For two probability measures $\mu,\nu\in\cP(\PM^V)$ let $\Gamma(\mu,\nu)$ be the set of all couplings of $\mu,\nu$.
Thus, $\Gamma(\mu,\nu)$ contains all probability measures $\gamma\in\cP(\PM^V\times\PM^V)$ such that for all $\sigma\in\PM^V$ we have
\begin{align*}
	\mu(\sigma)&=\sum_{\tau\in\PM^V}\gamma(\sigma,\tau)&&\mbox{and}&\nu(\sigma)&=\sum_{\tau\in\PM^V}\gamma(\tau,\sigma).
\end{align*}
Then the cut distance of $\mu,\nu\in\cP(\PM^V)$ is defined as
\begin{align}\label{eqcutm}
	\cutm(\mu,\nu)&=\frac1{|V|}\min_{\gamma\in\Gamma(\mu,\nu)}\max_{\substack{U\subset V\\\cB\subset\Omega^V\times\Omega^V}}\abs{\sum_{(\sigma,\tau)\in\cB}\sum_{u\in U}\gamma(\sigma,\tau)(\sigma(u)-\tau(u))}
\end{align}
Intuitively, $\cutm(\mu,\nu)$ gauges the largest remaining `discrepancy' between $\mu$ and $\nu$ under the best possible coupling.
The cut metric is {\em much} weaker that other metrics such as the total variation distance.
It can be verified that $\cutm(\nix,\nix)$ lives up to its name and is indeed a metric on $\cP(\PM^V)$~\cite{ACOPBethe}.
The definition~\eqref{eqcutm} of the cut metric is inspired by the cut metric from the theory of graphons~\cite{Lovasz}, which in turn generalises the work of Frieze and Kannan~\cite{FK}.
Moreover, both the definition~\eqref{eqcutm} and the notion of graphons are closely related to the Aldous-Hoover representation~\cite{Aldous,Hoover}, on which Panchenko's concept of an asymptotic Gibbs measure is based~\cite{Panchenko_2013_3}.

For a probability measure $\mu\in\cP(\PM^V)$ define
\begin{align*}
	\bar\mu(\sigma)&=\prod_{v\in V}\mu(\{\SIGMA(v)=\sigma(v)\})&&(\sigma\in\PM^V).
\end{align*}
In words, $\bar\mu\in\cP(\PM^V)$ is the product measure on $\PM^V$ with the same marginal probabilities $\bar\mu(\{\SIGMA(v)=1\})=\mu(\{\SIGMA(v)=1\})$ as $\mu$ itself.
Following \cite{Coja_2019}, we say that $\mu$ is {\em $\eps$-extremal} if $\cutm(\mu,\bar\mu)<\eps$.
Furthermore, we say that $\mu$ is {\em $\eps$-symmetric} if
\begin{align}\label{eqdefsym}
	\frac1{|V|^2}\sum_{v,w\in V}\sum_{s,s'\in\PM}\abs{\mu(\{\SIGMA(v)=s,\,\SIGMA(w)=s'\})-\mu(\{\SIGMA(v)=s\})\mu(\{\SIGMA(w)=s'\})}&<\eps.
\end{align}
In words, an $\eps$-symmetric measure $\mu$ has the property that for `most' pairs $v,w$ the spins $\SIGMA(v),\SIGMA(w)$ are approximately independent.
Clearly, this concept is closely related to the notion of replica symmetry; cf.~\eqref{eqRSij}.
The following close relationship between extremality and symmetry was established in~\cite{ACOPBethe}.

\begin{lemma}\label{lem_extremal_symmetric}
	For any $\eps>0$ there exist $\delta>0$ and $N>0$ such that for all sets $V$ of size $|V|>N$ and all $\mu\in\cP(\PM^V)$ the following is true.
	\begin{enumerate}[(i)]
		\item if $\mu$ is $\delta$-extremal, then $\mu$ is $\eps$-symmetric.
		\item if $\mu$ is $\delta$-symmetric, then $\mu$ is $\eps$-extremal.
	\end{enumerate}
\end{lemma}

\noindent
In addition, we are going to need the following two inequalities from~\cite{Coja_2019}.

\begin{lemma}[{\cite[\Lem s~3.12 and~3.13]{Coja_2019}}]\label{lem_extremal_margs}
	For any $\mu,\nu\in\cP(\PM^V)$ we have
	\begin{align}\label{eq_cutm_mag}
		\frac1{|V|}\sum_{v\in V}\abs{\scal{\SIGMA(v)}\mu-\scal{\SIGMA(v)}\nu}&\leq8\cutm(\mu,\nu).
	\end{align}
	Moreover, if $\mu,\nu\in\cP(\PM^V)$ are $\eps$-extremal, then
	\begin{align}\label{eqmargsclose}
		\cutm(\mu,\nu)\leq2\eps+\frac1{|V|}\sum_{v\in V}\abs{\scal{\SIGMA(v)}\mu-\scal{\SIGMA(v)}\nu}.
	\end{align}
\end{lemma}

\noindent
Thus, in order to prove that two $\eps$-extremal measures are close in the cut metric, we just need to bound the absolute distances of their magnetisations.
Furthermore, the following lemma provides an enhanced triangle inequality for the cut metric.

\begin{lemma}[{\cite[\Lem~3.14]{Coja_2019}}]\label{lem_xtriangle}
	Let $\mu^{(1)},\nu^{(1)},\ldots,\mu^{(\ell)},\nu^{(\ell)}\in\cP(\Omega^V)$ and let $p_1,\ldots,p_\ell\geq0$ be such that $\sum_{i=1}^\ell p_i=1$.
	Then
	\begin{align*}
		\cutm\bc{\sum_{i=1}^\ell p_i\mu^{(i)},\sum_{i=1}^\ell p_i\nu^{(i)}}\leq\sum_{i=1}^\ell p_i\cutm(\mu^{(i)},\nu^{(i)}).
	\end{align*}
\end{lemma}

The most important fact in this section, and a cornerstone of the proof of \Prop~\ref{prop_pure}, is the following `pinning lemma'.
For a set $V$, a subset $U\subseteq V$ and $\tau\in\PM^U$ we let $\cS_\tau$ be the subcube from~\eqref{eqsubcube}.
Think of $\cS_\tau$ as the set of configurations $\sigma$ where the `pin' each coordinate $u\in U$ to $\tau(u)$.

\begin{lemma}[{\cite[\Lem~3.5]{Coja_2018}}]\label{prop_pinning}
	For any $\delta>0$ there exists a bounded random variable $\vI\geq0$ such that for all finite sets $V$ and all probability distributions $\mu\in\cP(\PM^V)$ the following is true.
	Let $\vU\subseteq V$ be a uniformly random subset of size $|\vU|=\min\{\vI,|V|\}$.
	Then
	\begin{align*}
		\pr\brk{\sum_{\tau\in\PM^{\vU}}\mu(\cS_{\tau})\vecone\cbc{\mu(\nix\mid\cS_{\tau})\mbox{ is $\delta$-symmetric\,}}>1-\delta}&>1-\delta.
	\end{align*}
\end{lemma}

\noindent
Thus, if we choose a suitable random set $\vU\subset V$ of coordinates, then there is a high probability that under the decomposition $(\cS_\tau)_{\tau\in\PM^{\vU}}$ of the cube $\PM^V$ into subcubes almost the entire mass of $\mu$ is placed on subcubes where the conditional distribution $\mu(\nix\mid\cS_\tau)$ is $\delta$-symmetric.
The size of the set $\vU$ is determined by the random variable $\vI$ and given its size, $\vU$ is just a uniformly random subset of $V$.
Crucially, the random variable $\vI\geq0$ is bounded and depends on $\eps$ only.
Hence, the number of coordinates that get `pinned' is independent of the size of $V$ and of the measure $\mu$.

Since the conditional distributions $\mu(\nix\mid\cS_\tau)$ are likely mostly $\delta$-symmetric, they bear a conceptual semblance of the `pure states' of a physical system as set out in~\cite{Parisi}.
The `pure state decomposition' provided by \Lem~\ref{prop_pinning} is indeed `small' because the random variable $\vI$ is bounded independently of $V$ or $\mu$.
Furthermore, it is `canonical' to the extent that the `pure states' $\cS_\tau$ are simply subcubes of the discrete cube $\PM^V$.
Both these facts will play a key role in the proof of \Prop~\ref{prop_pure}.

Finally, we are going to have to investigate how local changes such as the addition of a few edges affect the Ising distribution.
An important tool to this end is the following notion.
Let $\mu,\nu\in\cP(\PM^V)$ be probability distributions.
Following~\cite{Coja_2019} we say that $\nu$ is {\em $c$-contiguous} w.r.t.\ $\mu$ if $\nu(\sigma)\leq c\mu(\sigma)$ for all $\sigma\in\PM^V$.
In addition, $\mu$ and $\nu$ are {\em mutually $c$-contiguous} if $\nu(\sigma)\leq c\mu(\sigma)$ and $\mu(\sigma)\leq c\nu(\sigma)$ for all $\sigma\in\PM^V$.

\begin{lemma}[{\cite[\Lem~3.17]{Coja_2019}}]\label{lem_contig}
	For any $c>0$, $\eps>0$ there exist $\delta>0$, $N>0$ such that for all sets $V$ of size $|V|\geq N$ and all $\mu,\nu\in\cP(\PM^V)$ the following is true.
	If $\mu$ is $\delta$-extremal and $\nu$ is $c$-contiguous with respect to $\mu$, then $\nu$ is $\eps$-extremal and $\cutm(\mu,\nu)<\eps$.
\end{lemma}

\subsection{Small subgraph conditioning}\label{sec_ssc}
The proof of \Thm~\ref{thm_an} is based on the small subgraph conditioning technique for which the following theorem provides the general framework.

\begin{theorem}[{\cite[\Thm~1]{Janson_1993}}]\label{thm_janson}
	Let $(\kappa_i)_{i\geq1}$ and $(\xi_i)_{i\geq1}$ be sequences such that $\kappa_i>0$ and $\xi_i>-1$ for all $i$.
	Also let $(\cC_{i,n})_{i,n\geq1}$ and $(\cZ_n)_{n\geq1}$ be random variables such that $\cC_{i,n}\in\ZZpos$ and $\ex[\cZ_n]\neq0$.
	Moreover, let $(\cY_i)_{i\geq1}$ be independent Poisson variables with $\ex[\cY_i]=\kappa_i$.
	Assume that the following conditions hold.
	\begin{description}
		\item[SSC0] $\sum_{i\geq1}\kappa_i\xi_i^2<\infty$.
		\item[SSC1] For all $\ell>0$ and $c_1,\ldots,c_\ell\geq0$ we have $\lim_{n\to\infty}\pr\brk{\forall1\leq i\leq\ell:\cC_{i,n}=c_i}=\pr\brk{\forall 1\leq i\leq\ell:\cY_i=c_i}.$
		\item[SSC2] For all $\ell>0$ and $c_1,\ldots,c_\ell\geq0$ we have
			\begin{align*}
				\lim_{n\to\infty}\frac{\ex\brk{\cZ_n\mid\forall 1\leq i\leq\ell:\cC_{i,n}=c_i}}{\ex[\cZ_n]}&=\prod_{i=1}^\ell(1+\xi_i)^{c_i}\exp(-\kappa_i\xi_i).
			\end{align*}
		\item[SSC3] $\limsup_{n\to\infty}\ex\brk{\cZ_n^2}/\ex\brk{\cZ_n}^2\leq\exp\sum_{i\geq1}\kappa_i\xi_i^2$.
	\end{description}
	Then the infinite product
	\begin{align}\label{eqJansonW}
		\cW&=\prod_{i=1}^\infty(1+\xi_i)^{\cY_i}\exp(-\kappa_i\xi_i)
	\end{align}
	converges almost surely and in $L^2(\RR)$ and $\ex[\cW]=1$, $\ex[\cW^2]=\exp\bc{\sum_{i\geq1}\kappa_i\xi_i^2}.$
	Furthermore,
	\begin{align*}
		\lim_{n\to\infty}\frac{\cZ_n}{\ex[\cZ_n]}&=\cW&&\mbox{in distribution}.
	\end{align*}
\end{theorem}

\section{Proof of \Thm~\ref{thm_ex}}\label{sec_thm_ex}

\noindent
We prove the various statements from \Sec s~\ref{sec_BP_intro} and~\ref{sec_B>0} towards the proof of \Thm~\ref{thm_ex}.
Some of these statements, specifically \Lem~\ref{lem_BP_tree}, \Prop~\ref{proposition_HH} and \Lem~\ref{fact_Zratio} will also play a role in the proof of \Thm~\ref{thm_no}.

\subsection{Proof of \Lem~\ref{lem_BP_tree}}\label{sec_lem_BP_tree}

\noindent
The lemma can be deduced from general facts about Belief Propagation~\cite{MM}, but we provide a self-contained proof for the sake of completeness.
Recalling the conditional distribution $\mu_{T,\beta,B}^{(\ell)}$ associated with a tree $T$ rooted at $\root$ from~\eqref{eqMuT}, we let
\begin{align}\label{eqZTs}
	Z_{T,s}^{(\ell)}(\beta,B)&=\sum_{\tau\in\PM^{V(T)}}\vecone\cbc{\forall{v\in\partial^\ell(T,\root)}:\tau(v)=1,\,\tau(\root)=s}\exp\bc{\beta\sum_{vw\in E(T)}\tau_v\tau_w+B\sum_{v\in V(T)}\tau_v}&&(s=\PM)
\end{align}
be the contribution to the partition function of $\mu_{T,\beta,B}^{(\ell)}$ from configurations with root spin $\tau(\root)=s$.
Further, let
\begin{align*}
	Z_{T}^{(\ell)}(\beta,B)=Z_{T,-1}^{(\ell)}(\beta,B)+Z_{T,1}^{(\ell)}(\beta,B).
\end{align*}
Then the definition~\eqref{eqMuT} of $\mu_{T,\beta,B}^{(\ell)}$ implies that
\begin{align}\label{eq_lem_BP_tree_1}
	\scal{\SIGMA(\root)}{\mu_{T,\beta,B}^{(\ell)}}&=\sum_{s\in\PM}\frac{sZ_{T,s}^{(\ell)}(\beta,B)}{Z_{T,-1}^{(\ell)}(\beta,B)+Z_{T,1}^{(\ell)}(\beta,B)}.
\end{align}
Rearranging the sum~\eqref{eqZTs}, we see that
\begin{align}\label{eq_lem_BP_tree_2}
	Z_{T,s}^{(\ell)}(\beta,B)&=\eul^{Bs}\prod_{v\in\partial(T,\root)}\eul^{\beta}Z_{T_v,s}^{(\ell-1)}(\beta,B)+\eul^{-\beta}Z_{T_v,-s}^{(\ell-1)}(\beta,B).
\end{align}
Combining~\eqref{eq_lem_BP_tree_1}--\eqref{eq_lem_BP_tree_2}, we obtain
\begin{align*}
	\scal{\SIGMA(\root)}{\mu_{T,\beta,B}^{(\ell)}}&=\frac{\sum_{s\in\PM}s\eul^{Bs}\prod_{v\in\partial(T,\root)}\eul^{\beta}Z_{T_v,s}^{(\ell-1)}(\beta,B)+\eul^{-\beta}Z_{T_v,-s}^{(\ell-1)}(\beta,B)}{\sum_{s\in\PM}\eul^{Bs}\prod_{v\in\partial(T,\root)}\eul^{\beta}Z_{T_v,s}^{(\ell-1)}(\beta,B)+\eul^{-\beta}Z_{T_v,-s}^{(\ell-1)}(\beta,B)}\\
												  &=
												  \frac{\sum_{s\in\PM}s\eul^{Bs}\prod_{v\in\partial(T,\root)}\eul^{\beta}\mu_{T_v,\beta,B}^{(\ell-1)}(\{\SIGMA(v)=s\})+\eul^{-\beta}\mu_{T_v,\beta,B}^{(\ell-1)}(\{\SIGMA(v)=-s\})}{\sum_{s\in\PM}\eul^{Bs}\prod_{v\in\partial(T,\root)}\eul^{\beta}\mu_{T_v,\beta,B}^{(\ell-1)}(\{\SIGMA(v)=s\})+\eul^{-\beta}\mu_{T_v,\beta,B}^{(\ell-1)}(\{\SIGMA(v)=-s\})}\\
												  &=
		\frac{\sum_{s\in\PM}s\eul^{sB}\prod_{v\in\partial(T,\root)}\bc{1+s\scal{\SIGMA(v)}{\mu_{T_v,\beta,B}^{(\ell-1)}}\tanh\beta}}{\sum_{s\in\PM}\eul^{sB}\prod_{v\in\partial(T,\root)}\bc{1+s\scal{\SIGMA(v)}{\mu_{T_v,\beta,B}^{(\ell-1)}}\tanh\beta}}\enspace,
\end{align*}
as claimed.

\subsection{Proof of \Prop~\ref{proposition_HH}}\label{sec_proposition_HH}

\noindent
We derive \Prop~\ref{proposition_HH} from \Thm~\ref{thm_hh_original}.
Let us start with the following fact.

\begin{fact}\label{fact_adde}
	Let $\beta,B\geq0$.
	For any graph $G$ and any edge $e\in E(G)$ we have
	\begin{align*}
		\frac{Z_G(\beta,B)}{Z_{G-e}(\beta,B)}&=\scal{\exp(\beta\SIGMA(e))}{\mu_{G-e,\beta,B}}\in[\eul^{-\beta},\eul^{\beta}].
	\end{align*}
\end{fact}
\begin{proof}
	By the definition of the Ising model,
	\begin{align}\nonumber
		\frac{Z_G(\beta,B)}{Z_{G-e}(\beta,B)}&=\sum_{\sigma\in\PM^{V(G)}}\exp\bc{B\sum_{v\in V(G)}\sigma(v)+\beta\sum_{vw\in E(G)\setminus\cbc e}\sigma(v)\sigma(w)}\frac{\exp(\beta\sigma(e))}{Z_{G-e}(\beta,B)}\\
											 &=\sum_{\sigma\in \PM^{V(G)}}\mu_{G-e,\beta,B}(\sigma)\exp(\beta\sigma(e)))=\scal{\exp(\beta\SIGMA(e))}{\mu_{\GG-e,\beta,B}},\nonumber
	\end{align}
	as claimed.
\end{proof}

\begin{proof}[Proof of \Prop~\ref{proposition_HH}]
	Let $m=m_n\sim dn/2$ be the number of edges of $\GG$.
	Assuming~\eqref{eq_fact_HH1}, we just need to check that the random variables $\vZ_{n,M}$ and $\vX_{n,M}$ from~\eqref{eqZX} together with the $\sigma$-algebras $\fF_{n,M}$ satisfy conditions {\bf CLT1} and {\bf CLT2} from \Thm~\ref{thm_hh_original}.
	We remind ourselves that $\fF_{n,M}$ is the $\sigma$-algebra generated by the first $M$ edges $\ve_1,\ldots,\ve_M$ of $\GG$.
	Hence, for any $1\leq M\leq m=m_n$ we have
	\begin{align}\nonumber
		|\vX_{n,M}|&=m^{-1/2}\abs{\ex[\log Z_{\GG}(\beta,B)\mid\fF_{n,M}]-\ex[\log Z_{\GG}(\beta,B)\mid\fF_{n,M-1}]}\\
				   &\leq m^{-1/2}\max_{G,\,e\in E(G),\,e'\not\in E(G)}\abs{\log Z_{G}(\beta,B)-\log Z_{G-e+e'}(\beta,B)}\enspace;\label{eq_proposition_HH_1}
	\end{align}
	to be precise, the maximum is taken over all possible outcomes $G$ of $\GG$, all edges $e\in E(G)$ and all non-edges $e'\not\in E(G)$.
	In other words, $G-e+e'$ is obtained from $G$ by removing edge $e\in E(G)$ and inserting $e'$ instead.
	Hence, Fact~\ref{fact_adde} shows that
	\begin{align}
		\frac{Z_G(\beta,B)}{Z_{G-e}(\beta,B)},	\frac{Z_{G-e+e'}(\beta,B)}{Z_{G-e}(\beta,B)}\in[\eul^{-\beta},\eul^{\beta}].\label{eq_proposition_HH_2}
	\end{align}
	Finally, {\bf CLT1} and~{\bf CLT2} follow from \eqref{eq_proposition_HH_1}--\eqref{eq_proposition_HH_2}, while  {\bf CLT3} follows directly from assumption~\eqref{eq_fact_HH1}.
\end{proof}

\subsection{Proof of \Lem~\ref{fact_telescope}}\label{sec_fact_telescope}
By the definition \eqref{eqZX} of $\vX_{n,M}$ we have
\begin{align}\nonumber
	m\vX_{n,M}^2&=\ex\brk{\log Z_{\GG}(\beta,B)\mid\fF_{n,M}}^2+\ex\brk{\log Z_{\GG}(\beta,B)\mid\fF_{n,M-1}}^2\\&\qquad-2\ex\brk{\log Z_{\GG}(\beta,B)\mid\fF_{n,M}}\ex\brk{\log Z_{\GG}(\beta,B)\mid\fF_{n,M-1}}.\label{eq_fact_telescope_1}
\end{align}
Moreover, writing out the first expression from~\eqref{eqDeltaM}, we obtain
\begin{align}\nonumber
	\ex\brk{\vDelta_{n,M}^+\mid\fF_{n,M}}&=\ex\brk{\prod_{h=1}^2\log Z_{\GG^+_{h,M}}(\beta,B)-\log Z_{\GG_{h,M}}(\beta,B)\mid\fF_{n,M}}\\
										 &=\ex\brk{\prod_{h=1}^2\log Z_{\GG^+_{h,M}}(\beta,B)\mid\fF_{n,M}}+\ex\brk{\prod_{h=1}^2\log Z_{\GG_{h,M}}(\beta,B)\mid\fF_{n,M}}\nonumber\\
										 &\qquad\qquad\qquad\qquad\ \qquad\qquad\qquad-2\ex\brk{\log Z_{\GG^+_{1,M}}(\beta,B)\log Z_{\GG_{2,M}}(\beta,B)\mid\fF_{n,M}}.
	\label{eq_fact_telescope_2}
\end{align}
The coupling {\bf CPL0}--{\bf CPL2} ensures that the two random graphs $(\GG_{h,M}^+)_{h=1,2}$ share the $M$ random edges $\ve_1,\ldots,\ve_M$ that are $\fF_{n,M}$-measurable, while the remaining edges that constitute $(\GG_{h,M}^+)_{h=1,2}$ are chosen independently given $\fF_{n,M}$.
Therefore,
\begin{align}\label{eq_fact_telescope_3}
	\ex\brk{\prod_{h=1}^2\log Z_{\GG^+_{h,M}}(\beta,B)\mid\fF_{n,M}}&=
	\prod_{h=1,2}\ex\brk{\log Z_{\GG^+_{h,M}}(\beta,B)\mid\fF_{n,M}}
	=\ex\brk{\log Z_{\GG}(\beta,B)\mid\fF_{n,M}}^2.
\end{align}
Combining~\eqref{eq_fact_telescope_2}--\eqref{eq_fact_telescope_3}, we obtain
\begin{align}\nonumber
	\ex\brk{\vDelta_{n,M}^+\mid\fF_{n,M}}&=\ex\brk{\log Z_{\GG}(\beta,B)\mid\fF_{n,M}}^2+\ex\brk{\prod_{h=1}^2\log Z_{\GG_{h,M}}(\beta,B)\mid\fF_{n,M}}\\&\qquad\qquad\qquad\qquad\qquad\qquad-2\ex\brk{\log Z_{\GG^+_{1,M}}(\beta,B)\log Z_{\GG_{2,M}}(\beta,B)\mid\fF_{n,M}}.\label{eq_fact_telescope_4}
\end{align}
Applying similar reasoning to $\vDelta_{n,M}^-,\vDelta_{n,M}^\pm$, we find
\begin{align}\nonumber
	\ex\brk{\vDelta_{n,M}^-\mid\fF_{n,M}}&=\ex\brk{\log Z_{\GG}(\beta,B)\mid\fF_{n,M-1}}^2+\ex\brk{\prod_{h=1}^2\log Z_{\GG_{h,M}}(\beta,B)\mid\fF_{n,M}}\\&\qquad-2\ex\brk{\log Z_{\GG^-_{1,M}}(\beta,B)\log Z_{\GG_{2,M}}(\beta,B)\mid\fF_{n,M}},\label{eq_fact_telescope_5}\\
	\ex\brk{\vDelta_{n,M}^\pm\mid\fF_{n,M}}&=\ex\brk{\log Z_{\GG}(\beta,B)\mid\fF_{n,M}}\ex\brk{\log Z_{\GG}(\beta,B)\mid\fF_{n,M-1}}+\ex\brk{\prod_{h=1}^2\log Z_{\GG_{h,M}}(\beta,B)\mid\fF_{n,M}}\nonumber\\
										   &\qquad-\ex\brk{\log Z_{\GG^+_{1,M}}(\beta,B)\log Z_{\GG_{2,M}}(\beta,B)\mid\fF_{n,M}}-\ex\brk{\log Z_{\GG^-_{1,M}}(\beta,B)\log Z_{\GG_{2,M}}(\beta,B)\mid\fF_{n,M}}.\label{eq_fact_telescope_6}
\end{align}
Finally, combining~\eqref{eq_fact_telescope_1} with~\eqref{eq_fact_telescope_4}--\eqref{eq_fact_telescope_6}, we obtain the assertion.

\subsection{Proof of \Lem~\ref{fact_Zratio}}\label{sec_fact_Zratio}

\noindent
Step {\bf CPL1} of the coupling ensures that $\GG^-_{h,M}$ is obtained from $\GG_{h,M}$ simply by adding edge $\ve_{h,m-M+1}$.
Therefore, the second equality in~\eqref{eqfact_Zratio} follows from Fact~\ref{fact_adde}.
To obtain the first equality, we take a closer look at {\bf CPL2}, which ensures that $\GG_{h,M}=\GG_{h,M}^+-\ve_M$ unless $\ve_M\in\{\ve_{h,1},\ldots,\ve_{h,m-M}\}$.
But the event $\ve_M\in\{\ve_{h,1},\ldots,\ve_{h,m-M}\}$ has probability $O(1/n)$ because $M\leq m=O(n)$.
Hence, Fact~\ref{fact_adde} implies that the first identity in~\eqref{eqfact_Zratio} holds with probability $1-O(1/n)$.

\subsection{Proof of \Prop~\ref{prop_fix_ex}}\label{sec_prop_fix_ex}

We begin by verifying that the sequence $(\Scal{\SIGMA(\root)}{\mu_{\TT,\beta,B}^{(\ell)}})_\ell$ of root marginalisations converges almost surely as $\ell\to\infty$ for all $B\geq0$.
For $B>0$ this was already established in~\cite[{\Sec~4}]{Dembo_2010}, and for $B=0$ convergence is implied by~\cite[\Lem~1.18]{Basak}.
But toward the proof of \Prop~\ref{prop_fix_ex} we need to know that the convergence of the root marginalisations occurs with respect to a specific filtration on the space of all Galton-Watson trees.

To be precise, let $\fT$ be the space of all possible outcomes of the Galton-Watson tree $\TT$.
Moreover, let $\fA_\ell$ be the $\sigma$-algebra on $\fT$ generated by the events $\cbc{\TT^{(\ell)}\ism T}$ for all finite rooted trees $T$.
Then $(\fA_\ell)_\ell$ is a filtration and the random variable
\begin{align}\label{eqRellT}
	\vec R_\ell(\TT)&=\scal{\SIGMA(\root)}{\mu_{\TT,\beta,B}^{(\ell)}}
\end{align}
is $\fA_\ell$-measurable.

\begin{lemma}\label{lem_conv}
	For all $B\geq0$ the sequence $(\vec R_\ell(\TT))_{\ell}$ is a supermartingale with respect to the filtration $(\fA_\ell)_\ell$, and $0\leq \vec R_\ell(\TT)\leq1$ for all $\ell$.
\end{lemma}
\begin{proof}
\Lem~\ref{lem_kostas} implies that $\vec R_\ell(\TT)\geq\vec R_{\ell+1}(\TT)\geq0$, because the Boltzmann distribution $\mu_{\TT,\beta,B}^{(\ell)}$ imposes a $+1$ boundary condition on all vertices at distance precisely $\ell$ from $\root$.
\end{proof}

Hence, the martingale convergence theorem shows that $(\vec R_\ell(\TT))_\ell$ converges to a random variable
\begin{align}\label{eqRT}
		\vec R(\TT)=\lim_{\ell\to\infty}\vec R_\ell(\TT)\in[0,1]&&\mbox{almost surely}.
	\end{align}
This now easily implies convergence of the measures $(\pi_{d,\beta,B}^{(\ell)})_\ell$ from~\eqref{eqpidbetaBell}.

\begin{corollary}\label{cor_conv}
	For all $B\geq0$ the sequence $(\pi_{d,\beta,B}^{(\ell)})_{\ell}$ converges in $W_1$ to a distribution $\pi_{d,\beta,B}$ supported on $[0,1]$.
\end{corollary}
\begin{proof}
	\Cor~\ref{cor_BP_tree} demonstrates that $(\pi_{d,\beta,B}^{(\ell)})_\ell$ $W_1$-converges to the law of $\vec R(\TT)$, which is supported on $[0,1]$ by \Lem~\ref{lem_conv}.
\end{proof}

As a next step we point out the following combinatorial interpretation of the probability distribution
\begin{align*}
	\pi_{d,\beta,B,t}^{\tensor\,(\ell)}&=\BP^{\tensor\,\ell}_{d,\beta,B,t}(\delta_{(1,1)})
\end{align*}
obtained by $\ell$-fold iteration of the operator from~\eqref{eqBPtensor}.

\begin{lemma}\label{lem_tensor_ell}
	Let $B\geq0$ and $t\in[0,1]$.
	Then $\pi_{d,\beta,B,t}^{\tensor\,(\ell)}$ is the distribution of the random pair
	$$\bc{\scal{\SIGMA(\root)}{\mu^{(\ell)}_{\TT^\tensor[h],\beta,B}}}_{h=1,2}.$$
\end{lemma}
\begin{proof}
	In analogy to \Cor~\ref{cor_BP_tree}, this follows from \Lem~\ref{lem_BP_tree}, the definition \eqref{eqBPtensor}--\eqref{eqBPtensorRec} of $\BP^\tensor_{d,\beta,B,t}$ and the construction of the pair $(\TT^\tensor[h])_{h=1,2}$ of Galton-Watson trees.
\end{proof}

\begin{corollary}\label{cor_tensor_ell}
	Let $B\geq0$, $t\in[0,1]$.
	Then the $W_1$-limit $\pi_{d,\beta,B,t}^\tensor=\lim_{\ell\to\infty}\pi_{d,\beta,B,t}^{\tensor\,(\ell)}$ exists and is supported on $[0,1]^2$.
\end{corollary}
\begin{proof}
	Let $\fT^\tensor$ be the space of all possible outcomes of the multi-type Galton-Watson tree $\TT^\tensor$ and let
	\begin{align}\label{eq_cor_tensor_ell_1}
		\vec R_\ell^\tensor(\TT^\tensor,h)&=\scal{\SIGMA(\root)}{\mu^{(\ell)}_{\TT^\tensor[h],\beta,B}}.
	\end{align}
	Since each $\TT^\tensor[h]$ individually has the same distribution as the plain Galton-Watson tree $\TT$, \Lem~\ref{lem_conv} shows that $(\vec R^\tensor_\ell(\TT^\tensor,h))_\ell$ converges almost surely to a random variable $\vec R^\tensor(\TT^\tensor,h)\in[0,1]$ for $h=1,2$.
	Therefore, by the triangle inequality the pairs $(\vec R_\ell^\tensor(\TT^\tensor,1),\vec R^\tensor_\ell(\TT^\tensor,2))_\ell$ converge to $(\vec R^\tensor(\TT^\tensor,1),\vec R^{\tensor}(\TT^\tensor,2))\in[0,1]^2$ on $\fT^\tensor$ almost surely.
	Consequently, $(\pi_{d,\beta,B}^{\tensor\,(\ell)})_{\ell}$ converges to the law of $(\vec R^\tensor(\TT^\tensor,1),\vec R^\tensor(\TT^\tensor,2))$ with respect to the $W_1$-metric.
\end{proof}

\begin{proof}[Proof of \Prop~\ref{prop_fix_ex}]
	The existence of the limit $\pi_{d,\beta,B,t}^\tensor$ follows directly from \Cor~\ref{cor_tensor_ell}.
	To prove the continuity statement, fix $\eps>0$.
	Let $\vec R_\ell^\tensor(\TT^\tensor,h)$ be the random variable from~\eqref{eq_cor_tensor_ell_1}, and let $\vec R^\tensor(\TT^\tensor,h)$ be the almost sure and thus $L^1$-limit of the sequence $(\vec R_\ell^\tensor(\TT^\tensor,h))_\ell$.
	Recalling the random variables $\vec R_\ell(\TT)$ from~\eqref{eqRellT} and their almost sure limit $\vec R(\TT)$, for $\ell>0$ we let
	\begin{align*}
		\fE(\eps,\ell)&=\cbc{T\in\fT:|\vec R_\ell(T)-\vec R(T)|>\eps}\subset\fT.
	\end{align*}
	Then there exists $\ell_0=\ell_0(\eps)>0$ such that for all $\ell>\ell_0$ we have $\pr\brk{\TT\in\fE(\eps,\ell)}<\eps$.
	Since the individual trees $\TT^\tensor[h]$ for $h=1,2$ have the same distribution as $\TT$ for all $t\in[0,1]$, we conclude that for all $\ell>\ell_0$,
	\begin{align}\nonumber
		\pr\brk{\abs{\vec R_\ell(\TT^\tensor,1)-\vec R(\TT^\tensor,1)}+\abs{\vec R_\ell(\TT^\tensor,2)-\vec R(\TT^\tensor,2)}>2\eps}&\leq\pr\brk{\TT^\tensor[1]\in\fE(\eps,\ell)}+\pr\brk{\TT^\tensor[2]\in\fE(\eps,\ell)}\\&=2\pr\brk{\TT\in\fE(\eps,\ell)}\label{eq_prop_fix_ex_1}
		<2\eps.
	\end{align}

	Now, pick $\delta=\delta(\eps,\ell)>0$ small enough and suppose that $t',t''\in[0,1]$ satisfy $|t'-t''|<\delta$.
	Then for either value $t\in\{t',t''\}$ the definition of the multi-type Galton-Watson tree $\TT^\tensor$ provides that a type~0 vertex spawns $\Po(d(1-t))$ offspring of type~0, and $\Po(dt)$ offspring of type~1 and type~2.
	(The offspring distributions of type~1 and type~2 vertices are independent of $t$.)
	Hence, provided that $|t'-t''|<\delta$ for a sufficiently small $\delta=\delta(\eps,\ell)>0$, we can couple these Poisson variables such that
	\begin{align}\label{eq_prop_fix_ex_2}
		\pr\brk{\TT^{\tensor\,(\ell)}(d,t')\neq\TT^{\tensor\,(\ell)}(d,t'')}<\eps.
	\end{align}
	Combining~\eqref{eq_prop_fix_ex_1} and~\eqref{eq_prop_fix_ex_2} and applying the triangle inequality, we obtain
	\begin{align*}
		W_1(\pi_{d,\beta,B,t'}^\tensor,\pi_{d,\beta,B,t''}^\tensor)&\leq4\eps+\pr\brk{\TT^{\tensor\,(\ell)}(d,t')\neq\TT^{\tensor\,(\ell)}(d,t'')}
		+2\pr\brk{\TT\in\fE(\eps,\ell)}<7\eps,
	\end{align*}
	whence $t\mapsto\pi_{d,\beta,B,t}^\tensor$ is $W_1$-continuous.
\end{proof}

\subsection{Proof of \Lem~\ref{lem_lwc_tensor}}\label{sec_lem_lwc_tensor}

The key step towards the proof is to establish the following statement.

\begin{lemma}\label{lem_lwc_tensor_ex}
	Suppose that $t\in[0,1]$, $\ell\geq0$ and $M\sim tm$.
	Then for any rooted tree $T$ we have
	\begin{align*}
		\pr\brk{v_1\mbox{ is an $\ell$-instance of }T}\sim \pr\brk{\TT^{\tensor\,(\ell)}\ism T}.
	\end{align*}
\end{lemma}
\begin{proof}
	We proceed by induction on $\ell$.
	In the case $\ell=0$ there is nothing to show because the `tree' $\TT_{d,t}^{\tensor\,(\ell)}$ consists of the root only.
	Going from $\ell-1$ to $\ell>0$, let $\vX_\tau$ be the number of children of $\root$ of type $\tau\in\{0,1,2\}$ in $\TT_{d,t}^\tensor$.
	Moreover, let $\vY_0$ be the number of shared edges $e\in\{\ve_1,\ldots,\ve_{M-1}\}$ incident with $v_1$.
	Additionally, let $\vY_h$ be the number of edges $e\in\{\ve_{h,1},\ldots,\ve_{h,m-M}\}$ incident with $v_1$.
	The construction of $(\GG_{h,M})_{h=1,2}$ and of $\TT^\tensor$ ensures that
	\begin{align}\label{eq_lem_lwc_tensor_ex_1}
		\dTV((\vX_0,\vX_1,\vX_2),(\vY_0,\vY_1,\vY_2))&=o(1).
	\end{align}

	Consider the event $\fE$ that the depth-$\ell$ neighbourhood of $v_1$ in the union $\GG_{1,M}\cup\GG_{2,M}=(V_n,E(\GG_{1,M})\cup E(\GG_{2,M}))$ is acyclic and contains a total of at most $\log n$ edges.
	Then the Chernoff bound shows together with a routine breadth-first search exploration argument that
	\begin{align}\label{eq_lem_lwc_tensor_ex_2}
		\pr\brk{\fE}&=1-O(1/n).
	\end{align}
	Let
	\begin{align*}
		\cY_0&=|\partial(\GG_{1,M},v_1)\cap\partial(\GG_{2,M},v_1)|,&
		\cY_h&=|\partial(\GG_{h,M},v_1)\setminus\partial(\GG_{3-h,M},v_1)|&&(h=1,2).
	\end{align*}
	Hence, $\cY_0$ is the number of `shared' neighbours of $v_1$ in $\GG_{1,M}$ and $\GG_{2,M}$, and $\cY_h$ is the number of `separate' neighbours of $v_1$ in $\GG_{h,M}$.
	Given $\fE$ we have $\cY_0\cap\cY_1=\cY_0\cap\cY_2=\cY_1\cap\cY_2=\emptyset$ and $|\cY_0|=\vY_0$ and $|\cY_h|=\vY_h$ for $h=1,2$.

	Now, given $\fE$ and a vertex $v\in\cY_0\cup\cY_1\cup\cY_2$ let $\cN_\ell(v)$ be the set of all vertices $w\in\nabla^\ell(\GG_{1,M}\cup\GG_{2,M},v_1)$ such that the shortest path from $w$ to $v_1$ passes through $v$.
	Obtain $\GG_{h,M}[v]$ from $\GG_{h,M}$ by deleting all the vertices in $\nabla^\ell(\GG_{h,M},v_1)\setminus\cN_\ell(v)$.
	Since on $\fE$ we have $|\nabla^\ell(\GG_{h,M},v_1)|=O(\log n)$, the graph $\GG_{h,M}[v]$ contains $n+O(\log n)\sim n$ vertices and $m+O(\log n)\sim dn/2$ edges.
	Thus, the pair $(\GG_{h,M}[v])_{h=1,2}$ is distributed as a pair of random graphs $(\GG_{h,M+O(\log n)}(n+O(\log n),m+O(\log n)))_{h=1,2}$, i.e., a pair of random graphs produced by the coupling {\bf CPL0}--{\bf CPL2} with asymptotically unchanged parameters.
	Therefore, by induction for every possible outcome $T'$ of $\TT_{d,t}^\tensor$ and any vertex $v\in\cY_0$ we have
	\begin{align}\label{eq_lem_lwc_tensor_ex_3}
		\pr\brk{v\mbox{ is an $(\ell-1)$-instance of $T'$ in $(\GG_{h,M}[v])_{h=1,2}$}\mid\fE}\sim \pr\brk{\TT^{\tensor\,(\ell-1)}_{d,t}\ism T'}.
	\end{align}
	Similarly, by \Lem~\ref{lem_lwc} for any vertex $v\in \cY_h$ and any possible outcome $T''$ of the plain Galton-Watson tree $\TT_d$ we have
	\begin{align}\label{eq_lem_lwc_tensor_ex_4}
		\pr\brk{\nabla^{\ell-1}(\GG_{h,M}[v],v)\ism T''\mid\fE}\sim \pr\brk{\TT^{(\ell-1)}_{d}\ism T''}&&(h=1,2).
	\end{align}
	Finally, given $\fE$ the events from~\eqref{eq_lem_lwc_tensor_ex_3}--\eqref{eq_lem_lwc_tensor_ex_4} are asymptotically independent.
	Therefore, the assertion follows from \eqref{eq_lem_lwc_tensor_ex_1}--\eqref{eq_lem_lwc_tensor_ex_4} and the recursive definition of the tree $\TT_{d,t}^\tensor$.
\end{proof}

\begin{proof}[Proof of \Lem~\ref{lem_lwc_tensor}]
	\Lem~\ref{lem_lwc_tensor_ex} implies that
	\begin{align}\label{eq_lem_lwc_tensor_1}
		\ex\brk{\vN_T^{(\ell)}(n,M)}&\sim n\pr\brk{\TT_{d,t}^\tensor\ism T}.
	\end{align}
	Furthermore, adding or removing a single edge $e=vw$ to or from $\GG_{1,M}\cup\GG_{2,M}$ can alter the number of $\ell$-instances of $T$ by only a bounded amount.
	To see this, let $\Delta$ be the maximum degree of $T$.
	For $e$ to turn a vertex $u$ from  being an $\ell$-instance to not being an $\ell$-instance or vice versa, one of the vertices $v,w$ has to have distance at most $\ell-1$ from $u$.
	Say this is vertex $v$.
	Then there has to be a path $u=u_1,\ldots,u_l=v$ with $l\leq\ell-1$ such that $u_1,\ldots,u_l$ have degree at most $\Delta$.
	Since there exist no more than $\Delta^{\ell-1}$ paths of length $\ell-1$ emanating from $v$ where every vertex has degree at most $\Delta$, the number of affected vertices $u$ is bounded.
	Consequently, a straightforward application of Azuma--Hoeffding shows that \whp
	\begin{align}\label{eq_lem_lwc_tensor_2}
		\vN_T^{(\ell)}(n,M)=\ex\brk{\vN_T^{(\ell)}(n,M)}+o(n).
	\end{align}
	The assertion follows from \eqref{eq_lem_lwc_tensor_1}--\eqref{eq_lem_lwc_tensor_2}.
\end{proof}

\subsection{Proof of \Prop~\ref{prop_tensor_B>0}}\label{sec_prop_tensor_B>0}

We begin with an observation about the \Erdos-\Renyi\ graph $\GG$.
Recall from~\eqref{eqMuT} that for a rooted graph $(G,r)$ we denote by $\mu^{(\ell)}_{G,r,\beta,B}$ the Ising distribution with a $+1$ boundary condition imposed on all vertices at distance $\ell$ from the root $r$.
\Lem~\ref{lem_kostas} directly implies that for any $(G,r)$, any $\beta,B\geq0$ and any $\ell\geq0$ we have
\begin{align}\label{eqlem_GGmag0}
	\scal{\SIGMA(r)}{\mu_{G,\beta,B}}\leq\scal{\SIGMA(r)}{\mu^{(\ell)}_{G,r,\beta,B}}.
\end{align}
The following lemma shows that on the \Erdos-\Renyi\ graph $\GG$ rooted at $v_1$ the bound~\eqref{eqlem_GGmag0} gets asymptotically tight in the limit $\ell\to\infty$, provided $B>0$.
The statement is implicit in~\cite[\Sec~6]{Dembo_2010} but we provide a short proof for completeness.

\begin{lemma}\label{lem_GGmag}
	If $B>0$ then $\limsup_{\ell\to\infty}\limsup_{n\to\infty}\ex\abs{\scal{\SIGMA(v_1)}{\mu_{\GG,\beta,B}}-\scal{\SIGMA(v_1)}{\mu^{(\ell)}_{\GG,v_1,\beta,B}}}=0.$
\end{lemma}
\begin{proof}
	Recall from~\eqref{eqpidbetaBell} that $\pi_{d,\beta,B}^{(\ell)}=\BP^\ell_{d,\beta,B}(\delta_1)$.
	Hence, \eqref{eqDMweakLimit} and Equation~\eqref{eqprop_DM_B>0} from \Prop~\ref{prop_DM_B>0} show that
	\begin{align}\label{eqlem_GGmag1}
		\limsup_{\ell\to\infty}\limsup_{n\to\infty}\ex[W_1(\pi_{\GG,\beta,B},\pi^{(\ell)}_{d,\beta,B})]&=\limsup_{n\to\infty}\ex[W_1(\pi_{\GG,\beta,B},\pi_{d,\beta,B})]=0.
	\end{align}
	Further,  \Lem~\ref{lem_lwc} (the local weak convergence result) and \Cor~\ref{cor_BP_tree} show that the mean of $\pi^{(\ell)}_{d,\beta,B}$ satisfies
	\begin{align}\label{eqlem_GGmag2}
		\int_{-1}^1\xi\dd\pi^{(\ell)}_{d,\beta,B}(\xi)&=\ex\scal{\SIGMA(\root)}{\mu_{\TT,\beta,B}^{(\ell)}}
		=\frac1n\sum_{i=1}^n\ex\scal{\SIGMA(v_i)}{\mu^{(\ell)}_{\GG,v_i,\beta,B}}+o(1)
		=\ex\scal{\SIGMA(v_1)}{\mu^{(\ell)}_{\GG,v_1,\beta,B}}+o(1).
	\end{align}
	Combining \eqref{eqlem_GGmag1} and \eqref{eqlem_GGmag2}, we obtain
	\begin{align}\label{eqlem_GGmag3}
		\lim_{n\to\infty}\ex\scal{\SIGMA(v_1)}{\mu_{\GG,\beta,B}}=
		\int_{-1}^1\xi\dd\pi_{d,\beta,B}(\xi)&=
		\lim_{\ell\to\infty}\int_{-1}^1\xi\dd\pi^{(\ell)}_{d,\beta,B}(\xi)=
		\lim_{\ell\to\infty}\lim_{n\to\infty}\ex\scal{\SIGMA(v_1)}{\mu^{(\ell)}_{\GG,v_1,\beta,B}}.
	\end{align}
	Finally, since \eqref{eqlem_GGmag0} shows that $\Scal{\SIGMA(v_1)}{\mu^{(\ell)}_{\GG,v_1,\beta,B}}\geq\scal{\SIGMA(v_1)}{\mu_{\GG,\beta,B}}$, the assertion follows from~\eqref{eqlem_GGmag3}.
\end{proof}

\begin{corollary}\label{cor_GGmag}
	If $d,\beta,B>0$ and $0\leq M\leq m$, then
		\begin{align*}
			\limsup_{\ell\to\infty}\limsup_{n\to\infty}\sum_{h=1,2}\ex\abs{\scal{\SIGMA(v_1)}{\mu_{\GG_{h,M},\beta,B}}-\scal{\SIGMA(v_1)}{\mu^{(\ell)}_{\GG_{h,M},v_1,\beta,B}}}&=0.
		\end{align*}
\end{corollary}
\begin{proof}
	Since each graph $\GG_{h,M}$ separately is distributed as the random graph $\GG(n,m-1)$ and $2(m-1)/n\sim 2m/n\sim d$, the assertion follows from \Lem~\ref{lem_GGmag} and the union bound.
\end{proof}

\begin{proof}[Proof of \Prop~\ref{prop_tensor_B>0}]
	Let
	\begin{align*}
		\vec\pi_{n,M,\beta,B}^{\tensor\,(\ell)}&=\frac1n\sum_{i=1}^n\delta_{\bc{\scal{\SIGMA(v_i)}{\mu_{\GG_{1,M},v_i,\beta,B}^{(\ell)}},\scal{\SIGMA(v_i)}{\mu_{\GG_{2,M},v_i,\beta,B}^{(\ell)}}}}
	\end{align*}
	be the empirical distribution of the joint magnetisations $(\langle{\SIGMA(v_i)},{\mu_{\GG_{h,M},v_i,\beta,B}^{(\ell)}}\rangle)_{h=1,2}$ with a $+1$ boundary condition imposed at distance $\ell$ from $v_i$.
	\Cor~\ref{cor_GGmag} shows that
	\begin{align}\label{eq_prop_tensor_B>0_1}
		\limsup_{\ell\to\infty}\limsup_{n\to\infty}W_1\bc{\vec\pi_{n,M,\beta,B}^{\tensor\,(\ell)},\vec\pi_{n,M,\beta,B}^{\tensor}}&=0.
	\end{align}
	Furthermore, the local weak convergence from \Lem~\ref{lem_lwc_tensor} implies together with \Lem~\ref{lem_tensor_ell} that for any $\ell\geq0$,
	\begin{align}\label{eq_prop_tensor_B>0_2}
		W_1\bc{\vec\pi_{n,M,\beta,B}^{\tensor\,(\ell)},\pi_{d,\beta,B,t}^{\tensor\,(\ell)}}&=o(1).
	\end{align}
	Thus, the assertion follows from	\eqref{eq_prop_tensor_B>0_1}--\eqref{eq_prop_tensor_B>0_2} and \Cor~\ref{cor_tensor_ell}.
\end{proof}

\subsection{Proof of \Prop~\ref{cor_tensor_B>0}}\label{sec_cor_tensor_B>0}
In light of \Lem~\ref{fact_telescope} we are going to study the terms $\vDelta_{n,M}^+,\vDelta_{n,M}^-,\vDelta_{n,M}^\pm$ from \eqref{eqDeltaM}.
Let $\fG_{n,M}$ be the $\sigma$-algebra generated by the random graphs $(\GG_{h,M})_{h=1,2}$ and set
\begin{align*}
	\vDelta_{n,M}=\vDelta^+_{n,M}+\vDelta^-_{n,M}-2\vDelta^{\pm}_{n,M}.
\end{align*}
Establishing the following lemma constitutes the key step towards the proof of \Prop~\ref{cor_tensor_B>0}.

\begin{lemma}\label{lem_intervals}
	Assume that $B>0$.
	Uniformly for all $0\leq M\leq m$ we have
	\begin{align*}
		\ex\brk{\vDelta_{n,M}\mid\fG_{n,M}}&=\cB_{\beta}^\tensor(\pi^\tensor_{d,\beta,B,M/m})-\cB_{\beta}^\tensor(\pi^\tensor_{d,\beta,B,0})+o(1)&&\mbox\whp
	\end{align*}
\end{lemma}

The proof of \Lem~\ref{lem_intervals} proceeds in several steps.
We begin by disecting the log-terms that make up the three expressions $\vDelta_{n,M}^+,\vDelta_{n,M}^-,\vDelta_{n,M}^\pm$ from~\eqref{eqDeltaM}.
Let $\vv_{M},\vw_{M}$ be the two vertices incident with edge $\ve_M$ from step {\bf CPL0} of the construction of $\GG_{h,M}$.
Then given $\fG_{n,M}$ the pair $(\vv_M,\vw_M)$ is within total variation distance $o(1)$ of a uniformly random pair of vertices from $V_n$.
Indeed, the only condition imposed on $(\vv_M,\vw_M)$ is that $\vv_M\neq\vw_M$ and that $\{\vv_M,\vw_M\}$ is not an edge of $\GG_{1,M}\cup\GG_{2,M}$ already, and a uniformly random pair of vertices from $V_n$ satisfies these two conditions \whp
Additionally, let $\vv_{h,M},\vw_{h,M}$ be the vertices that constitute the edge $\ve_{h,M,m-M}$ from {\bf CPL2} that gets added to $\GG_{h,M}$ to obtain $\GG_{h,M}^-$.

\begin{claim}\label{claim_intervals_A}
	Let $B>0$.
	Then 
	for every $0\leq M\leq m$ \whp\ we have
	\begin{align}\label{eq_lem_intervals3}
		\log\frac{Z_{\GG_{h,M}^+}(\beta,B)}{Z_{\GG_{h,M}}(\beta,B)}&=\log\cosh\beta+\log\bc{1+\scal{\SIGMA(\vv_{M})}{\mu_{\GG_{h,M}}}\scal{\SIGMA(\vw_{M})}{\mu_{\GG_{h,M}}}\tanh\beta}+o(1),\\
\label{eq_lem_intervals3a}
		\log\frac{Z_{\GG_{h,M}^-}(\beta,B)}{Z_{\GG_{h,M}}(\beta,B)}&=\log\cosh\beta+\log\bc{1+\scal{\SIGMA(\vv_{h,M})}{\mu_{\GG_{h,M}}}\scal{\SIGMA(\vw_{h,M})}{\mu_{\GG_{h,M}}}\tanh\beta}+o(1).
	\end{align}
\end{claim}
\begin{proof}
	According to \Lem~\ref{fact_Zratio} with probability $1-O(1/n)$ we have
	\begin{align}\label{eq_lem_intervals1}
		\log\frac{Z_{\GG_{h,M}^+}(\beta,B)}{Z_{\GG_{h,M}}(\beta,B)}&=\log\scal{\exp\bc{\beta\SIGMA(\ve_{M})}}{\mu_{\GG_{h,M},\beta,B}}=\log\cosh\beta+\log\bc{1+\scal{\SIGMA(\vv_{M})\SIGMA(\vw_{M})}{\mu_{\GG_{h,M}}}\tanh\beta}.
	\end{align}
	Since given $\fG_{n,M}$ the vertex pair $(\vv_{M},\vw_{M})$ is asymptotically uniformly random, the replica symmetry property \eqref{eqTensorB>0RS} from \Prop~\ref{prop_tensor_B>0} implies that \whp\
	\begin{align}\label{eq_lem_intervals2}
		\scal{\SIGMA(\vv_{M})\SIGMA(\vw_{M})}{\mu_{\GG_{h,M}}}=\scal{\SIGMA(\vv_{M})}{\mu_{\GG_{h,M}}}\scal{\SIGMA(\vw_{M})}{\mu_{\GG_{h,M}}}+o(1).
	\end{align}
	Combining~\eqref{eq_lem_intervals1} and~\eqref{eq_lem_intervals2} yields \eqref{eq_lem_intervals3}
	Similar reasoning applies to $\vv_{h,M},\vw_{h,M}$, whence we obtain~\eqref{eq_lem_intervals3a}.
\end{proof}

Following~\eqref{eqIndependentSamples}, for a probability measure $\pi\in\cP([-1,1]^2)$ we denote by $\vec\mu_{\pi,i}^{(k)}=(\vec\mu_{\pi,i,h}^{(k)})_{h=1,2}\in[-1,1]^2$ mutually independent random variables with distribution $\pi$ for $k\in\{0,1,2\}$ and $i\geq1$.
For $0\leq t\leq 1$ we define
\begin{align}\label{eq_Xi_1}
		\vec\Xi^+_t&=\prod_{h=1}^2\log\cosh\beta+\log\bc{1+\vec\mu_{\pi_{d,\beta,B,t}^\tensor,1,h}^{(0)}\vec\mu_{\pi_{d,\beta,B,t}^\tensor,2,h}^{(0)}\tanh\beta},\\
		\vec\Xi^-_t&=\prod_{h=1}^2\log\cosh\beta+\log\bc{1+\vec\mu_{\pi_{d,\beta,B,t}^\tensor,2h+1,h}^{(1)}\vec\mu_{\pi_{d,\beta,B,t}^\tensor,2h+2,h}^{(1)}\tanh\beta},\label{eq_Xi_2}\\
		\vec\Xi^\pm_t&=\prod_{h=1}^2\log\cosh\beta+\log\bc{1+\vec\mu_{\pi_{d,\beta,B,t}^\tensor,2h+1,h}^{(2)}\vec\mu_{\pi_{d,\beta,B,t}^\tensor,2h+2,h}^{(2)}\tanh\beta},\label{eq_Xi_3}
	\end{align}
and set $\vec\Xi_t=\vec\Xi^+_t+\vec\Xi^-_t-2\vec\Xi^\pm_t.$
Let $\vec\delta_{n,M}^+$ be the conditional distribution of the random variable $\vDelta_{n,M}^+$ given $\fG_{n,M}$.
Thus, depending on the random graphs $(\GG_{h,M})_{h=1,2}$, $\vec\delta_{n,M}^+$ is a {\em random} probability distribution on $\RR$.
Define $\vec\delta_{n,M}^-$ and $\vec\delta_{n,M}^\pm$ analogously with respect to $\vDelta_{n,M}^-$ and $\vDelta_{n,M}^\pm$.
Furthermore, let $\xi^+_t$, $\xi_t^-$ and $\xi_t^\pm$ be the laws of the random variables $\Xi^+_t,\Xi^-_t$ and $\Xi^\pm_t$, respectively.
We are going to show that $\vec\delta_{n,M}^+,\vec\delta_{n,M}^-,\vec\delta_{n,M}^\pm$ are close to $\xi_t^+,\xi_t^-,\xi_t^\pm$ with $t=M/m$ \whp

\begin{claim}\label{claim_intervals_C}
	Let $B>0$.
	For all $0\leq M\leq m$ we have
	\begin{align}\label{eq_claim_intervals_C}
		\ex\brk{W_1(\vec\delta_{n,M}^+,\xi_{M/m}^+)+W_1(\vec\delta_{n,M}^-,\xi_{M/m}^-)+W_1(\vec\delta_{n,M}^\pm,\xi_{M/m}^\pm)}&=o(1).
	\end{align}
\end{claim}

To be quite explicit, the expectation in \eqref{eq_claim_intervals_C} refers to the choice of $(\GG_{h,M})_{h=1,2}$.
As far as the conditional distributions $\vec\delta_{n,M}^+,\vec\delta_{n,M}^-,\vec\delta_{n,M}^\pm$ are concerned, the graphs $(\GG_{h,M})_{h=1,2}$ are deemed fixed and the only randomness stems from the extra edges added in the course of {\bf CPL1}--{\bf CPL2} to obtain $\GG_{h,M}^+$ and $\GG_{h,M}^-$, respectively.

\begin{proof}[Proof of Claim~\ref{claim_intervals_C}]
	Since the vertex pair $(\vv_M,\vw_M)$ is asymptotically uniformly random given $\fG_{n,M}$,
	Eq.~\eqref{eqprop_tensor_B>0} from \Prop~\ref{prop_tensor_B>0} show that \whp\ given $\fG_{n,M}$ the two pairs of magnetisations
	\begin{align*}
		\bc{\scal{\SIGMA(\vv_{M})}{\mu_{\GG_{h,M}}}}_{h=1,2},&&\bc{\scal{\SIGMA(\vw_{M})}{\mu_{\GG_{h,M}}}}_{h=1,2}
	\end{align*}
	converge in distribution to two independent samples from $\pi_{d,\beta,B,M/m}^\tensor$.
	Hence, \eqref{eq_lem_intervals3} implies that
	\begin{align*}
		W_1(\vec\delta_{n,M}^+,\xi_{M/m}^+)&=o(1)&&\mbox\whp
	\end{align*}
	Analogously, we obtain $W_1(\vec\delta_{n,M}^-,\xi_{M/m}^-),W_1(\vec\delta_{n,M}^\pm,\xi_{M/m}^\pm)=o(1)$ \whp
\end{proof}

\begin{proof}[Proof of \Lem~\ref{lem_intervals}]
	Combing Claims~\ref{claim_intervals_A} and~\ref{claim_intervals_C}, we see that \whp
	\begin{align}\label{eq_lem_intervals_50}
		\ex\brk{\vDelta_{n,M}\mid\fG_{n,M}}&=\ex\brk{\vec\Xi_{M/m}}+o(1).
	\end{align}
	Furthermore, inspecting the definition~\eqref{eq_Xi_1} and reminding ourselves of~\eqref{eqBFEtensor}, we calculate
	\begin{align}\label{eq_lem_intervals_51}
	\ex\brk{\vec\Xi_{M/m}^+}&=\log^2\cosh\beta+2\ex\brk{\log\bc{1+\vec\mu_{\pi_{d,\beta,t}^\tensor,1,1}^{(0)}\vec\mu_{\pi_{d,\beta,t}^\tensor,2,1}^{(0)}\tanh\beta}}\log\cosh\beta+\cB_\beta^\tensor(\pi_{d,\beta,B,M/m}^\tensor).
	\end{align}
	Similarly, we obtain
	\begin{align}\nonumber
		\ex\brk{\vec\Xi_{M/m}^-}&=\ex\brk{\vec\Xi_{M/m}^-}\\
								&=\log^2\cosh\beta+2\ex\brk{\log\bc{1+\vec\mu_{\pi_{d,\beta,t}^\tensor,1,1}^{(0)}\vec\mu_{\pi_{d,\beta,t}^\tensor,2,1}^{(0)}\tanh\beta}}\log\cosh\beta+\ex\brk{\log\bc{1+\vec\mu_{\pi_{d,\beta,t}^\tensor,1,1}^{(0)}\vec\mu_{\pi_{d,\beta,t}^\tensor,2,1}^{(0)}\tanh\beta}}^2\nonumber\\
								&=\log^2\cosh\beta+2\ex\brk{\log\bc{1+\vec\mu_{\pi_{d,\beta,t}^\tensor,1,1}^{(0)}\vec\mu_{\pi_{d,\beta,t}^\tensor,2,1}^{(0)}\tanh\beta}}\log\cosh\beta+\ex\brk{\log\bc{1+\vec\mu_{\pi_{d,\beta,0}^\tensor,1,1}^{(0)}\vec\mu_{\pi_{d,\beta,0}^\tensor,2,1}^{(0)}\tanh\beta}}^2\nonumber\\
								&=\log^2\cosh\beta+2\ex\brk{\log\bc{1+\vec\mu_{\pi_{d,\beta,t}^\tensor,1,1}^{(0)}\vec\mu_{\pi_{d,\beta,t}^\tensor,2,1}^{(0)}\tanh\beta}}\log\cosh\beta+\cB_\beta^\tensor(\pi_{d,\beta,B,0}^\tensor).
\label{eq_lem_intervals_52}
	\end{align}
	Combining~\eqref{eq_lem_intervals_50}--\eqref{eq_lem_intervals_52} completes the proof.
\end{proof}

\begin{corollary}\label{cor_intervals}
	If $B>0$ then $\Var\sum_{0\leq M<m}\ex\brk{\vDelta_{n,M}\mid\fF_{n,M}}=o(m^2)$.
\end{corollary}
\begin{proof}
	We expand
	\begin{align}\nonumber
		\Var\sum_{0\leq M<m}\ex\brk{\vDelta_{n,M}\mid\fF_{n,M}}&=\sum_{0\leq M<m}\Var\bc{\ex\brk{\vDelta_{n,M}\mid\fF_{n,M}}}\\&+2\sum_{0\leq M<M'<m}\ex\brk{\ex\brk{\vDelta_{n,M}\mid\fF_{n,M}}\ex\brk{\vDelta_{n,M'}\mid\fF_{n,M'}}}-\ex\brk{\vDelta_{n,M}}\ex\brk{\vDelta_{n,M'}}.
		\label{eq_cor_intervals_1}
	\end{align}
	Since the random variables $\vDelta_{n,M}$ are bounded uniformly, we immediately obtain
	\begin{align}\label{eq_cor_intervals_2}
		\sum_{0\leq M<m}\Var\bc{\ex\brk{\vDelta_{n,M}\mid\fF_{n,M}}}&=O(m).
	\end{align}

	To bound the second summand from~\eqref{eq_cor_intervals_1} fix $0\leq M<M'\leq m$.
	Since $\fF_{n,M}\subset\fG_{n,M}$, Claim~\ref{claim_intervals_C} implies that
	\begin{align}\label{eq_cor_intervals_10}
		\ex\brk{\vDelta_{n,M}\mid\fF_{n,M}}&=\ex\brk{\vec\Xi_{n,M/m}}+o(1),&
		\ex\brk{\vDelta_{n,M'}\mid\fF_{n,M'}}&=\ex\brk{\vec\Xi_{n,M'/m}}+o(1)&&\mbox\whp
	\end{align}
	Furthermore, taking expectations of the l.h.s.\ expressions from~\eqref{eq_cor_intervals_10}, we get
	\begin{align}\label{eq_cor_intervals_11}
		\ex\brk{\vDelta_{n,M}}&=\ex\brk{\vec\Xi_{n,M/m}}+o(1),&
		\ex\brk{\vDelta_{n,M'}}&=\ex\brk{\vec\Xi_{n,M'/m}}+o(1).
	\end{align}
	Combining~\eqref{eq_cor_intervals_10}--\eqref{eq_cor_intervals_11}, we arrive at the bound
	\begin{align}
		\sum_{0\leq M<M'<m}\ex\brk{\ex\brk{\vDelta_{n,M}\mid\fF_{n,M}}\ex\brk{\vDelta_{n,M'}\mid\fF_{n,M'}}}-\ex\brk{\vDelta_{n,M}}\ex\brk{\vDelta_{n,M'}}&=o(m^2).
		\label{eq_cor_intervals_12}
	\end{align}
	Finally, the assertion follows from~\eqref{eq_cor_intervals_1}, \eqref{eq_cor_intervals_2} and~\eqref{eq_cor_intervals_12}.
\end{proof}

\begin{proof}[Proof of \Prop~\ref{cor_tensor_B>0}]
	Since the random variables $\vDelta_{n,M}$ are all bounded uniformly, \Lem~\ref{lem_intervals} yields
	\begin{align*}
		\frac1m\sum_{0\leq M<m}\ex\brk{\vDelta_{n,M}}&=\frac1m\sum_{0\leq M<m}\cB_{\beta}^\tensor(\pi^\tensor_{d,\beta,B,M/m})-\cB_{\beta}^\tensor(\pi^\tensor_{d,\beta,B,0})+o(1).
	\end{align*}
	Due to the continuity statement from \Prop~\ref{prop_fix_ex} we can replace the sum on the r.h.s.\ by an integral to obtain
	\begin{align}\label{eq_cor_tensor_B>0_1}
		\frac1m\sum_{0\leq M<m}\ex\brk{\vDelta_{n,M}}&=\int_0^1\cB_{\beta}^\tensor(\pi^\tensor_{d,\beta,B,t})\dd t-\cB_{\beta}^\tensor(\pi^\tensor_{d,\beta,B,0})+o(1).
	\end{align}
	Further, \eqref{eq_cor_tensor_B>0_1} implies together with \Cor~\ref{cor_intervals} and Chebyshev's inequality that
	\begin{align}\label{eq_cor_tensor_B>0_10}
		\ex\abs{
		\int_0^1\cB_{\beta}^\tensor(\pi^\tensor_{d,\beta,B,t})\dd t-\cB_{\beta}^\tensor(\pi^\tensor_{d,\beta,B,0})-\frac1m\sum_{0\leq M<m}\vDelta_{n,M}}&=o(1).
	\end{align}
	Thus, the assertion follows from \eqref{eq_cor_tensor_B>0_10} and \Lem~\ref{fact_telescope}.
\end{proof}

\subsection{Proof of \Prop~\ref{prop_var_pos_B>0}}\label{sec_prop_var_pos_B>0}

The fundamental reason why the random variable $\log Z_{\GG}(\beta,B)$ has variance $\Omega(n)$ is that the limiting distribution $\pi_{d,\beta,B}$ of the vertex magnetisations fails to concentrate on a single point, as verified by the following lemma.

\begin{lemma}\label{lem_noatom_B>0}
	Assume that $B>0$.
	Then $\pi_{d,\beta,B}$ is not supported on a single point.
\end{lemma}
\begin{proof}
	We produce two distinct points that belong to the support of $\pi_{d,\beta,B}$.
	With probability $\exp(-d)$ the Galton-Watson tree $\TT$ consists of the root only.
	In this case we have $\Scal{\SIGMA(\root)}{\mu_{\TT,\root}^{(\ell)}}=\tanh B$ for all $\ell\geq1$.
	Hence,
	\begin{align}\label{eq_lem_noatom_B>0_1}
		\tanh B\in\supp{\pi_{d,\beta,B}}.
	\end{align}

	Further, with probability $d\exp(-2d)$ the tree $\TT$ consists of the root $\root$ and a single child $\theta$ (and the edge $\root\theta$, of course).
	A direct calculation shows that in this case we have
	\begin{align}\label{eq_lem_noatom_B>0_2}
		\scal{\SIGMA(\root)}{\mu_{\TT,\root}^{(\ell)}}&=\eta(\beta,B)=\frac{\exp(2(\beta+B))-\exp(2(\beta-B))}{2+\exp(2(\beta+B))+\exp(2(\beta-B))}.
	\end{align}
	The function $\eta(\beta,B)$ satisfies $\eta(0,B)=\tanh B$ and $\frac\partial{\partial\beta}\eta(\beta,B)>0$.
	Hence, \eqref{eq_lem_noatom_B>0_1} and~\eqref{eq_lem_noatom_B>0_2} demonstrate that the two distinct points $\tanh B,\eta(\beta,B)$ belong to the support of $\pi_{d,\beta,B}$.
\end{proof}

To show that the integral $\Sigma(d,\beta,B)=\int_0^1\cB^\tensor_{\beta}(\pi^\tensor_{d,\beta,B,t})-\cB^\tensor_{\beta,B,0}(\pi^\tensor_{d,\beta,B,0})\dd t$ from~\eqref{eqSigmadbetaB} is strictly positive, we first ascertain that the function being integrated is non-negative for all $t$.

\begin{lemma}\label{lem_integrant_B>0}
	Assume that $B>0$.
	Then $\cB^\tensor_\beta(\pi^\tensor_{d,\beta,B,t})\geq\cB^\tensor_\beta(\pi^\tensor_{d,\beta,B,0})$ for all $0\leq t\leq1$.
\end{lemma}
\begin{proof}
	Fix any $\eps>0$.
	\Lem~\ref{fact_telescope} and \Lem~\ref{lem_intervals} show that for sufficiently large $n$ for all $0\leq M\leq m$ we have
	\begin{align}\label{eq_lem_integrant_B>0_100}
		\cB_{\beta}^\tensor(\pi^\tensor_{d,\beta,B,M/m})-\cB_{\beta}^\tensor(\pi^\tensor_{d,\beta,B,0})\geq-\eps.
	\end{align}
	Since \Prop~\ref{prop_fix_ex} provides that the function $t\in[0,1]\mapsto\cB_\beta^\tensor(\pi^\tensor_{d,\beta,B,t})$ is continuous and thus uniformly continuous, the assertion follows from \eqref{eq_lem_integrant_B>0_100}.
\end{proof}

Finally, we use a convexity argument to prove that the function value is strictly positive at $t=1$.

\begin{lemma}\label{lem_boundary_B>0}
	Assume that $B>0$.
	Then $\cB^\tensor_\beta(\pi^\tensor_{d,\beta,B,1})>\cB^\tensor_\beta(\pi^\tensor_{d,\beta,B,0})$ for all $0\leq t\leq1$.
\end{lemma}
\begin{proof}
	We already noticed in Remark~\ref{rem_extreme} that $\pi^\tensor_{d,\beta,B,0}=\pi_{d,\beta,B}\tensor\pi_{d,\beta,B}$ and that $\pi_{d,\beta,B,1}$ is precisely the distribution of the `diagonal' pair $(\vec\mu_{\pi_{d,\beta,B},1},\vec\mu_{\pi_{d,\beta,B},1})\in[0,1]^2$.
	Hence, evaluating the expression~\eqref{eqBFEtensor}, we obtain
	\begin{align}\label{eq_lem_boundary_B>0_1}
		\cB_\beta^\tensor(\pi_{d,\beta,B,0}^\tensor)&=\ex\brk{\log\bc{1+\vec\mu_{\pi_{d,\beta,B},1}\vec\mu_{\pi_{d,\beta,B},2}\tanh\beta}}^2,&
		\cB_\beta^\tensor(\pi_{d,\beta,B,1}^\tensor)&=\ex\brk{\log^2\bc{1+\vec\mu_{\pi_{d,\beta,B},1}\vec\mu_{\pi_{d,\beta,B},2}\tanh\beta}}.
	\end{align}
	Therefore, \Lem~\ref{lem_noatom_B>0} and Jensen's inequality yield $\cB^\tensor_\beta(\pi^\tensor_{d,\beta,B,1})>\cB^\tensor_\beta(\pi^\tensor_{d,\beta,B,0})$.
\end{proof}

\begin{proof}[Proof of \Prop~\ref{prop_var_pos_B>0}]
	The proposition follows from \Lem s~\ref{lem_noatom_B>0}--\ref{lem_boundary_B>0} and the continuity of the function $t\mapsto\cB^\tensor_\beta(\pi^\tensor_{d,\beta,B,t})$ established by \Prop~\ref{prop_fix_ex}.
\end{proof}

\section{Proof of \Thm~\ref{thm_no}}\label{sec_B=0}

\subsection{Proof of \Prop~\ref{prop_BP_B=0}}\label{sec_prop_BP_B=0}
Part~(i) of \Prop~\ref{prop_BP_B=0} follows directly from \Lem~\ref{lem_conv}, which covers all $B\geq0$.

To prove \Prop~\ref{prop_BP_B=0}~(iii) we observe the following.

\begin{lemma}\label{lem_extinction}
	For all $d,\beta>0$ we have $\pi_{d,\beta,0}(\{0\})\in\{\rho_d,1\}$.
\end{lemma}
\begin{proof}
	The distribution $\pi_{d,\beta,0}$ is a fixed point of the operator $\BP_{d,\beta,0}$ from~\eqref{eqBPop}, i.e., $\pi_{d,\beta,0}=\BP_{d,\beta,0}(\pi_{d,\beta,0})$.
	Since $\supp{\pi_{d,\beta,0}}\subseteq[0,1]$ by \Prop~\ref{prop_BP_B=0}~(i), the definition \eqref{eqBPrec} of $\BP_{d,\beta,0}$ implies that
	\begin{align}\nonumber
		\pi_{d,\beta,0}(\{0\})&=\pr\brk{\sum_{s\in\PM}s\prod_{i=1}^{\vec\Delta}\bc{1+s\vec\mu_{\pi_{d,\beta,0},i}\tanh\beta}=0}\\
							  &=\pr\brk{\max_{1\leq i\leq\vDelta}\vec\mu_{\pi_{d,\beta,0}}=0}=\sum_{j=0}^\infty\pr\brk{\Po(d)=j}\pi_{d,\beta,0}(\{0\})=\exp(d(\pi_{d,\beta,0}(\{0\})-1)).\label{eqlem_extinction1}
	\end{align}
	Furthermore, the fixed points of the convex function $z\mapsto\exp(d(z-1))$ are precisely $z=\rho_d$ and $z=1$.
	Thus, the assertion follows from \eqref{eqlem_extinction1}.
\end{proof}

\noindent
Hence, to prove \Prop~\ref{prop_BP_B=0}~(ii) we are left to rule out that $\pi_{d,\beta}=\delta_0$ for $d>1$ and $\beta>\betaf(d)$.

\begin{lemma}\label{lemma_lyons}
	If $d>1$ and $\beta>\betaf(d)$ then $\int_0^1z\dd\pi_{d,\beta,0}(z)>0$.
\end{lemma}
\begin{proof}
	Recalling that $\rho_d$ denotes the probability that the $\Po(d)$-Galton-Watson process dies out, we obtain
	\begin{align*}
		\int_0^1z\dd\pi_{d,\beta,0}(z)&=\lim_{\ell\to\infty}\int_0^1z\dd\pi_{d,\beta,0}^{(\ell)}(z)&&\mbox{[as $\lim_{\ell\to\infty}\pi_{d,\beta,0}^{(\ell)}=\pi_{d,\beta,0}$ weakly]}\\
									  &=\lim_{\ell\to\infty}\ex\brk{\scal{\SIGMA(\root)}{\mu_{\TT, \root,\beta, 0}^{(\ell)}}}&&\mbox{[by \Cor~\ref{cor_BP_tree}]}\\
									  &\geq(1-\rho_d)\lim_{\ell\to\infty}\ex\brk{\scal{\SIGMA(\root)}{\mu_{\TT, \root,\beta, 0}^{(\ell)}}\mid\TT\mbox{ is infinite}}>0&&\mbox{[by \Lem s~\ref{lem_LYonORg} and~\ref{cl_BrNumOfT}]},
	\end{align*}
	as desired.
\end{proof}

	\noindent
Thus, \Prop~\ref{prop_BP_B=0}~(ii) follows from \Lem s~\ref{lem_extinction} and~\ref{lemma_lyons}.

Finally, part~(iii) is a direct consequence of the following lemma from~\cite{Basak}.

\begin{lemma}[{\cite[\Lem~3.3]{Basak}}]\label{lem_basak}
	The function $\beta\mapsto r_{d,\beta}\in[0,1)$ is monotonically increasing and right-continuous.
	Moreover, at any continuity point $\beta\in(0,\infty)$ of this function the identity~\eqref{eqmax} holds.
\end{lemma}

\begin{remark}\label{rem_basak}
	The expression displayed in~\cite[\Lem~3.3]{Basak} differs syntactically from $r_{d,\beta}$ as defined in~\eqref{eqrdbeta} because~\cite{Basak} deals with general unimodular trees.
	The identity of the two expressions has been verified explicitly in~\cite[Eq.~(6.11)]{Dembo_2010}.
\end{remark}

\subsection{Proof of \Prop~\ref{prop_pure}}\label{sec_prop_pure}

In order to prove \Prop~\ref{prop_pure} we pursue the following strategy.
We know from \Lem~\ref{prop_pinning} that the measure $\mu_{\GG,\beta,0}$ admits {\em some} `pure state decomposition', with a potentially large (but bounded, in terms of $\delta$) number of `pure states' $\cS_\tau$ such that $\mu_{\GG,\beta,0}(\nix\mid\cS_\tau)$ is $\delta$-symmetric.
What we need to show is that actually these many $\cS_\tau$ amalgamate into only two pure states, namely $\cS_+$ and $\cS_-$, and that the empirical distribution of the conditional marginals given $\cS_+$ converges to $\pi_{d,\beta,0}$.
To this end we will deduce from \Prop~\ref{prop_BP_B=0}~(ii) that \whp\ for `most' of the pure states $\cS_\tau$ from \Lem~\ref{prop_pinning} and for `most' $v\in V_n$ for large $\ell$ the conditional magnetisation $\scal{\SIGMA(v)}{\mu_{\GG,\beta,0}(\nix\mid\cS_\tau)}$ is close to $s_v\cdot\Scal{\SIGMA(v)}{\mu_{\GG,v,\beta,0}^{(\ell)}}$ for a sign $s_v\in\PM$;
recall from~\eqref{eqMuT} that $\mu_{\GG,v,\beta,0}^{(\ell)}$ signifies the Ising distribution with all-ones boundary condition imposed on vertices at distance $\ell$ from $v$.
In particular, $\Scal{\SIGMA(v)}{\mu_{\GG,v,\beta,0}^{(\ell)}}$ is determined exclusively by the local geometry of $\GG$ around $v$.
Additionally, we are going to show that the sign $s_v=s_\tau$ actually depends on the pure state $\cS_\tau$ only, and not on the vertex $v$.
In other words, the signs of the magnetisations are `consistent' within the pure states $\cS_\tau$.
Once these facts are established, we are going to invoke the enhanced cut metric triangle inequality from \Lem~\ref{lem_xtriangle} to show that  all the $\cS_\tau$ with positive magnetisation amalgamate into the single pure state $\cS_+$ whose empirical magnetisations converge to $\pi_{d,\beta,0}$.

To carry this programme out we first need to find a way to get from \Prop~\ref{prop_BP_B=0}~(ii), which deals with {\em edge} magnetisations, back to vertex magnetisations.
As a first step we will derive an upper bound on the edge magnetisations that comes purely in terms of the local geometry of the random graph.
This Belief Propagation-inspired bound is fairly general and can actually be stated for arbitrary graphs $G=(V,E)$.
To elaborate, we introduce the following mild extension of the definition~\eqref{eqMuT} of the Boltzmann distribution $\mu_{G,r,\beta,B}^{(\ell)}$ with a $+1$ `boundary condition' imposed at distance $\ell$ from vertex~$r$.
Namely, for an edge $e=vw\in E$ define
\begin{align}\label{eq_def_msg}
	\fm^{(\ell)}_{G,v\to w,\beta}&=\scal{\SIGMA(v)}{\mu_{G-e,v,\beta,0}^{(\ell)}}\in[0,1).
\end{align}
Thus, we obtain $G-e$ from $G$ by removing the edge $e$, then impose a $+1$ `boundary condition' on all vertices $u$ at distance precisely $\ell$ from $v$ in $G-e$ and finally evaluate the magnetisation of $v$.
The definition~\eqref{eq_def_msg} is similar to the idea of a `standard Belief Propagation message'~\cite{ACOPBP,MM}.
We obtain the following bound.

\begin{lemma}\label{lem_max}
	Let $\ell\geq1$ be an integer, let $G=(V,E)$ be a graph and let $e=vw\in E$ be an edge such that the distance between $v$ and $w$ in $G-e$ exceeds $2\ell$.
	Then
	\begin{align*}
		\scal{\SIGMA(e)}{\mu_{G,\beta,0}}&\leq\frac{\fm_{G,v\to w,\beta}^{(\ell)}\fm_{G,w\to v,\beta}^{(\ell)}+\tanh\beta}{1+\fm_{G,v\to w,\beta}^{(\ell)}\fm_{G,w\to v,\beta}^{(\ell)}\tanh\beta}
	\end{align*}
\end{lemma}
\begin{proof}
	As a direct consequence of the definition~\eqref{eqIsingmu}--\eqref{eqIsingZ} of the Ising model we obtain
	\begin{align}\label{eq_lem_max_1}
		\scal{\SIGMA(e)}{\mu_{G,\beta,0}}&=\frac{\scal{\SIGMA(e)\exp(\beta\SIGMA(e))}{\mu_{G-e,\beta,0}}}{\scal{\exp(\beta\SIGMA(e))}{\mu_{G-e,\beta,0}}}=\frac{\scal{\SIGMA(e)}{\mu_{G-e,\beta,0}}+\tanh\beta}{1+\scal{\SIGMA(e)}{\mu_{G-e,\beta,0}}\tanh\beta}.
	\end{align}
	Thus, since the function $x\in[-1,1]\mapsto\frac{x+\tanh\beta}{1+x\tanh\beta}$ is monotonically increasing, we are left to show that
	\begin{align}\label{eq_lem_max_10}
		\scal{\SIGMA(e)}{\mu_{G-e,\beta,0}}&\leq\fm_{G,v\to w,\beta}^{(\ell)}\fm_{G,w\to v,\beta}^{(\ell)}.
	\end{align}

	To this end let $\mu_{G-e,e,\beta}^{(\ell)}$ be the Ising distribution with a $+1$ boundary condition imposed on all vertices that either have distance precisely $\ell$ from $v$ or from $w$; in symbols,
	\begin{align*}
		\mu_{G-e,e,\beta}^{(\ell)}(\sigma)&\propto\vecone\cbc{\forall u\in\partial^\ell(G-e,v)\cup\partial^\ell(G-e,w):\sigma(u)=1}\exp\bc{\beta\sum_{f\in E(G-e)}\sigma(f)}.
	\end{align*}
	Since the distance between $v,w$ in $G-e$ exceeds $2\ell$, \Lem~\ref{lem_kostas} shows that
	\begin{align}\label{eq_lem_max_2}
		\scal{\SIGMA(e)}{\mu_{G-e,\beta,0}}&\le\scal{\SIGMA(e)}{\mu_{G-e,e,\beta}^{(\ell)}}.
	\end{align}
	Moreover, again because the distance between $v,w$ in $G-e$ exceeds $2\ell$, the spins $\SIGMA(v),\SIGMA(w)$ are stochastically independent once we impose a $+1$ boundary condition on $\partial^\ell(G-e,v)\cup\partial^\ell(G-e,w)$.
	Therefore,
	\begin{align}\label{eq_lem_max_3}
		\scal{\SIGMA(e)}{\mu_{G-e,e,\beta}^{(\ell)}}&=\scal{\SIGMA(v)}{\mu_{G-e,e,\beta}^{(\ell)}}\scal{\SIGMA(w)}{\mu_{G-e,e,\beta}^{(\ell)}}=\fm_{G,v\to w,\beta}^{(\ell)}\fm_{G,w\to v,\beta}^{(\ell)} \enspace.
	\end{align}
	Combining~\eqref{eq_lem_max_2}--\eqref{eq_lem_max_3} yields~\eqref{eq_lem_max_10}, thereby completing the proof.
\end{proof}

As a next step we are going to prove that the upper bound from \Lem~\ref{lem_max} is saturated for most edges of the random graph.
For symmetry reasons it suffices to consider the `last' edge $\ve_m=\vv\vw$ of $\GG=\GG(n,m)$.

\begin{lemma}\label{cor_max}
	Assume that $\beta\in(0,\infty)\setminus\fX_d$.
	Then for any $\eps>0$ there exists $\ell_0=\ell_0(\eps,d,\beta)>0$ such that for all $\ell>\ell_0$,
		\begin{align}\label{eq_cor_max_1}
			\ex\abs{\scal{\SIGMA(\ve_m)}{\mu_{\GG,\beta,0}}-\frac{\fm_{\GG,\vv\to\vw,\beta}^{(\ell)}\fm_{\GG,\vw\to\vv,\beta}^{(\ell)}+\tanh\beta}{1+\fm_{\GG,\vv\to\vw,\beta}^{(\ell)}\fm_{\GG,\vw\to\vv,\beta}^{(\ell)}\tanh\beta}}&<\eps+o(1).
		\end{align}
\end{lemma}
\begin{proof}
	The starting point for the proof is provided by~\eqref{eqmax} and the definition~\eqref{eqrdbeta} of $r_{d,\beta}$.
	Indeed, because the distribution of $\GG$ is invariant under edge permutations, \eqref{eqmax} can be rewritten as
	\begin{align}\label{eq_cor_max_2}
		\ex\brk{\scal{\SIGMA(\ve_m)}{\mu_{\GG,\beta,0}}}&=\ex\brk{\frac{\vec\mu_{\pi_{d,\beta,0},1}\vec\mu_{\pi_{d,\beta,0},2}+\tanh\beta}{1+\vec\mu_{\pi_{d,\beta,0},1}\vec\mu_{\pi_{d,\beta,0},2}\tanh\beta}}+o(1).
	\end{align}
	Further, since $\pi_{d,\beta,0}$ is the $W_1$-limit of the measures $\pi_{d,\beta,0}^{(\ell)}$ from \eqref{eqpidbetaBell}, \eqref{eq_cor_max_2} implies that for $\ell>\ell_0$ large enough,
	\begin{align}\label{eq_cor_max_3}
		\abs{\ex\brk{\scal{\SIGMA(\ve_m)}{\mu_{\GG,\beta,0}}}-\ex\brk{\frac{\vec\mu_{\pi_{d,\beta,0}^{(\ell)},1}\vec\mu_{\pi_{d,\beta,0}^{(\ell)},2}+\tanh\beta}{1+\vec\mu_{\pi_{d,\beta,0}^{(\ell)},1}\vec\mu_{\pi_{d,\beta,0}^{(\ell)},2}\tanh\beta}}}<\eps+o(1).
	\end{align}

	Let $\GG'=\GG(n,m-1)$ be the graph comprising the random edges $\ve_1,\ldots,\ve_{m-1}$.
	To derive~\eqref{eq_cor_max_1} from~\eqref{eq_cor_max_3}, we remind ourselves that given $\GG'$ the pair of vertices $\vv,\vw$ that constitutes $\ve_m$ is within total variation distance $o(1)$ of a uniformly random pair of vertices from $V_n$.
	Therefore, \Lem~\ref{lem_lwc} shows that \whp\ given $\GG'$ the depth-$\ell$ neighbourhoods $\nabla^\ell(\GG,\vv)$, $\nabla^\ell(\GG,\vw)$ are within total variation distance $o(1)$ of a pair $\TT^{(\ell)},\TT^{\prime\,(\ell)}$ of independent $\Po(d)$ Galton-Watson trees truncated at depth $\ell$.
	Hence, recalling the definitions~\eqref{eq_def_msg} of $\fm_{\GG,\vv\to\vw,\beta}^{(\ell)},\fm_{\GG,\vw\to\vv,\beta}^{(\ell)}$ and applying \Cor~\ref{cor_BP_tree}, we obtain
	\begin{align}\label{eq_cor_max_4}
		W_1\bc{\bc{\fm_{\GG,\vv\to\vw,\beta}^{(\ell)},\fm_{\GG,\vw\to\vv,\beta}^{(\ell)}},\bc{\vec\mu_{\pi_{d,\beta,0}^{(\ell)},1},\vec\mu_{\pi_{d,\beta,0}^{(\ell)},2}}}&=o(1).
	\end{align}
	Combining \eqref{eq_cor_max_3} and~\eqref{eq_cor_max_4}, we find
	\begin{align}\label{eq_cor_max_10}
		\abs{\ex\brk{\scal{\SIGMA(\ve_m)}{\mu_{\GG,\beta,0}}}-\ex\brk{\frac{\fm_{\GG,\vv\to\vw,\beta}^{(\ell)}\fm_{\GG,\vw\to\vv,\beta}^{(\ell)}+\tanh\beta}{1+\fm_{\GG,\vv\to\vw,\beta}^{(\ell)}\fm_{\GG,\vw\to\vv,\beta}^{(\ell)}\tanh\beta}}}<\eps+o(1).
	\end{align}
	Moreover, since $\vv,\vw$ have distance $\Omega(\log n)$ in $\GG(n,m-1)$ \whp, \Lem~\ref{lem_max} implies the upper bound
	\begin{align}\label{eq_cor_max_5}
		\scal{\SIGMA(\ve_m)}{\mu_{\GG,\beta,0}}\leq\frac{\fm_{\GG,\vv\to\vw,\beta}^{(\ell)}\fm_{\GG,\vw\to\vv,\beta}^{(\ell)}+\tanh\beta}{1+\fm_{\GG,\vv\to\vw,\beta}^{(\ell)}\fm_{\GG,\vw\to\vv,\beta}^{(\ell)}\tanh\beta}&&\mbox\whp
	\end{align}
	Combining~\eqref{eq_cor_max_10} and~\eqref{eq_cor_max_5} completes the proof.
\end{proof}

We are going to use \Lem~\ref{cor_max} in order to investigate the conditional magnetisations of the `pure states' $\cS_\tau$ provided by \Lem~\ref{prop_pinning};
recall from~\eqref{eqsubcube} that for $U\subset V_n$ and $\tau\in\PM^U$ we let $\cS_\tau=\cbc{\sigma\in\PM^{V_n}:\forall u\in U:\sigma(u)=\tau(u)}$.
We continue to denote by $\ve_m=\vv\vw$ the final edge of $\GG=\GG(n,m)$ and we let $\GG'=\GG(n,m-1)=\GG-\ve_m$.
Furthermore, given $\delta>0$ let $\vU_\delta\subset V_n$ be the random set from \Lem~\ref{prop_pinning} applied to $\mu_{\GG',\beta,0}$.
As a final preparation we need the following basic random graphs fact.

\begin{lemma}\label{lem_pure_b}
	Assume that $d>1$ and $\beta\in(\betaf(d),\infty)\setminus\fX_d$.
	Then for any fixed $\delta,\ell>0$ the event $\fD_{\delta,\ell}$ that
	\begin{enumerate}[(i)]
		\item the distance between $\vv$ and $\vw$ in $\GG'$ exceeds $2\ell+2$,
		\item the distance between $\vv$ and the set $\vU_\delta$ in $\GG'$ exceeds $\ell+1$, and
		\item the distance between $\vw$ and the set $\vU_\delta$ in $\GG'$ exceeds $\ell+1$.
	\end{enumerate}
	has probability $\pr\brk{\fD_{\delta,\ell}}=1-o(1).$
\end{lemma}
\begin{proof}
	We exploit the fact that given $\GG'$ the pair $(\vv,\vw)$ is within total variation distance $o(1)$ of a uniformly random pair $(\vv',\vw')$ of vertices from $V_n$ that is chosen independently of $\GG',\vU_\delta$.
	Hence, it suffices to prove that (i)--(iii) above hold for the pair $(\vv',\vw')$ \whp\
	A routine first moment calculation shows that for any fixed vertex $u\in V_n$ and any $1\leq l\leq\ell$ the expected number of paths of length $l$ from $\vv$ to $u$ is bounded from above by $(1+o(1))d^l/n=O(1/n)$.
	Therefore, the probability that $\vv',\vw'$ are within distance $\ell$ is $O(1/n)$.
	Similarly, since $|\vU_\delta|=O(1)$ (because the size of $\vU_\delta$ depends on $\delta$ but not on $n$), the event that either $\vv'$ or $\vw'$ is within distance $\ell$ of some vertex $u\in\vU_\delta$ has probability $O(1/n)$.
\end{proof}

Beginning with the investigation of the `pure states' proper, we are now going to show that with probability close to one for `most' $\cS_\tau$ the magnetisation $\scal{\SIGMA(\ve_m)}{\mu_{\GG',\beta,0}(\nix\mid\cS_\tau)}$ of the edge $\ve_m$ approximately equals the product $\scal{\SIGMA(\vv)}{\mu_{\GG',\beta,0}(\nix\mid\cS_\tau)}\scal{\SIGMA(\vw)}{\mu_{\GG',\beta,0}(\nix\mid\cS_\tau)}$ of the vertex magnetisations.
In the proofs of the following lemmas we will occasionally condition on $\GG'$ and $\vU_\delta$.
We emphasise that given $\GG',\vU_\delta$ the only remaining randomness stems from the choice of $\ve_m=\vv\vw$.

\begin{lemma}\label{lem_pure_a}
	Assume that $d>1$ and $\beta\in(\betaf(d),\infty)\setminus\fX_d$.
	For any $\eps>0$ there is $\delta_0=\delta_0(\eps)>0$ such that for any $0<\delta<\delta_0(\eps)$ the following is true.
	Let
	\begin{align*}
		\cT_{\eps}&=\cbc{\tau\in\PM^{\vU_\delta}: \abs{\scal{\SIGMA(\ve_m)}{\mu_{\GG',\beta,0}(\nix\mid\cS_\tau)}-\scal{\SIGMA(\vv)}{\mu_{\GG',\beta,0}(\nix\mid\cS_\tau)}\scal{\SIGMA(\vw)}{\mu_{\GG',\beta,0}(\nix\mid\cS_\tau)}}<\eps}&&\mbox{and}\\
		\fS_\eps&=\cbc{\sum_{\tau\in\cT_\eps}\mu_{\GG',\beta,0}(\cS_\tau)>1-\eps}.
	\end{align*}
	Then $\pr\brk{\fS_\eps}\geq1-\eps-o(1)$.
\end{lemma}
\begin{proof}
	Assuming $n$ is sufficiently large, \Lem~\ref{prop_pinning} shows that with probability at least $1-\delta$ the subcube decomposition $(\cS_\tau)_{\tau\in\PM^{\vU_\delta}}$ satisfies
	\begin{align}\label{eq_lem_pure_1}
		\sum_{\tau\in\PM^{\vU_\delta}}\mu_{\GG',\beta,0}(\cS_{\tau})\vecone\cbc{\mu_{\GG',\beta,0}(\nix\mid\cS_{\tau})\mbox{ is $\delta$-symmetric\,}}>1-\delta.
	\end{align}
	Given $\GG',\vU_\delta$ the vertex pair $(\vv,\vw)$ is within total variation distance $o(1)$ of a uniformly random pair of vertices.
	Therefore, the definition~\eqref{eqdefsym} of $\delta$-symmetry ensures that for any $\tau\in\PM^{\vU_\delta}$ such that $\mu_{\GG',\beta,0}(\nix\mid\cS_{\tau})$ is $\delta$-symmetric we have
	\begin{align}\label{eq_lem_pure_a1}
		\ex\brk{
		\abs{\scal{\SIGMA(\vv\vw)}{\mu_{\GG',\beta,0}(\nix\mid\cS_\tau)}-\scal{\SIGMA(\vv)}{\mu_{\GG',\beta,0}(\nix\mid\cS_\tau)}\scal{\SIGMA(\vw)}{\mu_{\GG',\beta,0}(\nix\mid\cS_\tau)}}\mid\GG',\vU_\delta}&<2\delta+o(1).
	\end{align}
	Thus, the assertion follows from~\eqref{eq_lem_pure_1}, \eqref{eq_lem_pure_a1} and Markov's inequality.
\end{proof}

Next we combine \Lem s~\ref{lem_max}--\ref{lem_pure_a} with \Lem~\ref{lem_kostas} to show that likely for `most' $\tau$ the product $$\scal{\SIGMA(\vv)}{\mu_{\GG',\beta,0}(\nix\mid\cS_\tau)}\scal{\SIGMA(\vw)}{\mu_{\GG',\beta,0}(\nix\mid\cS_\tau)}$$ of the conditional magnetisations is close to the product
\begin{align*}
	\scal{\SIGMA(\vv)}{\mu_{\GG',\beta,0}^{(\ell)}}\scal{\SIGMA(\vw)}{\mu^{(\ell)}_{\GG',\beta,0}}
\end{align*}
of the magnetisations with the all-$1$ boundary condition imposed at distance $\ell$.
Crucially, the latter depends only on the topology of $\GG'$ around $\vv,\vw$ but not on $\tau$.

\begin{lemma}\label{lem_pure_c}
	Assume that $d>1$ and $\beta\in(\betaf(d),\infty)\setminus\fX_d$.
	For any $\eps>0$ there exist $\delta_0=\delta_0(d,\beta,\eps)>0$, $\ell_0=\ell_0(d,\beta,\eps)$ such that for $0<\delta<\delta_0(\eps)$ and $\ell>\ell_0$ the following is true.
	Let
	\begin{align*}
		\cT_\eps'&=\cbc{\tau\in\PM^{\vU_\delta}:\abs{\scal{\SIGMA(\vv)}{\mu_{\GG',\beta,0}(\nix\mid\cS_\tau)}\scal{\SIGMA(\vw)}{\mu_{\GG',\beta,0}(\nix\mid\cS_\tau)}- \scal{\SIGMA(\vv)}{\mu_{\GG',\beta,0}^{(\ell)}}\scal{\SIGMA(\vw)}{\mu^{(\ell)}_{\GG',\beta,0}}}<\eps}&&\mbox{and}\\
		\fS_\eps'&=\cbc{\sum_{\tau\in\cT_\eps'}\mu_{\GG',\beta,0}(\cS_\tau)>1-\eps}.
	\end{align*}
	Then $\pr\brk{\fS_\eps'}\geq1-\eps-o(1).$
\end{lemma}
\begin{proof}
	Given $\eps>0$ choose $\xi=\xi(d,\beta,\eps)>0$ and $\delta_0=\delta_0(d,\beta,\eps,\xi)>0$ small enough and $\ell_0=\ell_0(d,\beta,\eps,\xi)>0$ large enough and assume that $0<\delta<\delta_0$ and $\ell>\ell_0$.
	Let $\vecone_{\vU_\delta}\in\PM^{\vU_\delta}$ be the all-ones vector.
	Then \Lem~\ref{lem_kostas} shows that
	\begin{align}\label{eq_lem_pure_4}
		\scal{\SIGMA(\vv)}{\mu_{\GG',\beta,0}(\nix\mid\cS_\tau)}&\leq\scal{\SIGMA(\vv)}{\mu_{\GG',\beta,0}(\nix\mid\cS_{\vecone_{\vU_\delta}})}&&\mbox{for all }\tau\in\PM^{\vU_\delta}.
	\end{align}
	Moreover, given the event $\fD_{\delta,\ell}$ from \Lem~\ref{lem_pure_b}, we have that the
  distance between $\vv$ and $\vU_\delta$ is at least $\ell+1$, and therefore, \Lem~\ref{lem_kostas} implies that
	\begin{align}\label{eq_lem_pure_5}
		\scal{\SIGMA(\vv)}{\mu_{\GG',\beta,0}(\nix\mid\cS_{\vecone_{\vU_\delta}})}\leq\scal{\SIGMA(\vv)}{\mu^{(\ell)}_{\GG',\vv,\beta,0}}.
	\end{align}
	Combining~\eqref{eq_lem_pure_4}--\eqref{eq_lem_pure_5}, due to the inversion symmetry of the Ising model without an external field we conclude that given $\fD_{\delta,\ell}$ for all $\tau\in\PM^{\vU_\delta}$,
	\begin{align}\label{eq_lem_pure_6}
		\abs{\scal{\SIGMA(\vv)}{\mu_{\GG',\beta,0}(\nix\mid\cS_{\tau})}}&\leq\scal{\SIGMA(\vv)}{\mu^{(\ell)}_{\GG',\vv,\beta,0}}=\fm_{\GG,\vv\to\vw,\beta,0}^{(\ell)},&&\mbox{and analogously}\\
		\abs{\scal{\SIGMA(\vw)}{\mu_{\GG',\beta,0}(\nix\mid\cS_{\tau})}}&\leq\scal{\SIGMA(\vw)}{\mu^{(\ell)}_{\GG',\vw,\beta,0}}=\fm_{\GG,\vw\to\vv,\beta,0}^{(\ell)}.\label{eq_lem_pure_7}
	\end{align}

	Further, consider the event
	\begin{align*}
		\fE_\xi&=\cbc{0\leq\frac{\fm_{\GG,\vv\to\vw,\beta,0}^{(\ell)}\fm_{\GG,\vw\to\vv,\beta,0}^{(\ell)}+\tanh\beta}{1+\fm_{\GG,\vv\to\vw,\beta,0}^{(\ell)}\fm_{\GG,\vw\to\vv,\beta,0}^{(\ell)}\tanh\beta}-\scal{\SIGMA(\vv\vw)}{\mu_{\GG,\beta,0}}<\xi}.
	\end{align*}
	\Lem s~\ref{lem_max}, \ref{cor_max} and~\ref{lem_pure_b} imply that
	\begin{align}\label{eq_lem_pure_8}
		\pr\brk{\fE_\xi}>1-\xi+o(1),
	\end{align}
	provided that $\ell$ is sufficiently large.
	Hence, with $\cT_\xi$ the set from \Lem~\ref{lem_pure_a}, consider
	\begin{align}\label{eq_lem_pure_c_0}
		\cT_\xi^*&=\cbc{\tau\in\cT_\xi:\scal{\SIGMA(\vv\vw)}{\mu_{\GG,\beta,0}(\nix\mid\cS_\tau)}\geq\frac{\fm_{\GG,\vv\to\vw,\beta,0}^{(\ell)}\fm_{\GG,\vw\to\vv,\beta,0}^{(\ell)}+\tanh\beta}{1+\fm_{\GG,\vv\to\vw,\beta,0}^{(\ell)}\fm_{\GG,\vw\to\vv,\beta,0}^{(\ell)}\tanh\beta}-\xi^{1/2}}.
	\end{align}
	If the event $\fD_{\delta,\ell}\cap\fE_\xi\cap\fS_\xi$ occurs, then by the law of total probability we have
	\begin{align}\label{eq_lem_pure_9}
		\sum_{\tau\not\in\cT^*_\xi}\mu_{\GG,\beta,0}(\cS_\tau)&<\xi^{1/4}.
	\end{align}

	As a direct consequence of the definition~\eqref{eqIsingmu}--\eqref{eqIsingZ} of the Ising model we find
	\begin{align}\label{eq_lem_pure_c_1}
		\scal{\SIGMA(\vv\vw)}{\mu_{\GG,\beta,0}(\nix\mid\cS_\tau)}&=\frac{\scal{\SIGMA(\ve_m)\exp(\beta\SIGMA(\ve_m))}{\mu_{\GG',\beta,0}(\nix\mid\cS_\tau)}}{\scal{\exp(\beta\SIGMA(\ve_m))}{\mu_{\GG',\beta,0}(\nix\mid\cS_\tau)}}=
	\frac{\scal{\SIGMA(\ve_m)}{\mu_{\GG',\beta,0}(\nix\mid\cS_\tau)}+\tanh\beta}{1+\scal{\SIGMA(\ve_m)}{\mu_{\GG',\beta,0}(\nix\mid\cS_\tau)}\tanh\beta}.
	\end{align}
	Furthermore, for any $\tau\in\cT_\xi^*\subset\cT_\xi$ we have
	\begin{align}\label{eq_lem_pure_c_2}
\abs{\scal{\SIGMA(\vv\vw)}{\mu_{\GG',\beta,0}(\nix\mid\cS_\tau)}-\scal{\SIGMA(\vv)}{\mu_{\GG',\beta,0}(\nix\mid\cS_\tau)}\scal{\SIGMA(\vw)}{\mu_{\GG',\beta,0}(\nix\mid\cS_\tau)}}<\xi.
	\end{align}
	Combining~\eqref{eq_lem_pure_c_1} and~\eqref{eq_lem_pure_c_2} with the definition~\eqref{eq_lem_pure_c_0} of $\cT_\xi^*$ we see that for $\tau\in\cT_\xi^*$,
	\begin{align}\label{eq_lem_pure_9a}
	\frac{\scal{\SIGMA(\vv)}{\mu_{\GG',\beta,0}(\nix\mid\cS_\tau)}\scal{\SIGMA(\vw)}{\mu_{\GG',\beta,0}(\nix\mid\cS_\tau)}+\tanh\beta}{1+\scal{\SIGMA(\vv)}{\mu_{\GG',\beta,0}(\nix\mid\cS_\tau)}\scal{\SIGMA(\vw)}{\mu_{\GG',\beta,0}(\nix\mid\cS_\tau)}\tanh\beta}&\geq\frac{\fm_{\GG,\vv\to\vw,\beta,0}^{(\ell)}\fm_{\GG,\vw\to\vv,\beta,0}^{(\ell)}+\tanh\beta}{1+\fm_{\GG,\vv\to\vw,\beta,0}^{(\ell)}\fm_{\GG,\vw\to\vv,\beta,0}^{(\ell)}\tanh\beta}-\xi^{1/3},
	\end{align}
	provided that $\xi>0$ is chosen sufficiently small.

	Since the function $z\in[-1,1]\mapsto(z+\tanh\beta)/(1+z\tanh\beta)$ is  differentiable and strictly increasing, \eqref{eq_lem_pure_9a} implies that on the event $\fD_{\delta,\ell}\cap\fE_\xi\cap\fS_\xi$ for $\tau\in\cT_\xi^*$ we have
	\begin{align}
	\scal{\SIGMA(\vv)}{\mu_{\GG',\beta,0}(\nix\mid\cS_\tau)}\scal{\SIGMA(\vw)}{\mu_{\GG',\beta,0}(\nix\mid\cS_\tau)}
		&>\fm_{\GG,\vv\to\vw,\beta,0}^{(\ell)}\fm_{\GG,\vw\to\vv,\beta,0}^{(\ell)}-\xi^{1/4}\nonumber\\&=
		\scal{\SIGMA(\vv)}{\mu_{\GG',\beta,0}^{(\ell)}}\scal{\SIGMA(\vw)}{\mu^{(\ell)}_{\GG',\beta,0}}-\xi^{1/4},\label{eq_lem_pure_10}
	\end{align}
	provided $\xi>0$ is chosen small enough.
	Thus, combining~\eqref{eq_lem_pure_6}--\eqref{eq_lem_pure_7} and~\eqref{eq_lem_pure_10}, we conclude that on the event $\fD_{\delta,\ell}\cap\fE_\xi\cap\fS_\xi$,
	\begin{align}
		\abs{\scal{\SIGMA(\vv)}{\mu_{\GG',\beta,0}(\nix\mid\cS_\tau)}\scal{\SIGMA(\vw)}{\mu_{\GG',\beta,0}(\nix\mid\cS_\tau)} -\scal{\SIGMA(\vv)}{\mu_{\GG',\beta,0}^{(\ell)}}\scal{\SIGMA(\vw)}{\mu^{(\ell)}_{\GG',\beta,0}}}<\eps&&\mbox{for all }\tau\in\cT_\xi^*.\label{eq_lem_pure_11}
	\end{align}

	Finally, \eqref{eq_lem_pure_11} shows that $\cT_\eps'\supseteq\cT_\xi^*$ and~\eqref{eq_lem_pure_9} shows that $\fS_\eps'\supseteq\fD_{\delta,\ell}\cap\fE_\xi\cap\fS_\xi$.
	Therefore, \eqref{eq_lem_pure_8}, \Lem~\ref{lem_pure_b} and \Lem~\ref{lem_pure_a} show that
	\begin{align*}
		\pr\brk{\fS_\eps'}\geq\pr\brk{\fD_{\delta,\ell}\cap\fE_\xi\cap\fS_\xi}>1-3\xi+o(1)>1-\eps+o(1),
	\end{align*}
	as desired.
\end{proof}

As a preparation for the next step we point out that the empirical distribution of the magnetisations with an all-ones boundary condition at distance $\ell$ is close to the distribution $\pi_{d,\beta,0}^{(\ell)}$ obtained after $\ell$ iterations of the BP operator from~\eqref{eqpidbetaBell}.

\begin{fact}\label{lemma_W1_B=0}
	For $\ell\geq0$ the distribution
	\begin{align}\label{eq_lemma_W1_B=0}
		\pi_{\GG',\beta,0}^{(\ell)}&=\frac1n\sum_{i=1}^n\delta_{\scal{\SIGMA(v_i)}{\mu^{(\ell)}_{\GG',v_i,\beta,0}}}
	\end{align}
	satisfies $\ex[W_1(\pi_{\GG',\beta,0}^{(\ell)},\pi_{d,\beta,0}^{(\ell)})]=o(1)$.
\end{fact}
\begin{proof}
	This is an immediate consequence of \Lem~\ref{lem_lwc} and \Cor~\ref{cor_BP_tree}.
\end{proof}

While \Lem~\ref{lem_pure_c} provides information about the products
$\scal{\SIGMA(\vv)}{\mu_{\GG',\beta,0}(\nix\mid\cS_\tau)}\scal{\SIGMA(\vw)}{\mu_{\GG',\beta,0}(\nix\mid\cS_\tau)}$, what we are really after is the single magnetisations $\scal{\SIGMA(\vv)}{\mu_{\GG',\beta,0}(\nix\mid\cS_\tau)}$.
The following lemma combines \Lem~\ref{lem_pure_c} and Fact~\ref{lemma_W1_B=0} to at least extract the absolute values $|\scal{\SIGMA(\vv)}{\mu_{\GG',\beta,0}(\nix\mid\cS_\tau)}|$.

\begin{lemma}\label{lem_pure_d}
	Assume that $d>1$ and $\beta\in(\betaf(d),\infty)\setminus\fX_d$.
	For any $\eps>0$ there exist $\delta_0=\delta_0(d,\beta,\eps)>0$, $\ell_0=\ell_0(d,\beta,\eps)$ such that for $0<\delta<\delta_0(\eps)$ and $\ell>\ell_0$ the following is true.
	Let
	\begin{align*}
		\cT_\eps''&=\cbc{\tau\in\PM^{\vU_\delta}:\scal{\SIGMA(\vv)}{\mu_{\GG',\beta,0}^{(\ell)}}-\eps\leq \abs{\scal{\SIGMA(\vv)}{\mu_{\GG',\beta,0}(\nix\mid\cS_\tau)}}\leq \scal{\SIGMA(\vv)}{\mu_{\GG',\beta,0}^{(\ell)}}}&&\mbox{and}\\
		\fS_\eps''&=\cbc{\sum_{\tau\in\cT_\eps''}\mu_{\GG',\beta,0}(\cS_\tau)>1-\eps}.
	\end{align*}
	Then $\pr\brk{\fS_\eps''}\geq1-\eps-o(1).$
\end{lemma}
\begin{proof}
	Given $\eps>0$ choose $\xi=\xi(d,\beta,\eps)>0$ and $\delta=\delta_0(d,\beta,\eps,\xi)>0$ sufficiently small and $\ell_0=\ell_0(d,\beta,\eps,\xi)>0$ sufficiently large so that in the notation from \Lem~\ref{lem_pure_c} we have
	\begin{align}\label{eq_lem_pure_d_1}
		\pr\brk{\pr\brk{\fS_\xi'\mid\GG',\vU_\delta}>1-\xi^{1/2}}>1-\xi^{1/2}-o(1).
	\end{align}
	Additionally, let $\fW_\xi$ be the event that the distribution $\pi_{\GG',\beta,0}^{(\ell)}$ from~\eqref{eq_lemma_W1_B=0} satisfies $W_1(\pi_{\GG',\beta,0}^{(\ell)},\pi_{d,\beta,0})<\xi$.
	Since~\eqref{eqDMweakLimit} ensures that
	\begin{align*}
		\lim_{\ell\to\infty}W_1(\pi_{d,\beta,0}^{(\ell)},\pi_{d,\beta,0})=0,
	\end{align*}
	Fact~\ref{lemma_W1_B=0} shows that for large enough $\ell$ we have
	\begin{align}\label{eq_lem_pure_d2}
		\pr\brk{\fW_\xi}=1-o(1).
	\end{align}

	Since $d>1$ and $\beta>\betaf(d)$, \Prop~\ref{prop_BP_B=0}~(ii) shows that there exists a number $\eta=\eta(d,\beta)>0$ that depends on $d,\beta$ only such that $\pi_{d,\beta,0}([2\eta,\infty])>2\eta$.
	Let
	\begin{align*}
		\fB_\eta&=\cbc{\scal{\SIGMA(\vw)}{\mu_{\GG',\beta,0}^{(\ell)}}>\eta}.
	\end{align*}
	Because given $\GG',\vU_\delta$ the vertex $\vw$ is within total variation distance $o(1)$ of a uniformly random element of $V_n$, we conclude that on the event $\fW_\xi$,
	\begin{align}\label{eq_lem_pure_d3}
		\pr\brk{\fB_\eta\mid\GG',\vU_\delta}&>\eta,
	\end{align}
	provided $\xi>0$ is chosen small enough.

	Now consider the event
	\begin{align*}
		\fV_\xi=\cbc{\pr[\fS_\xi'\mid\GG',\vU_\delta]>1-\xi^{1/2}}\cap\fW_\xi.
	\end{align*}
	Then \eqref{eq_lem_pure_d_1} and \eqref{eq_lem_pure_d2} imply that
	\begin{align}\label{eq_lem_pure_d4}
		\pr\brk{\fV_\xi}>1-\xi^{1/2}-o(1).
	\end{align}
	Furthermore, on the event $\fV_\xi$ we have
	\begin{align}\label{eq_lem_pure_d5}
		\pr[\fS_\xi'\cap\fB_\eta\mid\GG',\vU_\delta]>1-\xi^{1/2}/\eta-o(1)>1-\xi^{1/3}-o(1)
	\end{align}
	if $\xi>0$ is chosen sufficiently small (which is possible because $\eta=\eta(d,\beta)$ depends on $d,\beta$ only).

	Finally, if $\fS_\xi'\cap\fB_\eta$ occurs, then for any $\tau\in\cT_\xi'$ we have
	\begin{align*}
		\abs{\scal{\SIGMA(\vv)}{\mu_{\GG',\beta,0}(\nix\mid\cS_\tau)}}\geq \scal{\SIGMA(\vv)}{\mu_{\GG',\beta,0}^{(\ell)}}-\frac\xi\eta>\scal{\SIGMA(\vv)}{\mu_{\GG',\beta,0}^{(\ell)}}-\eps,
	\end{align*}
	provided that we chose $\xi=\xi(d,\beta,\eps)>0$ suitably small.
	Moreover, \eqref{eq_lem_pure_6} readily yields the upper bound
	\begin{align*}
		\abs{\scal{\SIGMA(\vv)}{\mu_{\GG',\beta,0}(\nix\mid\cS_\tau)}}\leq \Scal{\SIGMA(\vv)}{\mu_{\GG',\beta,0}^{(\ell)}}.
	\end{align*}
	Hence, if $\fS_\xi'\cap\fB_\eta$ occurs, then $\cT_\eps''\supseteq\cT_\eps'$.
	Therefore, the assertion follows from~\eqref{eq_lem_pure_d4} and~\eqref{eq_lem_pure_d5}.
\end{proof}

The next step is to establish consistency of the signs of the magnetisations.
In other words, while \Lem~\ref{lem_pure_d} determines $\scal{\SIGMA(\vv)}{\mu_{\GG',\beta,0}(\nix\mid\cS_\tau)}$ essentially up to a sign $s_{\vv,\tau}$, we are now going to show that $s_{\vv,\tau}$ actually only depends on $\tau$.
In other words, the `pure states' $\cS_\tau$ split into `positive' ones where $s_{\vv,\tau}=1$ and negative ones where $s_{\vv,\tau}=-1$ with high probability over the choice of $\vv$.
In effect, in the cut metric most pure states are close to either a `positive' product measure with magnetisations at least zero, or a negative one with magnetisations less than or equal to zero.

\begin{lemma}\label{lem_pure_e}
	Assume that $d>1$ and $\beta\in(\betaf(d),\infty)\setminus\fX_d$.
	For any $\eps>0$ there exist $\delta_0=\delta_0(d,\beta,\eps)>0$, $\ell_0=\ell_0(d,\beta,\eps)$ such that for $0<\delta<\delta_0(\eps)$ and $\ell>\ell_0$ the following is true.
	Consider the product measures
	\begin{align*}
		\bar\mu_{\GG',\beta,0}^+(\sigma)&=\prod_{v\in V_n}\mu_{\GG',\beta,0}^{(\ell)}(\{\SIGMA(v)=\sigma(v)\}),&
		\bar\mu_{\GG',\beta,0}^-(\sigma)&=\prod_{v\in V_n}\mu_{\GG',\beta,0}^{(\ell)}(\{\SIGMA(v)=-\sigma(v)\})
	\end{align*}
	and let
	\begin{align*}
		\cT_\eps'''&=\cbc{\tau\in\PM^{\vU_\delta}:\min\cbc{\cutm(\mu_{\GG',\beta,0}(\nix\mid\cS_\tau),\bar\mu^+_{\GG',\beta,0}),\cutm(\mu_{\GG',\beta,0}(\nix\mid\cS_\tau),\bar\mu^-_{\GG',\beta,0})}<\eps}&&\mbox{and}\\
		\fS_\eps'''&=\cbc{\sum_{\tau\in\cT_\eps'''}\mu_{\GG',\beta,0}(\cS_\tau)>1-\eps}.
	\end{align*}
	Then $\pr\brk{\fS_\eps'''}\geq1-\eps-o(1).$
\end{lemma}
\begin{proof}
	Given $\eps>0$ choose $\zeta=\zeta(d,\beta,\eps),\xi=\xi(d,\beta,\eps,\zeta),\delta_0(d,\beta,\eps,\zeta,\xi)>0$ sufficiently small and $\ell_0=\ell_0(d,\beta,\eps,\zeta,\xi)>0$ sufficiently large and assume that $0<\delta<\delta_0$ and $\ell>\ell_0$.
	With $\pi_{\GG',\beta,0}^{(\ell)}$ from~\eqref{eq_lemma_W1_B=0} let
	\begin{align*}
		\fW_\xi&=\cbc{W_1(\pi_{\GG',\beta,0}^{(\ell)},\pi_{d,\beta,0})<\xi}.
	\end{align*}
	Then \Lem s~\ref{lem_pure_a}, \ref{lem_pure_c}, \ref{lemma_W1_B=0} and~\ref{lem_pure_d} demonstrate that the event
	\begin{align*}
		\fV_\xi&=\fW_\xi\cap\cbc{\pr[\fS_\xi'\cap\fS_\xi''\mid\GG',\vU_\delta]>1-3\sqrt\xi}
	\end{align*}
	has probability
	\begin{align}\label{eq_lem_pure_e_1}
		\pr[\fV_\xi]>1-4\sqrt\xi+o(1).
	\end{align}
	Since $\vv$ is within total variation distance $o(1)$ of a uniformly random vertex from $V_n$ we conclude that on $\fV_\xi$,
	\begin{align}\nonumber
		\frac1n&\sum_{v\in V_n}\sum_{\tau\in\PM^{\vU_\delta}}\mu_{\GG',\beta,0}(\cS_\tau)\abs{\abs{\scal{\SIGMA(v)}{\mu_{\GG',\beta,0}(\nix\mid\cS_\tau)}}-\scal{\SIGMA(v)}{\mu_{\GG',\beta,0}^{(\ell)}}}
		\leq\ex\brk{\sum_{\tau\not\in\cT_\xi''}\mu_{\GG',\beta,0}(\cS_\tau)\mid\GG',\vU_\delta}+\xi+o(1)\\
			   &\leq1-\pr\brk{\fS''_\xi\mid\GG',\vU_\delta}+2\xi+o(1)<4\sqrt\xi+o(1).\label{eq_lem_pure_e_2}
	\end{align}
	Similarly, since $(\vv,\vw)$ is within total variation distance $o(1)$ of a uniformly random vertex pair, we see that on $\fV_\xi$,
	\begin{align}\nonumber
		\frac1{n^2}&\sum_{\tau\in\PM^{\vU_\delta}}\sum_{v,w\in V_n}\mu_{\GG',\beta,0}(\cS_\tau)\abs{\scal{\SIGMA(v)}{\mu_{\GG',\beta,0}(\nix\mid\cS_\tau)}\scal{\SIGMA(w)}{\mu_{\GG',\beta,0}(\nix\mid\cS_\tau)}-\scal{\SIGMA(v)}{\mu_{\GG',\beta,0}^{(\ell)}}\scal{\SIGMA(w)}{\mu_{\GG',\beta,0}^{(\ell)}}}\\
				   &\leq1-\pr\brk{\fS_\xi'\mid\GG',\vU_\delta}+2\xi+o(1)<4\sqrt\xi+o(1).
\label{eq_lem_pure_e_3}
	\end{align}
	Combining~\eqref{eq_lem_pure_e_2}--\eqref{eq_lem_pure_e_3}, we see that on $\fV_\xi$,
	\begin{align*}
		\frac1{n^2}\sum_{\tau\in\PM^{\vU_\delta}}
\sum_{v,w\in V_n}\mu_{\GG',\beta,0}(\cS_\tau)
\abs{\scal{\SIGMA(v)\SIGMA(w)}{\mu_{\GG',\beta,0}(\nix\mid\cS_\tau)}-\scal{\SIGMA(v)}{\mu_{\GG',\beta,0}(\nix\mid\cS_\tau)}\scal{\SIGMA(w)}{\mu_{\GG',\beta,0}(\nix\mid\cS_\tau)}}<\xi^{1/3}+o(1),
	\end{align*}
	provided $\xi$ is chosen small enough.
	Hence, letting
	\begin{align*}
		\Theta&=\cbc{\tau\in\PM^{\vU_\delta}:\mu_{\GG',\beta,0}(\nix\mid\cS_\tau)\mbox{ is $\xi^{1/4}$-symmetric}},\\
		\Theta'&=\cbc{\tau\in\PM^{\vU_\delta}:\sum_{v,w\in V_n}\abs{\scal{\SIGMA(v)}{\mu_{\GG',\beta,0}(\nix\mid\cS_\tau)}\scal{\SIGMA(w)}{\mu_{\GG',\beta,0}(\nix\mid\cS_\tau)}-\scal{\SIGMA(v)}{\mu_{\GG',\beta,0}^{(\ell)}}\scal{\SIGMA(w)}{\mu_{\GG',\beta,0}^{(\ell)}}}<\xi^{1/4} n^2},\\
		\Theta''&=\cbc{\tau\in\PM^{\vU_\delta}:\sum_{v\in V_n} \abs{\abs{\scal{\SIGMA(v)}{\mu_{\GG',\beta,0}(\nix\mid\cS_\tau)}}-\scal{\SIGMA(v)}{\mu_{\GG',\beta,0}^{(\ell)}}}<\xi^{1/4} n},
	\end{align*}
	we learn from~\eqref{eq_lem_pure_e_2}--\eqref{eq_lem_pure_e_3} that on the event $\fV_\xi$ for small enough $\xi>0$ and large enough $n$ we have
	\begin{align}\label{eq_lem_pure_e_4}
		\sum_{\tau\in\Theta\cap\Theta'\cap\Theta''}\mu_{\GG',\beta,0}(\cS_\tau)>1-\xi^{1/4}.
	\end{align}

	Now, for $\tau\in\Theta'\cap\Theta''$ consider
	\begin{align*}
		\cV_\tau^+&=\cbc{v\in V_n:\scal{\SIGMA(v)}{\mu_{\GG',\beta,0}(\nix\mid\cS_\tau)}\geq\zeta},&\cV_\tau^-&=\cbc{v\in V_n:\scal{\SIGMA(v)}{\mu_{\GG',\beta,0}(\nix\mid\cS_\tau)}\leq-\zeta}.
	\end{align*}
	Let $\Theta^{\#}=\{\tau\in\Theta'\cap\Theta'':\min\{|\cV_\tau^+|,|\cV_\tau^-|\}\geq\zeta n\}$.
	Thus, $\Theta^{\#}$ contains those $\tau\in\Theta'\cap\Theta''$ where given $\cS_\tau$ there are `many' vertices $v$ with a substantially positive magnetisation, as well as many $v$ with a substantially negative magnetisation.
	We are going to argue that this is not possible, i.e., $\Theta^{\#}=\emptyset$.
	Indeed, for $\tau\in\Theta^{\#}$ we have
	\begin{align}\nonumber
		\sum_{v,w\in V_n}&\abs{\scal{\SIGMA(v)}{\mu_{\GG',\beta,0}(\nix\mid\cS_\tau)}\scal{\SIGMA(w)}{\mu_{\GG',\beta,0}(\nix\mid\cS_\tau)}-\scal{\SIGMA(v)}{\mu_{\GG',\beta,0}^{(\ell)}}\scal{\SIGMA(w)}{\mu_{\GG',\beta,0}^{(\ell)}}}\\
						 &\geq\sum_{v\in\cV_\tau^+,\,w\in\cV_\tau^-}
							\zeta^2-\abs{\abs{\scal{\SIGMA(v)}{\mu_{\GG',\beta,0}(\nix\mid\cS_\tau)}\scal{\SIGMA(w)}{\mu_{\GG',\beta,0}(\nix\mid\cS_\tau)}}-\scal{\SIGMA(v)}{\mu_{\GG',\beta,0}^{(\ell)}}\scal{\SIGMA(w)}{\mu_{\GG',\beta,0}^{(\ell)}}}\nonumber\\
						 &\geq\zeta^4n^2-\xi^{1/5}n^2>\xi^{1/4}n^2,\label{eq_lem_pure_e_5}
	\end{align}
	provided that $\xi=\xi(\zeta)>0$ is chosen small enough.
	Hence, comparing \eqref{eq_lem_pure_e_5} with the definition of $\Theta'\supseteq\Theta^{\#}$, we conclude that $\Theta^{\#}=\emptyset$.

	Next we are going to argue that on $\fV_\xi$ the set $\Theta^0=\cbc{\tau\in\Theta'\cap\Theta'':\max\{|\cV_\tau^+|,|\cV_\tau^-|\}\leq\zeta n}$ of $\tau$ such that given $\cS_\tau$ there are very few substantially positive {\em or} negative magnetisations is empty as well.
	Indeed, by \Prop~\ref{prop_BP_B=0} for $d>1$ and $\beta>\betaf(d)$ there exists $\eta=\eta(d,\beta)>0$ such that $\pi_{d,\beta,0}([\eta,1])>\eta$.
	Because the sequence $(\pi_{d,\beta,0}^{(l)})_l$ converges to $\pi_{d,\beta,0}$ in $W_1$, we see that $\pi_{d,\beta,0}^{(\ell)}([\eta/2,1])>\eta/2$ for large enough $\ell$.
	Hence, on the event $\fW_\xi\supseteq\fV_\xi$ we have
	\begin{align}\label{eq_lem_pure_e_6}
		\sum_{v\in V_n}\vecone\cbc{\scal{\SIGMA(v)}{\mu_{\GG',\beta,0}^{(\ell)}}\geq\eta/3}\geq\eta n/3.
	\end{align}
	But for $\tau\in\Theta^0$ we find
	\begin{align}\label{eq_lem_pure_e_7}
		\sum_{v\in V_n}\abs{\scal{\SIGMA(v)}{\mu_{\GG',\beta,0}(\nix\mid\cS_\tau)}}\leq3\zeta n.
	\end{align}
	Thus, combining~\eqref{eq_lem_pure_e_6}--\eqref{eq_lem_pure_e_7} and assuming that $\zeta,\xi$ are chosen suitably small, we obtain
	\begin{align*}
		\sum_{v\in V_n}
		\abs{\abs{\scal{\SIGMA(v)}{\mu_{\GG',\beta,0}(\nix\mid\cS_\tau)}}-\scal{\SIGMA(v)}{\mu_{\GG',\beta,0}^{(\ell)}}}\geq(\eta/3-3\zeta)n>\xi^{1/4}n,
	\end{align*}
	in contradiction to the fact that $\tau\in\Theta''$.
	Hence, $\Theta^0=\emptyset$.

	Since $\Theta^0\cup\Theta^{\#}=\emptyset$ on $\fV_\xi$, we conclude that for any $\tau\in\Theta\cap\Theta'\cap\Theta''$ there exists $s_\tau\in\PM$ such that
	\begin{align}\label{eq_lem_pure_e_8}
		\sum_{v\in V_n}\abs{\scal{\SIGMA(v)}{\mu_{\GG',\beta,0}(\nix\mid\cS_\tau)}-s_\tau\scal{\SIGMA(v)}{\mu_{\GG',\beta,0}^{(\ell)}}}<2\zeta n.
	\end{align}
	Let
	\begin{align*}
		\Theta^+&=\cbc{\tau\in\Theta\cap\Theta'\cap\Theta'':s_\tau=1},&\Theta^-&=\Theta\cap\Theta'\cap\Theta''\setminus\Theta^+.
	\end{align*}
	For each $\tau\in\Theta^+\subset\Theta$ the conditional distribution $\mu_{\GG',\beta,0}(\nix\mid\cS_\tau)$ is $\xi^{1/8}$-symmetric.
	Therefore, \eqref{eq_lem_pure_e_8} implies together with \Lem s~\ref{lem_extremal_symmetric} and~\ref{lem_extremal_margs} that $\cutm(\mu_{\GG',\beta,0}(\nix\mid\cS_\tau),\bar\mu_{\GG',\beta,0}^+)<\eps$.
	Analogously,  we have $\cutm(\mu_{\GG',\beta,0}(\nix\mid\cS_\tau),\bar\mu_{\GG',\beta,0}^-)<\eps$ for $\tau\in\Theta^-$.
	Therefore, the assertion follows from~\eqref{eq_lem_pure_e_1} and~\eqref{eq_lem_pure_e_4}.
\end{proof}

\begin{proof}[Proof of \Prop~\ref{prop_pure}]
	Let $\eps>0$ and pick $\gamma=\gamma(d,\beta,\eps)$, $\zeta=\zeta(d,\beta,\gamma)>0$ and $\xi=\xi(d,\beta,\zeta)>0$ small enough and $\ell=\ell(d,\xi)>0$ sufficiently large.
	Then by Fact~\ref{lemma_W1_B=0} and \Lem~\ref{lem_pure_e} the events $\fS_\xi'''$ and $\fW_{\xi,\ell}=\{W_1(\pi_{\GG',\beta,0}^{(\ell)},\pi_{d,\beta,0})<\xi\}$ occur with probability
	\begin{align}\label{eq_cor_max_100}
		\pr\brk{\fS_\xi'''\cap\fW_\xi}>1-2\xi+o(1).
	\end{align}

	Suppose that $\fS_\xi'''\cap\fW_\xi$ occurs and let
	\begin{align*}
		\Theta^+&=\cbc{\tau\in\PM^{\vU_\delta}:\cutm(\mu_{\GG',\beta,0}(\nix\mid\cS_\tau),\bar\mu^+_{\GG',\beta,0})<\xi},&
		\Theta^-&=\cbc{\tau\in\PM^{\vU_\delta}:\cutm(\mu_{\GG',\beta,0}(\nix\mid\cS_\tau),\bar\mu^-_{\GG',\beta,0})<\xi}.
	\end{align*}
	In addition, for $\sigma\in\PM^{V_n}$ let $X(\sigma)=\frac1n\sum_{v\in V_n}\sigma(v)$ and define $\lambda_{d,\beta}=\int_0^1z\dd\pi_{d,\beta,0}(z)$.
	Then \Lem~\ref{lem_extremal_margs} and the triangle inequality yield
	\begin{align}\nonumber
		\abs{\scal{X(\SIGMA)}{\mu_{\GG,\beta,0}(\nix\mid\cS_\tau)}-\lambda_{d,\beta}}&=\abs{\frac1n\sum_{v\in V_n}\scal{\SIGMA(v)}{\mu_{\GG,\beta,0}(\nix\mid\cS_\tau)}-\lambda_{d,\beta}}\\
																					 &\leq8\cutm(\mu_{\GG',\beta,0}(\nix\mid\cS_\tau),\bar\mu_{\GG',\beta,0})+W_1(\pi_{\GG',\beta,0}^{(\ell)},\pi_{d,\beta,0})<9\xi.\label{eq_cor_max_101}
	\end{align}
	Furthermore, provided $\xi$ is sufficiently small, \Lem s~\ref{lem_extremal_symmetric} and~\ref{lem_extremal_margs} show that the measure $\mu_{\GG,\beta,0}(\nix\mid\cS_\tau)$ is $\zeta$-symmetric.
	Hence, \eqref{eq_cor_max_101} yields
	\begin{align}\nonumber
		\scal{X(\SIGMA)^2}{\mu_{\GG,\beta,0}(\nix\mid\cS_\tau)}&=
		\frac1{n^2}\sum_{v,w\in V_n}\scal{\SIGMA(v)\SIGMA(w)}{\mu_{\GG,\beta,0}(\nix\mid\cS_\tau)}\\
															   &\leq2\zeta+\frac1{n^2}\sum_{v,w\in V_n}\scal{\SIGMA(v)}{\mu_{\GG,\beta,0}(\nix\mid\cS_\tau)}\scal{\SIGMA(w)}{\mu_{\GG,\beta,0}(\nix\mid\cS_\tau)}<2\zeta+9\xi+\lambda_{d,\beta}^2.\label{eq_cor_max_102}
	\end{align}
	Combining~\eqref{eq_cor_max_101} and~\eqref{eq_cor_max_102} with Chebyshev's inequality, we conclude that
	\begin{align}\label{eq_cor_max_103}
		\mu_{\GG,\beta,0}\bc{\cbc{\abs{X(\SIGMA)-\lambda_{d,\beta}}>\zeta^{1/4}}\mid\cS_\tau}<\zeta^{1/4}&&\mbox{for all }\tau\in\Theta^+\mbox{ and analogously}\\
	\label{eq_cor_max_104}
		\mu_{\GG,\beta,0}\bc{\cbc{\abs{X(\SIGMA)+\lambda_{d,\beta}}>\zeta^{1/4}}\mid\cS_\tau}<\zeta^{1/4}&&\mbox{for all }\tau\in\Theta^-.
	\end{align}

	Since $\lambda_{d,\beta}>0$ by \Prop~\ref{prop_BP_B=0}, \eqref{eq_cor_max_103}--\eqref{eq_cor_max_104} imply that on the event $\fS'''_\xi\cap\fW_\xi$,
	\begin{align}\label{eq_cor_max_105}
		\mu_{\GG',\beta,0}(\PM^{V_n}\setminus(\cS_+\cup\cS_-))<1-\zeta^{1/4}-\xi<1-2\zeta^{1/4},
	\end{align}
	provided that $\zeta,\xi$ are chosen suitably small.
	In combination with the inherent symmetry of the Ising model with $B=0$, \eqref{eq_cor_max_105} shows that on the event $\fS'''_\xi\cap\fW_\xi$,
	\begin{align}\label{eq_cor_max_106}
		\mu_{\GG',\beta,0}(\cS_+)=\mu_{\GG',\beta,0}(\cS_-)>\frac12-2\zeta^{1/4}.
	\end{align}

	Further, let
	\begin{align*}
		\cS^+&=\bigcup_{\tau\in\Theta^+}\cS_\tau,&\cS^-&=\bigcup_{\tau\in\Theta^-}\cS_\tau.
	\end{align*}
	Then \eqref{eq_cor_max_103} and~\eqref{eq_cor_max_104} imply that on the event $\fS'''_\xi\cap\fW_\xi$,
	\begin{align}\label{eq_cor_max_107}
		\mu_{\GG',\beta,0}(\cS_+\triangle\cS^+),\mu_{\GG',\beta,0}(\cS_-\triangle\cS^-)&<8\zeta^{1/4}.
	\end{align}
	Consequently, by the definition of the cut metric we have
	\begin{align}\label{eq_cor_max_108}
		\cutm(\mu_{\GG',\beta,0}(\nix\mid\cS_+),\mu_{\GG',\beta,0}(\nix\mid\cS^+))&<\gamma,
	\end{align}
	provided that $\zeta=\zeta(\gamma)>0$ is chosen small enough.
	In addition, the conditional distributions $\mu_{\GG',\beta,0}(\nix\mid\cS_\tau)$ for $\tau\in\Theta^+$ satisfy
	\begin{align}\label{eq_cor_max_109}
		\cutm(\mu_{\GG',\beta,0}(\nix\mid\cS^+),\bar\mu_{\GG',\beta,0}^+)&<\xi.
	\end{align}
	Combining~\eqref{eq_cor_max_108}--\eqref{eq_cor_max_109} and applying the triangle inequality, we find $\cutm(\mu_{\GG',\beta,0}(\nix\mid\cS_+),\bar\mu_{\GG',\beta,0}^+)<2\gamma$.
	Therefore, \Lem s~\ref{lem_extremal_symmetric} and~\ref{lem_extremal_margs} show that on the event $\fS'''_\xi\cap\fW_\xi$,
	\begin{align}\label{eq_cor_max_110}
		\frac1{n^2}\sum_{v,w\in V_n}\abs{\scal{\SIGMA(v)\SIGMA(v)}{\mu_{\GG,\beta,0}(\nix\mid\cS_+)}-\scal{\SIGMA(v)}{\mu_{\GG,\beta,0}(\nix\mid\cS_+)}\scal{\SIGMA(w)}{\mu_{\GG,\beta,0}(\nix\mid\cS_+)}}&<\eps.
	\end{align}
	Moreover, \eqref{eq_cor_max_109} and \Lem~\ref{lem_extremal_margs} show that on $\fS'''_\xi\cap\fW_\xi$,
	\begin{align}\label{eq_cor_max_111}
		W_1(\vec\pi^+_{n,m,\beta},\pi_{d,\beta,0})<\eps.
	\end{align}
Similarly, since $\cutm(\mu_{\GG',\beta,0}(\nix\mid\cS_+),\bar\mu_{\GG',\beta,0}^+)<2\gamma$,
	\Lem~\ref{lem_extremal_margs} shows that on $\fS'''_\xi\cap\fW_\xi$,
	\begin{align}\label{eq_cor_max_112}
		\frac1n\sum_{v\in V_n}\abs{\scal{\SIGMA(v)}{\mu_{\GG,\beta,0}(\nix\mid\cS_+)}-\scal{\SIGMA(v)}{\mu^{(\ell)}_{\GG,v,\beta,0}}}<\eps.
	\end{align}
	Finally, the proposition follows from \eqref{eq_cor_max_100}, \eqref{eq_cor_max_106}, \eqref{eq_cor_max_110}, \eqref{eq_cor_max_111} and~\eqref{eq_cor_max_112} upon taking $\eps\to0$ slowly enough as $n\to\infty$.
\end{proof}

\subsection{Proof of \Prop~\ref{prop_tensor}}\label{sec_prop_tensor}

The steps of the proof are broadly similar to those that we pursued towards the proof of \Prop~\ref{prop_tensor_B>0}.
It is just that we have to take into account the conditioning on $\cS_+$ in the definition~\eqref{eq_emp_tensor} of $\vec\pi_{n,m,\beta,M}^\oplus$.

\begin{lemma}\label{cor_GGmag_B=0}
	If $d,\beta>0$ and $0\leq M\leq m$, then
		\begin{align*}
			\limsup_{\ell\to\infty}\limsup_{n\to\infty}\sum_{h=1,2}\ex\abs{\scal{\SIGMA(v_1)}{\mu_{\GG_{h,M},\beta,B}(\nix\mid\cS_+)}-\scal{\SIGMA(v_1)}{\mu^{(\ell)}_{\GG_{h,M},v_1,\beta,B}}}&=0.
		\end{align*}
\end{lemma}
\begin{proof}
	Since each graph $\GG_{h,M}$ separately is distributed as the random graph $\GG(n,m-1)$, the assertion follows from \Prop~\ref{prop_pure}~(iii) and the union bound.
\end{proof}

\begin{proof}[Proof of \Prop~\ref{prop_tensor}]
	Letting
	\begin{align*}
		\vec\pi_{n,M,\beta,0}^{\oplus\,(\ell)}&=\frac1n\sum_{i=1}^n\delta_{\bc{\scal{\SIGMA(v_i)}{\mu_{\GG_{1,M},v_i,\beta,0}^{(\ell)}},\scal{\SIGMA(v_i)}{\mu_{\GG_{2,M},v_i,\beta,0}^{(\ell)}}}},
	\end{align*}
	we obtain from \Lem~\ref{cor_GGmag_B=0} that
	\begin{align}\label{eq_prop_tensor_1}
		\limsup_{\ell\to\infty}\limsup_{n\to\infty}W_1\bc{\vec\pi_{n,M,\beta}^{\oplus\,(\ell)},\vec\pi_{n,M,\beta}^{\oplus}}&=0.
	\end{align}
	Furthermore, the local weak convergence from \Lem~\ref{lem_lwc_tensor} implies together with \Lem~\ref{lem_tensor_ell} that for any $\ell\geq0$,
	\begin{align}\label{eq_prop_tensor_2}
		W_1\bc{\vec\pi_{n,M,\beta}^{\oplus\,(\ell)},\pi_{d,\beta,0,t}^{\tensor\,(\ell)}}&=o(1).
	\end{align}
	Finally, the assertion follows from	\eqref{eq_prop_tensor_1}, \eqref{eq_prop_tensor_2} and \Cor~\ref{cor_tensor_ell}.
\end{proof}

\subsection{Proof of \Prop~\ref{cor_tensor}}\label{sec_cor_tensor}

As in \Sec~\ref{sec_cor_tensor_B>0} we let $\fG_{n,M}$ be the $\sigma$-algebra generated by $(\GG_{h,M})_{h=1,2}$ and we set $\vDelta_{n,M}=\vDelta^+_{n,M}+\vDelta^-_{n,M}-2\vDelta^{\pm}_{n,M}$.
The principal difference between the present proof and the proof of \Prop~\ref{cor_tensor_B>0} is that here we need to take the conditioning on $\cS_+$ and the inversion symmetry of the Ising model with $B=0$ into account.

\begin{lemma}\label{lem_int}
	Uniformly for all $0\leq M\leq m$ we have
	$\ex\brk{\vDelta_{n,M}\mid\fG_{n,M}}=\cB_{\beta}^\tensor(\pi^\tensor_{d,\beta,0,M/m})-\cB_{\beta}^\tensor(\pi^\tensor_{d,\beta,0,0})+o(1)$ \whp
\end{lemma}

We pursue a similar strategy as towards the proof of \Lem~\ref{lem_intervals}.
Let $\vv_{M},\vw_{M}$ be the vertices of the final `shared' edge $\ve_M$ and let $\vv_{h,M},\vw_{h,M}$ be the vertices of $\ve_{h,M,m-M}$.

\begin{claim}\label{claim_int_A}
	Uniformly for all $0\leq M\leq m$ \whp\ we have
	\begin{align}\label{eq_lem_int3}
		\log\frac{Z_{\GG_{h,M}^+}(\beta,0)}{Z_{\GG_{h,M}}(\beta,0)}&=\log\cosh\beta+\log\bc{1+\scal{\SIGMA(\vv_{M})}{\mu_{\GG_{h,M}}(\nix\mid\cS_+)}\scal{\SIGMA(\vw_{M})}{\mu_{\GG_{h,M}}(\nix\mid\cS_+)}\tanh\beta}+o(1),\\
\label{eq_lem_int3a}
		\log\frac{Z_{\GG_{h,M}^-}(\beta,0)}{Z_{\GG_{h,M}}(\beta,0)}&=\log\cosh\beta+\log\bc{1+\scal{\SIGMA(\vv_{h,M})}{\mu_{\GG_{h,M}}(\nix\mid\cS_+)}\scal{\SIGMA(\vw_{h,M})}{\mu_{\GG_{h,M}}(\nix\mid\cS_+)}\tanh\beta}+o(1).
	\end{align}
\end{claim}
\begin{proof}
	\Lem~\ref{fact_Zratio} demonstrates that with probability $1-O(1/n)$,
	\begin{align}\label{eq_lem_int1}
		\log\frac{Z_{\GG_{h,M}^+}(\beta,B)}{Z_{\GG_{h,M}}(\beta,B)}&=\log\scal{\exp\bc{\beta\SIGMA(\ve_{M})}}{\mu_{\GG_{h,M},\beta,B}}=\log\cosh\beta+\log\bc{1+\scal{\SIGMA(\vv_{M})\SIGMA(\vw_{M})}{\mu_{\GG_{h,M}}}\tanh\beta}.
	\end{align}
	By construction, each individual graph $\GG_{h,M}$ is distributed as a random graph $\GG(n,m-1)$.
	Therefore, by \Prop~\ref{prop_pure}~(i) and the union bound, \whp\ we have $\mu_{\GG_{h,M},\beta,0}(\PM^{V_n}\setminus(\cS_+\cup\cS_-))=o(1)$.
	Therefore, the inversion symmetry of the Ising model in the case $B=0$ implies that \whp\
	\begin{align*}
		\scal{\SIGMA(\vv_M)\SIGMA(\vw_M)}{\mu_{\GG_{h,M},\beta,0}}&=
		\scal{\SIGMA(\vv_M)\SIGMA(\vw_M)}{\mu_{\GG_{h,M},\beta,0}(\nix\mid\cS_+)}+o(1).
	\end{align*}
	Further, since $\vv_M,\vw_M$ are asymptotically uniform and independent, \Prop~\ref{prop_pure}~(ii) ensures that \whp
	\begin{align}\label{eq_lem_int2}
		\scal{\SIGMA(\vv_{1,M})\SIGMA(\vv_{2,M})}{\mu_{\GG_{h,M},\beta,0}(\nix\mid\cS_+)}=\scal{\SIGMA(\vv_{1,M})}{\mu_{\GG_{h,M},\beta,0}(\nix\mid\cS_+)}\scal{\SIGMA(\vv_{2,M})}{\mu_{\GG_{h,M},\beta,0}(\nix\mid\cS_+)}+o(1).
	\end{align}
	Combining~\eqref{eq_lem_int1} and~\eqref{eq_lem_int2}, we obtain \eqref{eq_lem_int3}.
	A similar argument yields~\eqref{eq_lem_int3a}.
\end{proof}

Let $\vec\Xi^+_t,\vec\Xi^-_t,\vec\Xi_t^\pm$ be the expressions from~\eqref{eq_Xi_1}--\eqref{eq_Xi_3} with $B=0$ and set $\vec\Xi_t=\vec\Xi^+_t+\vec\Xi^-_t-2\vec\Xi^\pm_t$.
Moreover, let $\vec\delta_{n,M}^+$ be the distribution of $\vDelta_{n,M}^+$ given $\fG_{n,M}$ and define $\vec\delta_{n,M}^-$, $\vec\delta_{n,M}^\pm$ analogously w.r.t.\ $\vDelta_{n,M}^-$, $\vDelta_{n,M}^\pm$.
Let $\xi^+_t$, $\xi_t^-$, $\xi_t^\pm$ be the distributions of $\Xi^+_t,\Xi^-_t,\Xi^\pm_t$.

\begin{claim}\label{claim_int_C}
	For all $0\leq M\leq m$ we have
	\begin{align}\label{eq_claim_int_C}
		\ex\brk{W_1(\vec\delta_{n,M}^+,\xi_{M/m}^+)+W_1(\vec\delta_{n,M}^-,\xi_{M/m}^-)+W_1(\vec\delta_{n,M}^\pm,\xi_{M/m}^\pm)}&=o(1).
	\end{align}
\end{claim}
\begin{proof}
	Since the vertex pair $(\vv_M,\vw_M)$ is asymptotically uniform, \Prop~\ref{prop_tensor} shows that \whp\ given $\fG_{n,M}$ the pairs
	\begin{align*}
		\bc{\scal{\SIGMA(\vv_{M})}{\mu_{\GG_{h,M}}}}_{h=1,2},&&\bc{\scal{\SIGMA(\vw_{M})}{\mu_{\GG_{h,M}}}}_{h=1,2}
	\end{align*}
	converge in $W_1$ to $(\vec\mu_{\pi_{d,\beta,B,M/m}^\tensor,1,h}^{(0)})_{h=1,2}$ and $(\vec\mu_{\pi_{d,\beta,B,M/m}^\tensor,2,h}^{(0)})_{h=1,2}$.
	Hence, \eqref{eq_lem_int3} yields $W_1(\vec\delta_{n,M}^+,\xi_{M/m}^+)=o(1)$ \whp\
	The bounds on $W_1(\vec\delta_{n,M}^-,\xi_{M/m}^-)$ and $W_1(\vec\delta_{n,M}^\pm,\xi_{M/m}^\pm)$ follow analogously.
\end{proof}

\begin{proof}[Proof of \Lem~\ref{lem_int}]
	Claims~\ref{claim_int_A} and~\ref{claim_int_C} imply that \whp
	\begin{align}\label{eq_lem_int_50}
		\ex\brk{\vDelta_{n,M}\mid\fG_{n,M}}&=\ex\brk{\vec\Xi_{M/m}}+o(1).
	\end{align}
	Furthermore, retracing the steps that we performed towards Eqs.~\eqref{eq_lem_intervals_51} and~\eqref{eq_lem_intervals_52}, we find
	\begin{align}\label{eq_lem_interval_51}
	\ex\brk{\vec\Xi_{M/m}^+}&=\log^2\cosh\beta+2\ex\brk{\log\bc{1+\vec\mu_{\pi_{d,\beta,t}^\tensor,1,1}^{(0)}\vec\mu_{\pi_{d,\beta,t}^\tensor,2,1}^{(0)}\tanh\beta}}\log\cosh\beta+\cB_\beta^\tensor(\pi_{d,\beta,B,M/m}^\tensor),\\
		\ex\brk{\vec\Xi_{M/m}^-}&=\ex\brk{\vec\Xi_{M/m}^\pm}
								=\log^2\cosh\beta+2\ex\brk{\log\bc{1+\vec\mu_{\pi_{d,\beta,t}^\tensor,1,1}^{(0)}\vec\mu_{\pi_{d,\beta,t}^\tensor,2,1}^{(0)}\tanh\beta}}\log\cosh\beta+\cB_\beta^\tensor(\pi_{d,\beta,B,0}^\tensor).
\label{eq_lem_int_52}
	\end{align}
	Finally, the assertion follows from~\eqref{eq_lem_int_50}--\eqref{eq_lem_int_52}.
\end{proof}

\begin{corollary}\label{cor_int}
	We have $\Var\sum_{0\leq M<m}\ex\brk{\vDelta_{n,M}\mid\fF_{n,M}}=o(m^2)$.
\end{corollary}
\begin{proof}
	This is identical to the proof of \Cor~\ref{cor_intervals}, except that we invoke \Lem~\ref{lem_int} rather than \Lem~\ref{lem_intervals}.
\end{proof}

\begin{proof}[Proof of \Prop~\ref{cor_tensor}]
	Once again the argument is identical to the proof of \Prop~\ref{cor_tensor_B>0}, except that we invoke \Lem~\ref{lem_int} and \Cor~\ref{cor_int} instead of \Lem~\ref{lem_intervals} and \Cor~\ref{cor_intervals}.
\end{proof}

\subsection{Proof of \Prop~\ref{prop_var_pos}}\label{sec_prop_var_pos}

As in the proof of \Prop~\ref{prop_var_pos_B>0}, we start from the observation that the limiting distribution $\pi_{d,\beta,0}$ fails to concentrate on a single point.

\begin{lemma}\label{lem_noatom_B=0}
	Assume that $d>1$, $\beta>\betaf(d)$.
	Then $\pi_{d,\beta,0}$ is not supported on a single point.
\end{lemma}
\begin{proof}
	This is a direct consequence of \Prop~\ref{prop_BP_B=0}~(ii).
\end{proof}

\begin{lemma}\label{lem_integrant_B=0}
	If $d>1$ and $\beta>\betaf(d)$, then $\cB^\tensor_\beta(\pi^\tensor_{d,\beta,B,t})\geq\cB^\tensor_\beta(\pi^\tensor_{d,\beta,B,0})$ for all $0\leq t\leq1$.
\end{lemma}
\begin{proof}
	Let $\eps>0$.
	\Lem s~\ref{fact_telescope} and~\ref{lem_int} imply that for sufficiently large $n$ for all $0\leq M\leq m$ we have
	\begin{align}\label{eq_lem_integrant_B=0_1}
		\cB_{\beta}^\tensor(\pi^\tensor_{d,\beta,B,M/m})-\cB_{\beta}^\tensor(\pi^\tensor_{d,\beta,B,0})\geq-\eps.
	\end{align}
	Finally, the continuity of $t\in[0,1]\mapsto\cB_\beta^\tensor(\pi^\tensor_{d,\beta,B,t})$,
which follows from \Prop~\ref{prop_fix_ex}, implies together with \eqref{eq_lem_integrant_B=0_1} that $\cB^\tensor_\beta(\pi^\tensor_{d,\beta,B,t})-\cB^\tensor_\beta(\pi^\tensor_{d,\beta,B,0})\geq0$ for all $t\in[0,1]$.
\end{proof}

\begin{lemma}\label{lem_boundary_B=0}
	If $d>1$ and $\beta>\betaf(d)$, then
	$\cB^\tensor_\beta(\pi^\tensor_{d,\beta,B,1})>\cB^\tensor_\beta(\pi^\tensor_{d,\beta,B,0})$.
\end{lemma}
\begin{proof}
	Since~\eqref{eq_lem_boundary_B>0_1} holds for $B=0$, the assertion follows from \Lem~\ref{lem_noatom_B=0} and Jensen's inequality.
\end{proof}

\begin{proof}[Proof of \Prop~\ref{prop_var_pos_B>0}]
Since  $t\mapsto\cB^\tensor_\beta(\pi^\tensor_{d,\beta,B,t})$ is continuous
by \Prop~\ref{prop_fix_ex}, the assertion follows from \Lem s~\ref{lem_noatom_B=0}--\ref{lem_boundary_B=0}.
\end{proof}

\section{Proof of \Thm~\ref{thm_an}}\label{sec_thm_an1}

\noindent
{\em Throughout this section we assume that $d\tanh\beta<1$.}

\subsection{Proof of \Prop~\ref{prop_EZ}}\label{sec_prop_EZ}

Let $\omega=\omega(n)=\sqrt n\log n$ and let $\fB=\fB_n=\cbc{\alpha\in\ZZ:|\alpha-n/2|\leq\omega}$.

\begin{lemma}\label{lem_EZa}
	We have
	\begin{align}\label{eq_lem_EZa_EZ}
		\ex\brk{Z_\GG(\beta,0)}&=\sum_{\alpha=0}^n\sum_{M=0}^m\binom n\alpha\exp(\beta(2M-m))\binom{\binom{\alpha}2+\binom{n-\alpha}2}M\binom{\alpha(n-\alpha)}{m-M}\binom{\binom n2}m^{-1},\\
		\ex\brk{\Zbal(\beta,0)}&=\sum_{\alpha\in\fB}\sum_{M=0}^m\binom n\alpha\exp(\beta(2M-m))\binom{\binom{\alpha}2+\binom{n-\alpha}2}M\binom{\alpha(n-\alpha)}{m-M}\binom{\binom n2}m^{-1}.
\label{eq_lem_EZa_EZbal}
	\end{align}
\end{lemma}
\begin{proof}
	We begin with the formula for the mean of $Z_\GG(\beta,0)$.
	By the linearity of expectation,
	\begin{align}\label{eq_lem_EZa_1}
		\ex[Z_\GG(\beta,0)]&=\sum_{\sigma\in\PM^{V_n}}\ex\brk{\exp\bc{\beta\sum_{e\in E(\GG)}\sigma(e)}}.
	\end{align}
	Since the distribution of the random graph is invariant under permutations of the vertex set, the contribution of a configuration $\sigma$ to~\eqref{eq_lem_EZa_1} depends only on the number of vertices $v\in V_n$ with $\sigma(v)=1$.
	Indeed, for a given $0\leq\alpha\leq n$ there are a total of $\binom n\alpha$ configurations $\sigma$ such that $\sigma(v)=1$ for precisely $\alpha$ vertices $v\in V_n$.
	Let $\sigma^{(\alpha)}$ be the configuration with $\sigma(v_1)=\cdots=\sigma(v_\alpha)=1$ and $\sigma(v_{\alpha+1})=\cdots=\sigma(v_n)=-1$.
	Then we can rewrite~\eqref{eq_lem_EZa_1} as
	\begin{align}\label{eq_lem_EZa_2}
		\ex[Z_\GG(\beta,0)]&=\sum_{\alpha=0}^n\binom n\alpha\ex\brk{\exp\bc{\beta\sum_{e\in E(\GG)}\sigma^{(\alpha)}(e)}}.
	\end{align}
	Further, the term $\sum_{e\in E(\GG)}\sigma^{(\alpha)}(e)$ simply equals the number of edges $e=vw$ with $\sigma^{(\alpha)}(v)=\sigma^{(\alpha)}(w)$ minus the number of edges with $\sigma^{(\alpha)}(v)\neq\sigma^{(\alpha)}(w)$.
	Now, given $\alpha$ there are precisely $\binom\alpha 2+\binom{(n-\alpha)}2$ potential edges $e=vw$ with $\sigma^{(\alpha)}(v)=\sigma^{(\alpha)}(w)$, and $a(n-\alpha)$ potential edges with $\sigma^{(\alpha)}(v)\neq\sigma^{(\alpha)}(w)$.
	Hence, the probability that the random graph contains precisely $M$ edges $e=vw$ with $\sigma^{(\alpha)}(v)=\sigma^{(\alpha)}(w)$ equals
	\begin{align}\label{eq_lem_EZa_3}
		\binom{\binom{\alpha}2+\binom{n-\alpha}2}M\binom{\alpha(n-\alpha)}{m-M}\binom{\binom n2}m^{-1}.
	\end{align}
	Combining~\eqref{eq_lem_EZa_2} and~\eqref{eq_lem_EZa_3}, we obtain the expression for \eqref{eq_lem_EZa_EZ}.
	The derivation of~\eqref{eq_lem_EZa_EZbal} is similar, except that we only sum on configurations $\sigma$ with  $|\sum_{v\in V_n}\sigma(v)|\leq2\sqrt n\log n$.
\end{proof}

Let $h(z)=-z\log z-(1-z)\log(1-z)$.

\begin{lemma}\label{lem_EZb}
	We have $ \ex\brk{Z_\GG(\beta,0)}\sim\ex\brk{\Zbal(\beta,0)}.$
\end{lemma}
\begin{proof}
	Applying Stirling's formula, we obtain for $\alpha \in \fB$
	\begin{align}
		\binom{\binom\alpha2+\binom{n-\alpha}2}M\binom{\alpha(n-\alpha)}{m-M}\binom{\binom n2}m^{-1} &\sim
						\bcfr{\alpha (n- \alpha)}{\binom n2}^{m}\binom mM	\bcfr{{\alpha \choose 2} + {n - \alpha \choose 2}}{\alpha(n-\alpha)}^{M}       \exp\bc{-{2M^2 \over n^2} - {2(m-M)^2 \over n^2} + {m^2 \over n^2}}.\label{eqlem_EZb1}
	\end{align}
	Alternatively for $\alpha \notin \fB$ we have the following upper bound.

	\begin{align}
		\binom{\binom\alpha2+\binom{n-\alpha}2}M\binom{\alpha(n-\alpha)}{m-M}\binom{\binom n2}m^{-1} &
		=
		O(1) \cdot
		\bcfr{\alpha (n- \alpha)}{\binom n2}^{m}\binom mM	\bcfr{{\alpha \choose 2} + {n - \alpha \choose 2}}{\alpha(n-\alpha)}^{M}       \exp\bc{-{2M^2 \over n^2} - {2(m-M)^2 \over n^2} + {m^2 \over n^2}}.\label{eqlem_EZb1_1}
	\end{align}

	Since the final exponential term is bounded (i.e, $\Theta(1)$) in both \eqref{eqlem_EZb1} and \eqref{eqlem_EZb1_1}, \Lem~\ref{lem_EZa} and \eqref{eqlem_EZb1}--\eqref{eqlem_EZb1_1} yield

	\begin{align}
		\ex[Z_{\GG}(\beta,0)]&=c_1\cdot \sum_{\alpha \in \fB}\binom n\alpha\bcfr{\alpha (n- \alpha)}{\exp(\beta)\binom n2}^{m}\sum_{M=0}^m\binom mM	\exp(2\beta M))\bcfr{{\alpha \choose 2} + {n - \alpha \choose 2}}{\alpha(n-\alpha)}^{M}\nonumber
		\\
		& +
		O(1)\cdot \sum_{\alpha \notin \fB}\binom n\alpha\bcfr{\alpha (n- \alpha)}{\exp(\beta)\binom n2}^{m}\sum_{M=0}^m\binom mM	\exp(2\beta M))\bcfr{{\alpha \choose 2} + {n - \alpha \choose 2}}{\alpha(n-\alpha)}^{M},\label{eqlem_EZb2}
		\\
		\ex[\Zbal(\beta,0)]&=c_1 \cdot \sum_{\alpha\in\fB}\binom n\alpha\bcfr{\alpha (n- \alpha)}{\exp(\beta)\binom n2}^{m}\sum_{M=0}^m\binom mM	\exp(2\beta M))\bcfr{{\alpha \choose 2} + {n - \alpha \choose 2}}{\alpha(n-\alpha)}^{M},\label{eqlem_EZb3}
	\end{align}
	where $c_1 = c_1(d, \beta) = \Theta(1)$.

	Further, by the binomial theorem
	\begin{align}\nonumber
		\bcfr{\alpha (n- \alpha)}{\exp(\beta)\binom n2}^{m}\sum_{M=0}^m\binom mM	\exp(2\beta M))\bcfr{{\alpha \choose 2} + {n - \alpha \choose 2}}{\alpha(n-\alpha)}^{M}&=
		\bc{\frac{\alpha(n-\alpha)}{\binom n2}\exp(-\beta)+\frac{\binom\alpha 2+\binom{n-\alpha}2}{\binom n2}\exp(\beta)}^m\\
		&=\Theta(1)\bc{2\frac\alpha n\bc{1-\frac\alpha n}\exp(-\beta)+\bc{1-2\frac\alpha n+2\bcfr{\alpha}n^2}\exp(\beta)}^m.\label{eqlem_EZb4}
	\end{align}
	Approximating $\binom n\alpha$ by Stirling's formula and setting $z=\alpha/n$, we obtain from~\eqref{eqlem_EZb2}--\eqref{eqlem_EZb4} the estimates
	\begin{align}
		\ex[Z_{\GG}(\beta,0)]&=\sum_{\alpha \in \fB}\exp(nh(z)+c_2)\bc{2z\bc{1-z}\exp(-\beta)+\bc{1-2z+2z^2}\exp(\beta)}^m\nonumber
		\\
		&+ \sum_{\alpha \notin \fB}\exp(nh(z)+O\bc{\log(n)})\bc{2z\bc{1-z}\exp(-\beta)+\bc{1-2z+2z^2}\exp(\beta)}^m,\label{eqlem_EZb5}
		\\
		\ex[\Zbal(\beta,0)]&=\sum_{\alpha\in\fB}\exp(nh(z)+c_2)\bc{2z\bc{1-z}\exp(-\beta)+\bc{1-2z+2z^2}\exp(\beta)}^m,\label{eqlem_EZb6}
	\end{align}
	where $c_2 = c_2(n, d, \beta) = O(\log n)$.
	We proceed to investigate the exponential parts of the above summands.
	Let
	\begin{align}\label{eqf}
		f(z)&=h(z)+\frac mn \log(2z(1-z)\exp(-\beta)+(1-2z+2z^2)\exp(\beta)).
	\end{align}
	The derivatives of $f(z)$ work out to be
	\begin{align}\label{eqf'}
		f'(z)&=\log(1-z)-\log z+\frac{2m}n\cdot\frac{(2z-1)(\exp(\beta)-\exp(-\beta))}{2z(1-z)\exp(-\beta)+(1-2z+2z^2)\exp(\beta)},\\
		f''(z)&=-\frac1{1-z}-\frac1z+\frac{4m}n\cdot\frac{\exp(\beta)-\exp(-\beta)}{2z(1-z)\exp(-\beta)+(1-2z+2z^2)\exp(\beta)}\nonumber\\
			  &\qquad\qquad\qquad\ 	-\frac{4m}n\cdot\bcfr{(2z-1)(\exp(\beta)-\exp(-\beta))}{2z(1-z)\exp(-\beta)+(1-2z+2z^2)\exp(\beta)}^2.\label{eqf''}
	\end{align}
	Hence, $f'(1/2)=0$ and $f''(z)<0$ for all $z\in(0,1)$ if $d\tanh\beta<1$.
	Consequently, \eqref{eqlem_EZb6} yields
	\begin{align}\label{eqlem_EZb7}
		\ex[\Zbal(\beta,0)]&\geq\exp(n f(1/2)+c_2)=\exp\bc{n\log 2+\frac mn\log\cosh\beta+O(\log n)},
	\end{align}
	while for $\alpha\not\in\fB$ we find
	\begin{align}\label{eqlem_EZb8}
		\exp(nf(z))&\leq\exp\bc{nf(1/2)-n\Omega((z-1/2)^2)+O(\log n)}=o\bc{\ex[\Zbal(\beta,0)]/n};
	\end{align}
	to obtain the last inequality we used the choice of $\omega=\sqrt n\log n$.
	Finally, the assertion follows from~\eqref{eqlem_EZb5}, \eqref{eqlem_EZb6}, \eqref{eqlem_EZb7} and~\eqref{eqlem_EZb8}.
\end{proof}

\begin{lemma}\label{lem_EZc}
	Let
	\begin{align*}
		g(z) &= h(z) + \frac mn\log\bc{ z(1-z) + {\exp(2\beta) \over 2} (z^2 + (1-z)^2) },\\
		\lambda&=\sqrt{2 \over \pi n } \,  \bcfr{2n}{\exp(\beta)(n-1)}^m \exp\bc{-{m^2 \over n^2}  \tanh^2 \beta-{m \over n}\cdot {2\exp(2\beta) \over 1 + \exp(2\beta)} }.
	\end{align*}
	Then
	\begin{align}\label{eqlem_EZc}
		\ex\brk{\Zbal(\beta,0)}\sim\lambda\sum_{\alpha\in\fB}^n\exp(ng(\alpha/n)).
	\end{align}
\end{lemma}
\begin{proof}
	Combining~\eqref{eqlem_EZb1} with \Lem~\ref{lem_EZa}, we obtain
	\begin{align}
		\ex\brk{\Zbal(\beta,0)}&\sim\sum_{\alpha\in\fB}\binom n\alpha\bcfr{\alpha (n- \alpha)}{\exp(\beta)\binom n2}^{m}\sum_{M=0}^m\binom mM	\bcfr{{\alpha \choose 2} + {n - \alpha \choose 2}}{\alpha(n-\alpha)}^{M}       \exp\bc{2\beta M-{2M^2 \over n^2} - {2(m-M)^2 \over n^2} + {m^2 \over n^2}}.\label{eqlem_EZc1}
	\end{align}
	For a given $\alpha$ the function $M\mapsto\binom mM	\bcfr{{\alpha \choose 2} + {n - \alpha \choose 2}}{\alpha(n-\alpha)}^{M}       \exp\bc{2\beta M-{2M^2 \over n^2} - {2(m-M)^2 \over n^2} + {m^2 \over n^2}}$ attains its maximum value at
	\begin{align*}
		M_*(\alpha)=m\bc{1+\frac{\exp(-2\beta)\alpha(n-\alpha)}{\binom\alpha2+\binom{n-\alpha}2}}^{-1}.
	\end{align*}
	In particular, values $M\sim M_*(\alpha)$ dominate the inner sum (on $M$) in~\eqref{eqlem_EZc1}.
	Consequently, since for $\alpha\in\fB$ we obtain $M\sim M_*(\alpha)\sim{m \over 1 + \exp(-2\beta)}$, by the binomial theorem the expression~\eqref{eqlem_EZc1} simplifies to
	\begin{align*}
		\ex\brk{\Zbal(\beta,0)}&\sim
		\bcfr{2n}{\exp(\beta)(n-1)}^m
		\exp\bc{-{m^2 \over n^2}  \tanh^2 \beta}
		\sum_{\alpha\in\fB} {n \choose \alpha}   \bc{{\alpha \over n}  \bc{1- {\alpha \over n}}}^m \bc{1 + \exp(2\beta)  {{\alpha \choose 2} + {n - \alpha \choose 2} \over \alpha(n-\alpha)}}^m .
	\end{align*}
	As a next step we apply Stirling's formula to the binomial coefficient $\binom n\alpha$ for $\alpha\in\fB$,  obtaining
	\begin{align}\nonumber
		\ex\brk{\Zbal(\beta,0)}&\sim
		\sqrt{2 \over \pi n }\bcfr{2n}{\exp(\beta)(n-1)}^m\exp\bc{-{m^2 \over n^2}  \tanh^2 \beta}
		 \\
							   &\qquad  \sum_{\alpha \in \fB} \exp\bc{  nh\bc{\alpha \over n} + m \brk{\log{\alpha \over n} + \log\bc{1 - {\alpha \over n}} +  \log\bc{1 + \exp(2\beta)  {{\alpha^2} + (n - \alpha)^2 \over 2\alpha(n-\alpha)} - {2\exp(2\beta) \over n}}}}\label{eqlem_EZc2}.
	\end{align}
	Expanding the final log-term in~\eqref{eqlem_EZc2} as
	\begin{align*}
		\log\bc{1 + \exp(2\beta) \cdot {{\alpha^2} + (n - \alpha)^2 \over 2\alpha(n-\alpha)} - {2\exp(2\beta) \over n}}&=
		\log\bc{1 + \exp(2\beta) \cdot {{\alpha^2} + (n - \alpha)^2 \over 2\alpha(n-\alpha)}}\\&\qquad-	{2\exp(2\beta) \over n}\bc{1 + \exp(2\beta) \cdot {{\alpha^2} + (n - \alpha)^2 \over 2\alpha(n-\alpha)}}^{-1}+O(1/n^2)
	\end{align*}
	finally completes the derivation of~\eqref{eqlem_EZc}.
\end{proof}

We proceed to evaluate the sum $\sum_{\alpha\in\fB}^n\exp(ng(\alpha/n))$ via the Laplace method.

\begin{lemma}\label{lem_EZd}
	We have
	\begin{align*}
		\sum_{\alpha\in\fB}^n\exp(ng(\alpha/n))&\sim\sqrt{\frac{\pi n}{2(1-d\tanh\beta)}}		\exp\bc{n  g(1/2)}.
	\end{align*}
\end{lemma}
	\begin{proof}
		Comparing $g(z)$ with the function $f(z)$ from~\eqref{eqf}, we discover that $g(z)=f(z)+\frac mn(\beta-\log 2)$.
		Therefore, \eqref{eqf'}--\eqref{eqf''} demonstrate that $g(z)$ is concave and attains its unique maximum at $z=1/2$.
		As a consequence, we can approximate the sum $\sum_{\alpha\in\fB}^n\exp(ng(\alpha/n))$ by an integral over $B=[\min\fB,\max\fB]$, which gives
		\begin{align}\nonumber
			\sum_{\alpha\in\fB}\exp(ng(\alpha/n))&\sim\int_B\exp(ng(\alpha/n))\dd\alpha\sim \exp\bc{n g(1/2)} \int_{B} \exp\bc{{g''(1/2) \over 2n}   \alpha^2}\dd \alpha\\
			& \sim\sqrt{-\frac{n}{g''(1/2)}}\exp\bc{n  g(1/2)} \int_{-\infty}^{\infty} \exp\bc{- \frac{z^2}{2}} \dd z
			=  \sqrt{-\frac{2\pi n}{g''(1/2)}}\exp\bc{n  g(1/2)} .\label{eqlem_EZd1}
		\end{align}
		Since $g''(1/2)=f''(1/2)=4(d\tanh\beta -1)$ by~\eqref{eqf''}, the assertion follows from~\eqref{eqlem_EZd1}.
	\end{proof}

\begin{proof}[Proof of \Prop~\ref{prop_EZ}]
	The fact that $\ex[Z_{\GG}(\beta,0)]\sim\ex[\Zbal(\beta,0)]$ follows from \Lem~\ref{lem_EZb}.
	Moreover, since
	\begin{align*}
		g(1/2)=\log 2+\frac mn\log\frac{1+\exp(2\beta)}4,
	\end{align*}
	the asymptotic formula for $\ex[\Zbal(\beta,0)]$ follows from \Lem s~\ref{lem_EZc} and~\ref{lem_EZd} and a bit of algebra.
\end{proof}

\subsection{Proof of \Prop~\ref{prop_ssc}}\label{sec_prop_ssc}

We prove \Prop~\ref{prop_ssc} by way of an auxiliary random graph model that we call the {\em planted Ising model}.
Let $\fB=\cbc{\sigma\in\PM^{V_n}:|\sum_{i=1}^n\sigma(v_i)|\leq2\sqrt n\log n}$ be the set of all balanced configurations, i.e., configurations that are counted by $\Zbal(\beta)$ as defined in~\eqref{eqZbal}.
For $\sigma\in\fB$ we define a random graph $\GG^*(\sigma)=\GG^*_{n,m}(\sigma,\beta)$ by letting
\begin{align}\label{eqG*sigma}
	\pr\brk{\GG^*(\sigma)=G}&\propto\exp\bc{\beta\sum_{e\in E(G)}\sigma(e)}\pr\brk{\GG=G}.
\end{align}
To be precise, the $\propto$-symbol in~\eqref{eqG*sigma} hides the normalising factor
\begin{align}\nonumber
	Z^*(\sigma,\beta)&=\ex\brk{\exp\bc{\beta\sum_{e\in E(\GG)}\sigma(e)}}\\
	&=\binom{\binom n2}m^{-1}\sum_{m_+=0}^m\binom{\binom{|\sigma^{-1}(1)|}2+\binom{|\sigma^{-1}(-1)|}2}{m_+}\binom{|\sigma^{-1}(1)|\cdot|\sigma^{-1}(-1)|}{m-m_+}\exp(\beta(2m_+-m))
\label{eqZ*sigma}
\end{align}
Additionally, we define a random balanced configuration $\SIGMA^*\in\fB$ by letting
\begin{align}\label{eqsigma*}
	\pr\brk{\SIGMA^*=\sigma}&=\vecone\cbc{\sigma\in\fB}\frac{\ex\brk{\exp\bc{\beta\sum_{e\in E(\GG)}\sigma(e)}}}{\ex\brk{\Zbal(\beta)}}.
\end{align}
Finally, our main object of interest in this section is the random graph $\GG^*(\SIGMA^*)$ obtained by first choosing $\SIGMA^*$ according to~\eqref{eqsigma*} and then choosing a random graph from the resulting planted distribution~\eqref{eqG*sigma}.
The following lemma explains our interest in this object.

\begin{lemma}\label{lem_nishi}
	For any random variable $Y$ we have
	\begin{align*}
		\frac{\ex\brk{\Zbal(\beta)Y(\GG)}}{\ex[\Zbal(\beta)]}&=\ex[Y(\GG^*(\SIGMA^*))].
	\end{align*}
\end{lemma}
\begin{proof}
	By the definitions~\eqref{eqsigma*} and~\eqref{eqG*sigma} of $\SIGMA^*$ and $\GG^*(\SIGMA^*)$ we have
	\begin{align*}
		\ex[Y(\GG^*(\SIGMA^*))]\ex[\Zbal(\beta)]&=\sum_{\sigma\in\fB}\ex\brk{\exp\bc{\beta\sum_{e\in E(\GG)}\sigma(e)}}\ex[Y(\GG^*(\sigma))]\\
												&=\sum_{\sigma\in\fB}\ex\brk{\exp\bc{\beta\sum_{e\in E(\GG)}\sigma(e)}Y(\GG)}=\ex[\Zbal(\beta)Y(\GG)],
	\end{align*}
	as claimed.
\end{proof}

As a next step we are going to investigate the energy contribution of the planted configuration $\SIGMA^*$ in the planted graph $\GG^*(\SIGMA^*)$.

\begin{lemma}\label{lem_planted_mag}
	With probability $1-\exp(-\Omega(\sqrt n))$ we have
	\begin{align*}
		\abs{\frac{m\exp(2\beta)}{2+2\exp(2\beta)}-\sum_{vw\in E(\GG^*(\SIGMA^*))}\vecone\{\SIGMA^*(v)=\SIGMA^*(w)=1\}}&= O\bc{n^{2/3}},\\
		\abs{\frac{m\exp(2\beta)}{2+2\exp(2\beta)}-\sum_{vw\in E(\GG^*(\SIGMA^*))}\vecone\{\SIGMA^*(v)=\SIGMA^*(w)=-1\}}&= O\bc{n^{2/3}}.
	\end{align*}
\end{lemma}
\begin{proof}
	Let $\sigma\in\fB$.
	Using Stirling's formula, we approximate the expression~\eqref{eqZ*sigma} as
	\begin{align}\nonumber
		Z^*(\sigma,\beta)&= \binom{\binom n2}m^{-1}\sum_{m_+ + m_- \leq m} {m \choose m_+, m_-} {\binom{|\sigma^{-1}(1)|}2 \choose m_+} {\binom{|\sigma^{-1}(-1)|}2 \choose m_-} \binom{|\sigma^{-1}(1)|\cdot|\sigma^{-1}(-1)|}{m-m_+ - m_-}\exp(\beta(2(m_+ + m_-) -m))
		\\
		&\nonumber = \frac{\Theta(n^{-1/2})}{\exp(\beta m)}\sum_{m_+ + m_- \leq m}\binom m{m_+,m_-}\bcfr{|\sigma^{-1}(1)|}{n}^{2m_+}\bcfr{|\sigma^{-1}(-1)|}{n}^{2m_-}
		\\
		& \nonumber\qquad\qquad\qquad\qquad\quad\quad\cdot \bcfr{2|\sigma^{-1}(1)|\cdot|\sigma^{-1}(-1)|}{n^2}^{m-m_+-m_-}\exp(2\beta(m_++m_-)) \\
		&=\exp(O(\omega))2^{-m}\exp(-\beta m)\sum_{m_+ + m_- \leq m}\binom m{m_+,m_-}\exp((2\beta-\log2)(m_+ + m_-)).
		\label{eqlem_planted_mag1}
	\end{align}
	Further, let
	\begin{align*}
		f(y_+,y_-)&=-y_+\log y_+-y_-\log y_--(1-y_+-y_-)\log(1-y_+-y_-)+(2\beta-\log2)(y_++y_-).
	\end{align*}
	This function is strictly concave and its derivatives work out to be
	\begin{align*}
		\frac{\partial f}{\partial y_+}&=-\log y_++\log(1-y_+-y_-)+2\beta-\log2,&
		\frac{\partial f}{\partial y_-}&=-\log y_-+\log(1-y_+-y_-)+2\beta-\log2.
	\end{align*}
	Hence, $f$ attains its unique maximum at the point
	\begin{align}\label{eqlem_planted_mag2}
		y_+^*&=y_-^*=\frac{\exp(2\beta)}{2+2\exp(2\beta)}.
	\end{align}
	Letting $\cM_0=\{(m_+,m_-):|m_+-m\exp(2\beta)/(2+2\exp(2\beta))| < n^{2/3},~|m_ - - m\exp(2\beta)/(2+2\exp(2\beta))|<n^{2/3}\}$ and applying Stirling's formula once more, we to obtain
	\begin{align}
		\nonumber &\exp(O(\omega))2^{-m}\exp(-\beta m)\sum_{m_+, m_- \notin \cM_0}\binom m{m_+,m_-}\exp((2\beta-\log2)(m_+ + m_-))
		\\
		=
		\nonumber &\exp(O(\omega))2^{-m}\exp(-\beta m)\sum_{m_+, m_- \notin \cM_0}\exp(mf(m_+/m,m_-/m))
		\\
		&\leq \exp(-\Theta(n^{2/3}))2^{-m}\exp(-\beta m)\sum_{m_+, m_- \notin \cM_0}\exp(m f(y_+^*, y_-^*)).
		\label{eqlem_planted_mag3}
	\end{align}
	From \eqref{eqlem_planted_mag1}--\eqref{eqlem_planted_mag3} we obtain the following expression.
	\begin{align}\nonumber
		Z^*(\sigma,\beta)&=(1+\exp(-\Theta(n^{2/3}))
		\cdot \binom{\binom n2}m^{-1}\sum_{m_+ , m_- \in \cM_0} {m \choose m_+, m_-} {\binom{|\sigma^{-1}(1)|}2 \choose m_+} {\binom{|\sigma^{-1}(-1)|}2 \choose m_-} \binom{|\sigma^{-1}(1)|\cdot|\sigma^{-1}(-1)|}{m-m_+ - m_-}
		\\
		&\qquad\qquad\qquad\qquad\qquad\qquad\qquad\qquad\cdot\exp(2\beta(m_++m_-)-\beta m).
		\label{eqlem_planted_mag4}
	\end{align}
	Finally, the assertion follows from~\eqref{eqG*sigma}, \eqref{eqZ*sigma} and~\eqref{eqlem_planted_mag4}.
\end{proof}

\begin{lemma}\label{lem_planted_cycle}
	For any $\ell\geq3$ we have $\ex[\vC_\ell(\GG^*(\SIGMA^*))]\sim\frac{d^\ell}{2\ell}(1+\tanh^\ell\beta).$
\end{lemma}
\begin{proof}
	Let $u_0,\ldots,u_{\ell-1}$ be any sequence of $\ell$ distinct vertices.
	We need to compute the probability these vertices form a cycle.
	To this end, let
	\begin{align*}
		\vm_{1,1}&=\sum_{vw\in E(\GG^*(\SIGMA^*))\SIGMA^*}\vecone\{\SIGMA^*(v)=\SIGMA^*(w)=1\},&
		\vm_{-1,-1}&=\sum_{vw\in E(\GG^*(\SIGMA^*))\SIGMA^*}\vecone\{\SIGMA^*(v)=\SIGMA^*(w)=-1\},\\
		\vm_{1,-1}&=\vm_{-1,1}=m-\vm_{1,1}-\vm_{-1,-1}.
	\end{align*}
	Given $\vm_{1,1},\ldots,\vm_{-1,-1}$ and $\SIGMA^*$ we consider the $2\times 2$-matrix $\vQ=(\vQ_{st})_{s,t\in\PM}$ with entries
	\begin{align*}
		\vQ_{1,1}&=\frac{\vm_{1,1}}{\binom{|\SIGMA^{*\,-1}(1)|}2},&
		\vQ_{-1,-1}&=\frac{\vm_{-1,-1}}{\binom{|\SIGMA^{*\,-1}(-1)|}2},&
		\vQ_{1,-1}&=\vQ_{-1,1}=\frac{\vm_{1,-1}}{|\SIGMA^{*\,-1}(1)\times\SIGMA^{*\,-1}(-1)|}.
	\end{align*}
	Moreover, let
	\begin{align*}
		\vC(u_0,\ldots,u_{\ell-1})&=\vecone\cbc{u_0,\ldots,u_{\ell-1}\mbox{ form a cycle in }\GG^*(\SIGMA^*)}.
	\end{align*}
	Then
	\begin{align*}
		\ex\brk{\vC(u_0,\ldots,u_{\ell-1})\mid\vQ,\SIGMA^*}&=\prod_{i=1}^\ell\vQ_{\SIGMA^*(u_i),\SIGMA^*(u_{i+1\!\!\!\!\!\mod\ell})}.
	\end{align*}
	Hence, letting $Q=(Q_{s,t})_{s,t\in\PM}$ be the matrix with entries
	\begin{align*}
		Q_{1,1}&=Q_{-1,-1}=\frac{2d\exp(2\beta)}{1+\exp(2\beta)},&
		Q_{1,-1}&=Q_{-1,1}=\frac{2d}{1+\exp(2\beta)}
	\end{align*}
	and applying \Lem~\ref{lem_planted_mag}, we obtain
	\begin{align}\label{eqlem_planted_cycle1}
		\ex\brk{\vC(u_0,\ldots,u_{\ell-1})}&\sim(2n)^{-\ell}\tr(Q^\ell).
	\end{align}
	Since $Q$ has eigenvalues $2d$ and $2d(\exp(2\beta)-1)/(\exp(2\beta)+1)=2d\tanh\beta$, \eqref{eqlem_planted_cycle1} implies that
	\begin{align}\label{eqlem_planted_cycle2}
		\ex\brk{\vC(u_0,\ldots,u_{\ell-1})}&\sim \bcfr dn^\ell(1+\tanh^\ell\beta).
	\end{align}
	Finally, since every $\ell$-cycle induces $2\ell$ sequences $u_0,\ldots,u_\ell$ that form a cycle (due to the two possible orientations of the cycle and the $\ell$ choices of the starting point), the assertion follows from~\eqref{eqlem_planted_cycle2}.
\end{proof}

\begin{corollary}\label{cor_planted_cycle}
	For any $\ell\geq3$ and any $k_3,\ldots,k_\ell\geq0$ we have
	\begin{align*}
    \ex\brk{\prod_{j=3}^\ell\prod_{i=1}^{k_j}\bc{\vC_j(\GG^*(\SIGMA^*))-i}}\sim\prod_{j=3}^\ell\brk{\frac{d^j}{2j}(1+\tanh^j\beta)}^{k_j}.
	\end{align*}
\end{corollary}
\begin{proof}
	The proof is based on the computation of the joint factorial moments of $\vC_j(\GG^*(\SIGMA^*))$.
	Specifically, towards the proof of \Lem~\ref{lem_planted_cycle} we computed the first moment of $\vC_j(\GG^*(\SIGMA^*))$ by calculating the probability that a specific sequence of vertices forms a cycle.
	In order to calculate the joint factorial moments of $(\vC_j(\GG^*(\SIGMA^*)))_{3\leq j\leq\ell}$, we generalise this argument by computing the probability that several sequences of vertices form cycles of given lengths simultaneously.
	Since the probability that $\GG^*(\SIGMA^*)$ contains a pair of bounded-length overlapping cycles is of order $O(1/n)$, we just need to take into account pairwise disjoint vertex sequences, in which case the argument is a routine extension of the proof of \Lem~\ref{lem_planted_cycle}.
\end{proof}

\begin{proof}[Proof of \Prop~\ref{prop_ssc}]
	Let $L\geq3$, let $c_3,\ldots,c_L\geq0$ be integers and let
	\begin{align*}
		Y(G)=\vecone\cbc{\forall 3\leq\ell\leq L:\vC_\ell(G)=c_\ell}.
	\end{align*}
	Then \Cor~\ref{cor_planted_cycle} implies that
	\begin{align*}
    \ex\brk{Y(\GG^*(\SIGMA^*))}&\sim\prod_{\ell=3}^L\pr\brk{\Po\bc{(1+\chi_\ell)\vartheta_\ell}=c_\ell}.
	\end{align*}
	Hence, \Lem~\ref{lem_nishi} shows that
	\begin{align}\label{eqprop_ssc1}
    \ex\brk{\Zbal(\beta)Y(\GG)}&\sim\ex[\Zbal(\beta)]\prod_{\ell=3}^L\pr\brk{\Po\bc{(1+\chi_\ell)\vartheta_\ell}=c_\ell}.
	\end{align}
	Finally, as is well known the number of cycles of lengths $3\leq \ell\leq L$ in $\GG$ are asymptotically jointly Poisson with means $d^\ell/(2\ell)$.
	Therefore, \eqref{eqprop_ssc1} implies that
	\begin{align*}
    \frac{\ex\brk{\Zbal(\beta)Y(\GG)\mid\forall 3\leq\ell\leq L:\vC_\ell(\GG)=c_i}}{\ex[\Zbal(\beta)]}&\sim\prod_{\ell=3}^L\frac{\pr\brk{\Po\bc{(1+\chi_\ell)\vartheta_\ell}=c_\ell}}{\pr\brk{\Po(\vartheta_\ell)=c_\ell}}=\prod_{\ell=3}^L(1+\chi_\ell)^{c_\ell}\exp(-\vartheta_\ell\chi_\ell),
	\end{align*}
	as desired.
\end{proof}

\subsection{Proof of \Prop~\ref{prop_EZ2}}\label{sec_prop_EZ2}

As before we let $\omega=\omega(n)=\sqrt n\log n$ and $\fB=\fB_n=\cbc{\alpha\in\ZZ:|\alpha-n/2|\leq\omega}$.
Moreover, for $0\leq\alpha_1,\alpha_2\leq n$ we let $\fR(\alpha_1,\alpha_2)=\cbc{\rho\in\{0,\ldots,n\}:\rho\leq\min\{\alpha_1,\alpha_2\}}$.
Furthermore, let
\begin{align*}
	\cS&=\{\{\sigma,\tau\}:\sigma,\tau\in\PM^2\},&\cS_1&=\cbc{s\in\cS:|s|=1},&\cS_{-1}&=\cbc{\{(1,1),(-1,-1)\},\{(1,-1),(-1,1)\}},&\cS_0&=\cS\setminus(\cS_1\cup\cS_{-1}).
\end{align*}
and let
\begin{align*}
	\Psi=\cbc{(\psi_s)_{s\in\cS}\in\{0,\ldots,m\}^\cS:\sum_{s\in\cS}\psi_s=m},&
\end{align*}
Additionally, for $\psi=(\psi_s)_{s\in \cS}\in\Psi$ let
\begin{align*}
	\cE(\psi)&=\sum_{s\in\cS_1}\psi_s-\sum_{s\in\cS_{-1}}\psi_s.
\end{align*}
Finally, given $\alpha_1,\alpha_2,\rho$ let
\begin{align*}
	\rho_{++}=\rho_{\{(1,1)\}}&=\rho,&\rho_{+-}=\rho_{\{(1,-1)\}}&=\alpha_1-\rho,&\rho_{-+}=\rho_{\{(-1,1)\}}&=\alpha_2-\rho,&\rho_{--}=\rho_{\{(-1,-1)\}}&=n-\alpha_1-\alpha_2+\rho.
\end{align*}

\begin{lemma}\label{lem_EZ2a}
	We have
	\begin{align*}
		\ex[\Zbal(\beta,0)^2]&=
		\sum_{\substack{\alpha_1, \alpha_2 \in \fB\\\rho\in\fR(\alpha_1,\alpha_2)\\\psi\in\Psi}}  \exp\bc{2\beta\cE(\psi)} {n\choose \alpha_1}  {\alpha_1 \choose \rho}  {n - \alpha_1 \choose \alpha_2 - \rho} \prod_{s\in \cS_1} { {\rho_{s} \choose 2} \choose \psi_{s} }  \prod_{\{s,t\}\in\cS_0\cup\cS_{-1}} {\rho_{s}  \rho_{t} \choose \psi_{\cbc{s,t}}}    {{n \choose 2} \choose m}^{-1}.
	\end{align*}
\end{lemma}
\begin{proof}
	By the linearity of expectation,
	\begin{align}\label{eqlem_EZ2a1}
		\ex[\Zbal(\beta,0)^2]&=\sum_{\sigma,\tau}
		\vecone\cbc{\abs{\sum_{i=1}^n\sigma(v_i)},\abs{\sum_{i=1}^n\tau(v_i)}\leq2\omega}\ex\exp\bc{\beta\sum_{vw\in E(\GG)}\sigma\bc v\sigma\bc w+\tau\bc v\tau\bc w}.
	\end{align}
	As in the proof of \Lem~\ref{lem_EZa}, we use the invariance of $\GG$ under vertex permutations to rearrange the above sum.
	Specifically, given $\alpha_1,\alpha_2\in\fB$ and $\rho\in\fR(\alpha_1,\alpha_2)$ let $\sigma^{(\alpha_1)}$, $\tau^{(\alpha_1,\alpha_2,\rho)}$ be the configurations
	\begin{align*}
		\sigma^{(\alpha_1)}(v_1)=\cdots\sigma^{(\alpha_1)}(v_{\alpha_1})&=1,&
		\sigma^{(\alpha_1)}(v_{\alpha_1+1})=\cdots\sigma^{(\alpha_1)}(v_n)&=-1,\\
		\tau^{(\alpha_1,\alpha_2,\rho)}(v_1)=\cdots=\tau^{(\alpha_1,\alpha_2,\rho)}(v_\rho)&=1,&\tau^{(\alpha_1,\alpha_2,\rho)}(v_{\rho+1})=\cdots=\tau^{(\alpha_1,\alpha_2,\rho)}(v_{\alpha_1})&=-1,\\
		\tau^{(\alpha_1,\alpha_2,\rho)}(v_{\alpha_1+1})=\cdots=\tau^{(\alpha_1,\alpha_2,\rho)}(v_{\alpha_1+\alpha_2-\rho})&=1,&
		\tau^{(\alpha_1,\alpha_2,\rho)}(v_{\alpha_1+\alpha_2-\rho+1})=\cdots=\tau^{(\alpha_1,\alpha_2,\rho)}(v_n)&=-1.
	\end{align*}
	Thus, $\alpha_1$ and $\alpha_2$ account for the total number of $+1$ entries in $\sigma^{(\alpha_1)}$ and $\tau^{(\alpha_1,\alpha_2,\rho)}$, respectively, while $\rho$ accounts for the number of vertices set to $+1$ under both configurations.
	Then we can rewrite~\eqref{eqlem_EZ2a1} as
	\begin{align}\label{eqlem_EZ2a2}
		\ex[\Zbal(\beta,0)^2]&=\sum_{\substack{\alpha_1,\alpha_2\in\fB\\\rho\in\fR(\alpha_1,\alpha_2)}}
		\binom n{\alpha_1}\binom{\alpha_1}\rho\binom{n-\alpha_1}{\alpha_2-\rho}
		\ex\exp\bc{\beta\sum_{vw\in E(\GG)}\sigma^{(\alpha_1)}(v)\sigma^{(\alpha_1)}\bc w+\tau^{(\alpha_1,\alpha_2,\rho)}\bc v\tau^{(\alpha_1,\alpha_2,\rho)}\bc w}.
	\end{align}

	To evaluate the expectation from~\eqref{eqlem_EZ2a2}, consider $\psi\in\Psi$.
	If for any $s=(s_1,s_2),t=(t_1,t_2)\in\PM^2$ the graph $\GG$ contains precisely $\psi_{\{s,t\}}$ edges $e=vw$ such that
	\begin{align*}
		\sigma^{(\alpha_1)}(v)&=s_1,&\tau^{(\alpha_1,\alpha_2,\rho)}(v)&=s_2,&
		\sigma^{(\alpha_1)}(w)&=t_1,&\tau^{(\alpha_1,\alpha_2,\rho)}(w)&=t_2,
	\end{align*}
	then
	\begin{align}\label{eqlem_EZ2a3}
		\sum_{vw\in E(\GG)}\sigma^{(\alpha_1)}(v)\sigma^{(\alpha_1)}\bc w+\tau^{(\alpha_1,\alpha_2,\rho)}\bc v\tau^{(\alpha_1,\alpha_2,\rho)}\bc w&=2\cE(\psi).
	\end{align}
	Furthermore, the number of possible graphs with said number of edges as prescribed by $\psi$ equals
	\begin{align}\label{eqlem_EZ2a4}
		\prod_{s\in\cS_1}\binom{\binom{\rho_2}2}{\psi_s}\prod_{\{s,t\}\in\cS_0\cup\cS_{-1}}\binom{\rho_s\rho_t}{\psi_{\{s,t\}}}.
	\end{align}
	Since the total number of graphs with $m$ edges equals $\binom{\binom n2}m$, the assertion follows from~\eqref{eqlem_EZ2a2}, \eqref{eqlem_EZ2a3} and~\eqref{eqlem_EZ2a4}.
\end{proof}

\begin{lemma}\label{lem_EZ2b}
	We have
	\begin{align*}
		\ex[\Zbal(\beta,0)^2]\leq(1+o(1))&
		\sum_{\alpha_1, \alpha_2 \in \fB,\rho\in\fR(\alpha_1,\alpha_2),\psi\in\Psi}  \exp\bc{2\beta\cE(\psi)}
		\binom{n}{\rho_{++},\ldots,\rho_{--}}
		\binom{m}{\psi}
		\prod_{s\in \cS_1} \bcfr{{\rho_{s} \choose 2}}{\binom n2}^{\psi_{s}}\prod_{\{s,t\}\in\cS_0\cup\cS_{-1}} \bcfr{\rho_{s}  \rho_{t}}{\binom n2}^{\psi_{\{s,t\}}}\\
																											 &\qquad\qquad\qquad\qquad\qquad\qquad\cdot\exp\bc{\frac{m^2}{n^2}-\sum_{s\in\cS_1}\frac{\psi_{\{s\}}(\psi_{\{s\}}-1)}{\rho_s(\rho_s-1)}-\frac12\sum_{\{s,t\}\in\cS_0\cup\cS_{-1}}\frac{\psi_{\{s,t\}}(\psi_{\{s,t\}}-1)}{\rho_s\rho_t}}.
	\end{align*}
\end{lemma}
\begin{proof}
	We rewrite several of the binomial coefficients from the expression from \Lem~\ref{lem_EZ2a} in terms of falling factorials.
	First,
	\begin{align}\nonumber
		\binom{\binom n2}m&=\frac1{m!}\binom n2_m=\frac1{m!}\binom{n}2^m\prod_{j=1}^{m-1}\frac{\binom n2-j}{\binom n2}=\frac1{m!}\binom  n2^m\exp\bc{-\sum_{j=1}^{m-1}\frac j{\binom n2}+o(1)}\\
						  &=\frac1{m!}\binom{n}2^m\exp\bc{-\binom m2\binom n2^{-1}+o(1)}=
						  \frac1{m!}\binom{n}2^m\exp\bc{-\frac{m^2}{n^2}+o(1)}.\label{eqlem_EZ2b1}
	\end{align}
	Similarly,
	\begin{align}\label{eqlem_EZ2b2}
		\binom{\binom{\rho_s}2}{\psi_s}&\leq\frac{1}{\psi_s!}\binom{\rho_s}2^{\psi_s}\exp\bc{-\frac{\psi_s(\psi_s-1)}{\rho_s(\rho_s-1)}+o(1)},\\
		\binom{\rho_s\rho_t}{\psi_{\{s,t\}}}&\leq\frac1{\psi{\{s,t\}}!}(\rho_s\rho_t)^{\psi_{\{s,t\}}}\exp\bc{-\frac{\psi_{\{s,t\}}(\psi_{\{s,t\}}-1)}{2\rho_s\rho_t}}.\label{eqlem_EZ2b3}
	\end{align}
	Combining~\eqref{eqlem_EZ2b1}--\eqref{eqlem_EZ2b3} with \Lem~\ref{lem_EZ2a} yields the desired bound.
\end{proof}

\begin{lemma}\label{lem_EZ2c}
	Given $\alpha_1,\alpha_2,\rho$ let
	\begin{align*}
		\psi^*_s&=\frac m{\zeta^*}\bcfr{\rho_s\exp(\beta)}n^2,\qquad
		\psi^*_{\{s,t\}}=\frac{2m\rho_s\rho_t\exp(-2\beta\vecone\{\{s,t\}\in\cS_{-1}\})}{\zeta^* n^2},\qquad\mbox{where}\\
		\zeta^*&=\sum_{s\in\cS_1}\bcfr{\rho_s\exp(\beta)}n^2+\sum_{s\in\cS_0\cup\cS_{-1}}\frac{2\rho_s\rho_t\exp(-2\beta\vecone\{\{s,t\}\in\cS_{-1}\})}{n^2}.
	\end{align*}
	Moreover, let
	\begin{align*}
		\Psi^*=\Psi^*(\alpha_1,\alpha_2,\rho)=\cbc{\psi\in\Psi:\|\psi-\psi^*\|\leq\sqrt n\log n}.
	\end{align*}
	Then
	\begin{align*}
		\ex[\Zbal(\beta,0)^2]\leq(1+o(1))&
		\sum_{\alpha_1, \alpha_2 \in \fB,\rho\in\fR(\alpha_1,\alpha_2),\psi\in\Psi^*}  \exp\bc{2\beta\cE(\psi)}
		\binom{n}{\rho_{++},\ldots,\rho_{--}}
		\binom{m}{\psi} \prod_{s\in \cS_1} \bcfr{{\rho_{s} \choose 2}}{\binom n2}^{\psi_{s}}\prod_{\{s,t\}\in\cS_0\cup\cS_{-1}} \bcfr{\rho_{s}  \rho_{t}}{\binom n2}^{\psi_{\{s,t\}}}\\
																											 &\qquad\qquad\qquad\qquad\qquad\qquad\cdot\exp\bc{\frac{m^2}{n^2}-\sum_{s\in\cS_1}\frac{\psi_{\{s\}}(\psi_{\{s\}}-1)}{\rho_s(\rho_s-1)}-\frac12\sum_{\{s,t\}\in\cS_{0}\cup\cS_{-1}}\frac{\psi_{\{s,t\}}(\psi_{\{s,t\}}-1)}{\rho_s\rho_t}}.
	\end{align*}
\end{lemma}
\begin{proof}
	For given $\alpha_1,\alpha_2\in\fB$, $\rho\in\fR(\alpha_1,\alpha_2)$ and $\psi\in\Psi$ let
	\begin{align*}
		Y(\psi)&=\exp\bc{2\beta\cE(\psi)}\binom{m}{\psi} \prod_{s\in \cS_1} \bcfr{{\rho_{s} \choose 2}}{\binom n2}^{\psi_{s}}\prod_{\{s,t\}\in\cS_0\cup\cS_{-1}} \bcfr{\rho_{s}  \rho_{t}}{\binom n2}^{\psi_{\{s,t\}}}\\
																											 &\qquad\qquad\qquad\qquad\qquad\qquad\cdot\exp\bc{\frac{m^2}{n^2}-\sum_{s\in\cS_1}\frac{\psi_{\{s\}}(\psi_{\{s\}}-1)}{\rho_s(\rho_s-1)}-\frac12\sum_{\{s,t\}\in\cS_{0}\cup\cS_{-1}}\frac{\psi_{\{s,t\}}(\psi_{\{s,t\}}-1)}{\rho_s\rho_t}}.
	\end{align*}
	be the contribution of $\psi$ to the second moment formula from \Lem~\ref{lem_EZ2b}.
	To estimate $Y(\psi)$ let
	\begin{align*}
		q_s&=\psi_s/m,\quad q_{s,t}=\psi_{\{s,t\}}/m,\quad p_s=\psi_s^*/m,\quad p_{\{s,t\}} =\psi_{\{s,t\}}^*/m.
	\end{align*}
	Furthermore, set
	\begin{align*}
		\cD(\psi)&=\sum_{s\in\cS_1}q_s\log\frac{q_s}{p_s}+\sum_{s,t\in\cS_0\cup\cS_{-1}}q_{\{s,t\}}\log\frac{q_{\{s,t\}}}{p_{\{s,t\}}}.
	\end{align*}
	Applying Stirling's formula, we find
	\begin{align*}
		Y(\psi)&=n^{O(1)}\exp\bc{-m\cD(\psi)+m\log\zeta^*}.
	\end{align*}
	Furthermore, the expression $\cD(\psi)$ is just the Kullback-Leibler divergence of the probability distributions $(q_s)_{s\in\cS}$ and $(p_s)_{s\in\cS}$.
	Since the Kullback-Leibler divergence is strictly convex and takes its minimum of zero at the point $q_s=p_s$ for all $s\in\cS$, we obtain
	\begin{align*}
		\sum_{\psi\in\Psi}Y(\psi)&\sim\sum_{\psi\in\Psi^*(\alpha_1,\alpha_2,\rho)}Y(\psi),
	\end{align*}
	which implies the assertion.
\end{proof}

\begin{lemma}\label{lem_EZ2d}
	Let $\fR^*=\{\rho\in\{0,\ldots,n\}:|\rho-n/4|\leq n^{3/4}\log^4n\}$.
	Then
	\begin{align*}
		\ex[\Zbal(\beta,0)^2]\leq(1+o(1))&
		\sum_{\alpha_1, \alpha_2 \in \fB,\rho\in\fR^*,\psi\in\Psi^*}  \exp\bc{2\beta\cE(\psi)}
		\binom{n}{\rho_{++},\ldots,\rho_{--}}
		\binom{m}{\psi} \prod_{s\in \cS_1} \bcfr{{\rho_{s} \choose 2}}{\binom n2}^{\psi_{s}}\prod_{\{s,t\}\in\cS_0\cup\cS_{-1}} \bcfr{\rho_{s}  \rho_{t}}{\binom n2}^{\psi_{\{s,t\}}}\\
																											 &\qquad\qquad\qquad\qquad\qquad\qquad\cdot\exp\bc{\frac{m^2}{n^2}-\sum_{s\in\cS_1}\frac{\psi_{\{s\}}(\psi_{\{s\}}-1)}{\rho_s(\rho_s-1)}-\frac12\sum_{\{s,t\}\in\cS_{0}\cup\cS_{-1}}\frac{\psi_{\{s,t\}}(\psi_{\{s,t\}}-1)}{\rho_s\rho_t}}.
	\end{align*}
\end{lemma}
\begin{proof}
	We expand the r.h.s.\ of the bound on the second moment from \Lem~\ref{lem_EZ2c} to the leading order.
	To this end, let
	\begin{align*}
		\gamma_{++}&=\gamma_{--}=\rho/n,&\gamma_{+-}&=\gamma_{-+}=\frac12-\gamma_{++}
	\end{align*}
	and define
	\begin{align*}
		\varphi(\gamma_{++})&=-2\gamma_{++}\log\gamma_{++}-2\gamma_{+-}\log\gamma_{+-}+\frac mn\log\bc{2(\gamma_{++}^2+\gamma_{+-}^2)\exp(2\beta)+2(\gamma_{++}^2+\gamma_{+-}^2)\exp(-2\beta)+8\gamma_{++}\gamma_{+-}}\\
							&=-2\gamma_{++}\log\gamma_{++}-2\gamma_{+-}\log\gamma_{+-}+\frac mn\log\brk{\cosh(2\beta)+8\bc{2\gamma_{++}^2-\gamma_{++}}\sinh^2\beta}.
	\end{align*}
	We obtain the following bound on the contribution of the $\rho$-summand for any given $\alpha_1,\alpha_2\in\fB$ and any $\psi\in\Psi^*$:
	\begin{align}\nonumber
		S(\alpha_1,\alpha_2,\rho,\psi)&=\exp\bc{2\beta\cE(\psi)+\frac{m^2}{n^2}-\sum_{s\in\cS_1}\frac{\psi_{\{s\}}(\psi_{\{s\}}-1)}{\rho_s(\rho_s-1)}-\frac12\sum_{\{s,t\}\in\cS_{0}\cup\cS_{-1}}\frac{\psi_{\{s,t\}}(\psi_{\{s,t\}}-1)}{\rho_s\rho_t}}\\
									  &\qquad\qquad\cdot\binom{n}{\rho_{++},\ldots,\rho_{--}} \binom{m}{\psi} \prod_{s\in \cS_1} \bcfr{{\rho_{s} \choose 2}}{\binom n2}^{\psi_{s}}\prod_{\{s,t\}\in\cS_0\cup\cS_{-1}} \bcfr{\rho_{s}  \rho_{t}}{\binom n2}^{\psi_{\{s,t\}}}\nonumber\\
									  &=\binom{n}{\rho_{++},\ldots,\rho_{--}}\exp(m\log\zeta^*+O(\sqrt n\log^2n))=\exp(n\varphi(\gamma_{++})+O(\sqrt n\log^2n)).\label{eqlem_EZ2d1}
	\end{align}

	We are going to show that the function $\varphi$ attains its maximum at $\gamma_{++}\sim1/4$.
	Indeed, the derivatives of $\varphi$ come out as
	\begin{align*}
		\varphi'(\gamma_{++})&=2(\log(1/2-\gamma_{++})-\log\gamma_{++})+\frac mn\cdot\frac{8(4\gamma_{++}-1)\sinh^2\beta}{\cosh(2\beta)+8\bc{2\gamma_{++}^2-\gamma_{++}}\sinh^2\beta},\\
		\varphi''(\gamma_{++})&=-\frac 2{\gamma_{++}(1-2\gamma_{++})}
		+\frac mn\cdot\frac{32\sinh^2\beta}{\cosh(2\beta)+8\bc{2\gamma_{++}^2-\gamma_{++}}\sinh^2\beta}-\frac mn\cdot\bcfr{8(4\gamma_{++}-1)\sinh^2\beta}{\cosh(2\beta)+8\bc{2\gamma_{++}^2-\gamma_{++}}\sinh^2\beta}^2.
	\end{align*}
	Hence, under the assumption $d\tanh^2\beta<1$ the function $\varphi(\gamma_{++})$ is strictly concave and attains its maximum at $\gamma_{++}=1/4$.
	Therefore, the assertion follows from~\eqref{eqlem_EZ2d1}.
\end{proof}

\begin{corollary}\label{cor_EZ2d}
	We have
	\begin{align*}
		\ex&[\Zbal(\beta,0)^2]\leq(1+o(1))2^{-m}n^{-2m}\exp\bc{\frac mn+{m^2 \over n^2}\bc{1-\frac{4\cosh^2(2\beta)}{\cosh^2\beta}}}\Sigma^\tensor,\qquad\mbox{where}\\
		\Sigma^\tensor&=\sum_{\substack{\alpha_1, \alpha_2 \in \fB\\\rho\in\fR^*}}  \binom{n}{\rho_{++},\ldots,\rho_{--}} \bc{\sum_{s\in\cS_1} {\rho_{s} \choose 2}  \exp(2\beta) + (\rho_{++} + \rho_{--})  (\rho_{+-} + \rho_{-+}) + \bc{\rho_{++}  \rho_{--} + \rho_{+-}  \rho_{-+} }  \exp(-2\beta)}^m.
	\end{align*}
\end{corollary}
\begin{proof}
	This follows by combining \Lem s~\ref{lem_EZ2c} and~\ref{lem_EZ2d} with Stirling's formula.
\end{proof}

\begin{lemma}\label{lem_EZ2e}
	We have
	\begin{align*}
		\Sigma^\tensor&\sim
	\frac{4^nn^{2m-3/2}}{\bc{1 - d\tanh \beta} \sqrt{1 - d \tanh^2 \beta}}\exp\bc{-{16 m \over n}    + {m \over n} \cdot \log \bc{\exp\bc{2\beta} + \exp\bc{-2\beta} + 2 \over 8}}.
	\end{align*}
\end{lemma}
\begin{proof}
	Let
	\begin{align*}
		f(\gamma_1,\gamma_2,\delta)&=h(\gamma_1)+\gamma_1 h(\delta/\gamma_1)+(1-\gamma_1)h((\gamma_2-\delta)/(1-\gamma_1))+\frac mn \log g(\gamma_1,\gamma_2,\delta),\qquad\mbox{where}\\
		g(\gamma_1,\gamma_2,\delta)&=
			 4\bc{\cosh(2\beta) -1}  \delta^2 + 2\bc{\cosh(2\beta) - 1}  \bc{1-2\gamma_1-2\gamma_2}  \delta
			\\
			&\qquad + \exp\bc{2\beta}  \bc{\gamma_1^2 + \gamma_2^2 + \gamma_1 \gamma_2 - \gamma_1 - \gamma_2 + 1/2} + \exp\bc{-2\beta }  \gamma_1\gamma_2 -\gamma_1^2 - \gamma_2^2 + \gamma_1 + \gamma_2 - 2\gamma_1\gamma_2.
	\end{align*}
	These functions satisfy
	\begin{align*}
		g(1/2,1/2,1/4)&=  \cosh^2\beta/2,&
		f(1/2,1/2,1/4)&=2\log2+\frac mn\log\frac{\cosh^2\beta}2.
	\end{align*}
	Using Stirling's formula, we obtain
	\begin{align}\label{eqlem_EZ2e1}
		\Sigma^\tensor&\sim
		{4\sqrt{2} \over (\pi n)^{3/2}} n^{2m} \exp\bc{- {m \over n} \cdot {4\exp\bc{2\beta} \over (\exp(\beta) + \exp(-\beta))^2 }}  \sum_{\alpha_1, \alpha_2 \in \fB,\,\rho\in\fR^* }  \exp\bc{n  f\bc{{\alpha_1 \over n}, {\alpha_2 \over n}, {\rho_{++}\over n}}}.
	\end{align}
	We are going to evaluate the sum by way of the Laplace method.
	Thus, we calculate the first and second derivatives of $f$.
	The first derivatives work out to be
		\begin{align*}
			{\partial f \over \partial \gamma_1} &=  \log\bc{1- \gamma_1 - {\gamma_2 + \delta}} - \log \bc{\gamma_1 - {\delta}}\\
												 & + {m \over n}  {-2\delta  \bc{\exp\bc{\beta} - \exp\bc{-\beta} }^2 + 2\gamma_1 \exp\bc{2\beta} + \gamma_2 \exp\bc{2\beta} + \gamma_2 \exp\bc{-2\beta} - 2\gamma_1- 2 \gamma_2 + 1 - \exp\bc{2\beta} \over g(\gamma_1, \gamma_2, \delta)},
			\\
			{\partial f \over \partial \gamma_2 }&= \log\bc{1 - \gamma_1 - \gamma_2 + \delta} - \log\bc{\gamma_2 - \delta},\\
												 & + {m \over n}  {-2\delta  \bc{\exp\bc{\beta} - \exp\bc{-\beta} }^2 + 2\gamma_2 \exp\bc{2\beta} + \gamma_1 \exp\bc{-2\beta}+ \gamma_1 \exp\bc{2\beta} - 2\gamma_2 - 2 \gamma_1 + 1- \exp\bc{2\beta} \over g(\gamma_1, \gamma_2, \delta)},
			\\
			{\partial f \over \partial \delta} &= \log(\gamma_1 - \delta) - \log \delta + \log\bc{\gamma_2 - \delta} - \log\bc{1-\gamma_1 - \gamma_2 + \delta}\\
											   & + {m \over n}  {4\bc{\exp\bc{\beta} - \exp\bc{-\beta}}^2  \delta +  \bc{\exp\bc{\beta} -\exp\bc{-\beta} }^2  \bc{1-2\gamma_1-2\gamma_2} \over g(\gamma_1, \gamma_2, \delta)}.
		\end{align*}
	Hence, the Jacobian satisfies
	\begin{align}\label{eqJacobian}
		Df\big|_{\gamma_1=\gamma_2=1/2,\,\delta=1/4}&=0.
	\end{align}
	Further, for the second derivatives we obtain
		\begin{align*}
			{\partial^2 f \over \partial \gamma_1^2} &= -{1 \over 1 - \gamma_1 - \gamma_2 + \delta} - {1 \over \gamma_1 - \delta} + {\frac mn} \cdot { \bc{2\exp\bc{2\beta} - 2} \cdot g(\gamma_1, \gamma_2, \delta) - \bc{\partial g/\partial\gamma_1}^2 \over g(\gamma_1, \gamma_2, \delta)^2},
			\\
			{\partial^2 f \over \partial \gamma_2^2} &= -{1 \over 1 - \gamma_1 - \gamma_2 + \delta} - {1 \over \gamma_2 - \delta} + {\frac mn} \cdot { \bc{2\exp\bc{2\beta} - 2} \cdot g(\gamma_1, \gamma_2, \delta) - (\partial g/\partial\gamma_2)^2 \over g(\gamma_1, \gamma_2, \delta)^2},
			\\
			{\partial^2 f \over \partial \delta^2} &= - {1 \over \gamma_1 - \delta} - {1 \over \gamma_2 - \delta} - {1 \over \delta} - {1 \over 1 - \gamma_1 - \gamma_2 + \delta} + {\frac mn} \cdot {\bc{4\exp\bc{2\beta} + 4\exp\bc{-2\beta} - 8} \cdot g(\gamma_1, \gamma_2, \delta) - (\partial g/\partial\delta)^2 \over g(\gamma_1, \gamma_2, \delta)^2},
			\\
			{\partial^2 f \over \partial \gamma_1 \partial \gamma_2} &= -{1 \over 1 - \gamma_1 - \gamma_2 + \delta}
			+
			{\frac mn} \cdot {\bc{\exp\bc{-2\beta} + \exp\bc{2\beta} - 2} \cdot g(\gamma_1, \gamma_2, \delta) - (\partial g/\partial\gamma_1)(\partial g/\partial\gamma_2)\over g(\gamma_1, \gamma_2, \delta)^2},
			\\
			{\partial^2 f \over \partial \gamma_1 \partial \delta} &= {1 \over \gamma_1 - \delta} + {1 \over 1 - \gamma_1 - \gamma_2 + \delta}
			+
			{\frac mn} \cdot {\bc{-2\exp\bc{2\beta} - 2\exp\bc{-2\beta} + 4} \cdot g(\gamma_1, \gamma_2, \delta) - (\partial g/\partial\gamma_1)(\partial g/\partial\delta) \over g(\gamma_1, \gamma_2, \delta)^2},
			\\
			{\partial^2 f \over \partial \gamma_2 \partial \delta} &= {1 \over \gamma_2 - \delta} + {1 \over 1 - \gamma_1 - \gamma_2 + \delta}
			+
			{\frac mn} \cdot {\bc{-2\exp\bc{2\beta} - 2\exp\bc{-2\beta} + 4} \cdot g(\gamma_1, \gamma_2, \delta) - (\partial g/\partial\gamma_2)(\partial g/\partial\delta) \over g(\gamma_1, \gamma_2, \delta)^2}.
		\end{align*}
In effect, the Hessian at the point $\gamma_1=\gamma_2=1/2,\delta=1/4$ works out to be
	\begin{align*}
		D^2f\big|_{\gamma_1=\gamma_2=1/2,\delta=1/4} =
		\begin{pmatrix}
			-8 + 4d  \tanh \beta  (\tanh(\beta) + 1) & -4 + 4d  \tanh^2 \beta & 8 - 8d  \tanh^2 \beta
			\\
			-4 + 4d  \tanh^2 \beta & -8 + 4d  \tanh \beta  (\tanh(\beta) + 1) & 8 - 8d  \tanh^2 \beta
			\\
			8 - 8d  \tanh^2 \beta & 8 - 8d  \tanh^2 \beta & -16 + 16d  \tanh^2 \beta
		\end{pmatrix}.
	\end{align*}
Hence, $D^2f\big|_{\gamma_1=\gamma_2=1/2,\delta=1/4}$ is negative definite  and
	\begin{align*}
		\det D^2f\big|_{\gamma_1=\gamma_2=1/2,\delta=1/4}  = 256  \bc{d \tanh^2 \beta - 1}  \bc{d \tanh \beta - 1}^2.
	\end{align*}

Finally, let $B=[\min\fB,\max\fB]$, $R=[\min\fR^*,\max\fR^*]$.
Then due to~\eqref{eqJacobian} the Euler-Maclaurin formula shows that
\begin{align*}
	\abs{\Sigma^\tensor-\int_B\int_B\int_{R} \exp\bc{n  h\bc{{\alpha_1 \over n}, {\alpha_2 \over n}, {\rho_{++}\over n}}} \dd \alpha_1 \dd \alpha_2 \dd \rho}&=o\bc{\exp\bc{n  f(1/2, 1/2, 1/4)}}.
\end{align*}
Therefore, we obtain
		\begin{align}\nonumber
			\sum_{\alpha_1, \alpha_2 \in \fB,\,\rho\in\fR^*}&\exp\bc{n  f\bc{{\alpha_1 \over n}, {\alpha_2 \over n}, {\rho_{++}\over n}}}
			\sim \int_B\int_B\int_{R} \exp\bc{n  h\bc{{\alpha_1 \over n}, {\alpha_2 \over n}, {\rho_{++}\over n}}} \dd \alpha_1 \dd \alpha_2 \dd \rho\nonumber\\
			& \sim n^{3/2} \exp\bc{n  f(1/2, 1/2, 1/4)} \int_{-\infty}^\infty \int_{-\infty}^\infty\int_{-\infty}^\infty
			\exp\bc{{1 \over 2}
				\begin{pmatrix}
					\gamma_1 \\ \gamma_2 \\ \delta
					\end{pmatrix}^{\top} D^2f|_{(1/2, 1/2, 1/4)}  \begin{pmatrix}
					\gamma_1 \\ \gamma_2 \\ \delta
				\end{pmatrix}
			} \dd \gamma_1 \dd \gamma_2 \dd \delta
			\nonumber\\
			& \sim \frac{{(2\pi n)^{3/2}}}{\sqrt{-\det D^2f|_{(1/2,1/2,1/4)}}}\exp\bc{n  f(1/2, 1/2, 1/4)}\label{eqlem_EZ2e2}.
		\end{align}
Combining~\eqref{eqlem_EZ2e1} and~\eqref{eqlem_EZ2e2} completes the proof.
\end{proof}

\begin{proof}[Proof of \Prop~\ref{prop_EZ2}]
	The proposition follows from \Cor~\ref{cor_EZ2d} and \Lem~\ref{lem_EZ2e} and a bit of simplifying.
\end{proof}

\subsection{Proof of \Thm~\ref{thm_an}}\label{sec_thm_an}

\Prop s~\ref{prop_EZ}, \ref{prop_ssc} and~\ref{prop_EZ2} verify the assumptions of \Thm~\ref{thm_janson} for the random variable $\Zbal(\beta,0)$.
Hence, \Thm~\ref{thm_janson} demonstrates that
\begin{align}\label{eqthm_an1}
	\lim_{n\to\infty}\frac{\Zbal(\beta,0)}{\ex[\Zbal(\beta,0)]}&=\cW,&&\mbox{where}&
	\cW&=\prod_{\ell\geq3}(1+\chi_\ell)^{\cY_\ell}\exp(-\vartheta_\ell\chi_\ell).
\end{align}
Here the $(\cY_\ell)_{\ell\geq3}$ are independent $\Po(\vartheta_\ell)$ variables, cf.~\eqref{eqJansonW}, and the convergence of the infinite product holds almost surely and in $L^2$.

In order to derive \Thm~\ref{thm_an} from~\eqref{eqthm_an1} we show that the sum
\begin{align*}
	\log\cW&=\sum_{\ell\geq3}\cY_\ell\log(1+\chi_\ell)-\vartheta_\ell\chi_\ell
\end{align*}
converges in $L^1$.
Indeed, since $\log(1+z)\leq z+z^2$ and $d\tanh\beta<1$, we obtain for certain $C=C(d,\beta)>0$
\begin{align}\nonumber
	\ex|\log\cW|&\leq
	\sum_{\ell\geq3}\ex\abs{\cY_\ell\log(1+\tanh^\ell\beta)-\frac{d^\ell}{2\ell}\tanh^\ell\beta}=
	\sum_{\ell\geq3}\ex\abs{\frac{d^\ell}{2\ell}\tanh^\ell\beta+\cY_\ell\sum_{j=1}^\infty(-1)^j\frac{\tanh^{j\cdot\ell}\beta}j}\\
				&\leq\sum_{\ell\geq3}\ex\abs{\cY_\ell-\frac{d^\ell}{2\ell}}\tanh^\ell\beta+\frac{(d\tanh\beta)^\ell}{2\ell}\sum_{j=1}^\infty\tanh^{j\cdot\ell}\beta
				\leq C+\sum_{\ell\geq3}\Var\bc{\cY_\ell}\tanh^{2\ell}\beta
        \leq C+\sum_{\ell\geq3}(d\tanh\beta)^{\ell}<\infty.\label{eqthm_an2}
\end{align}
Combining~\eqref{eqthm_an1} and~\eqref{eqthm_an2}, we see that for any bounded continuous function $g:\RR\to\RR$,
\begin{align}\nonumber
	\ex[g(\log(\Zbal(\beta,0))-\log\ex[\Zbal(\beta,0)])]&=\ex\brk{(g\circ\log)\bcfr{\Zbal(\beta,0)}{\ex[\Zbal(\beta,0)]}}\\
														&=\ex\brk{(g\circ\log)(\cW)}+o(1)=\ex[g(\log\cW)]+o(1).\label{eqthm_an3}
\end{align}
Thus, $\log\Zbal(\beta,0)-\ex[\log\Zbal(\beta,0)]$ converges to $\log\cW$ in distribution.
Furthermore, \Prop~\ref{prop_EZ} implies that
\begin{align*}
	\ex\abs{\bc{\log Z_{\GG}(\beta,0)-\log\ex[Z_{\GG}(\beta,0)]}-\bc{\log\Zbal(\beta,0)-\log\ex[\Zbal(\beta,0)]}}&=o(1),
\end{align*}
whence we conclude that $\log Z_{\GG}(\beta,0)-\ex[\log Z_{\GG}(\beta,0)]$ converges to $\log\cW$ in distribution.
Finally, we obtain the explicit formula for the limiting distribution stated in \Thm~\ref{thm_an} by substituting in the expression for $\ex[\log Z_{\GG}(\beta,0)]$ from \Prop~\ref{prop_EZ}.


\end{document}